\documentclass[twoside, 14pt]{article}
\usepackage{amsmath,amssymb,amsthm,mathrsfs}
\usepackage[margin=2.4cm]{geometry} 
\usepackage{lipsum}
\usepackage{titlesec,hyperref}
\usepackage{fancyhdr}
\usepackage[numbers,sort&compress]{natbib}
\usepackage{color}
\usepackage[titletoc]{appendix}

\fancypagestyle{plain}
{
	\fancyhf{}
	
}

\titleformat{\subsection}{\it}{\thesubsection.\enspace}{1.5pt}{}
\titleformat{\subsubsection}{\it}{\thesubsubsection.\enspace}{1.5pt}{}

\newtheorem{theo}{Theorem}[section]
\newtheorem{lemm}[theo]{Lemma}

\newtheorem{prop}[theo]{Proposition}
\newtheorem{rema}{Remark}[section]
\numberwithin{equation}{section}

\allowdisplaybreaks

\def\beq{\begin{equation}}
	\def\bal{\begin{aligned}}
		\def\dal{\end{aligned}}
	\def\deq{\end{equation}}
\def\beqq{\begin{equation*}}
	\def\deqq{\end{equation*}}

\def\p{\partial}

\def\al{\alpha}
\def \f{\frac}
\def \var{\varepsilon}
\def\ep{\varepsilon}

\def \i {\int_{\mathbb{R}^3_+}}

\def \ah{\alpha_h}
\def\me{\mathcal{E}}
\def\md{\mathcal{D}}

\def \wu{w^u}
\def \wb{w^b}
\def \bu{\overline{u}}
\def \bb{\overline{b}}
\def \bwu{\overline{w}^u}
\def \bwb{\overline{w}^b}

\begin{document}
	\begin{sloppypar}
		\title{Global uniform regularity for the 3D incompressible
			MHD equations with slip boundary condition
			near a background magnetic field  \hspace{-4mm}}
		\author{$\mbox{Jincheng Gao}^1$ \footnote{Corresponding author. Email: gaojch5@mail.sysu.edu.cn}, \quad
			$\mbox{Lianyun Peng}^2$ \footnote{Email: lianyun.peng@polyu.edu.hk}, \quad
			$\mbox{Jiahong Wu}^3$ \footnote{Email: jwu29@nd.edu}, \quad
			$\mbox{Zheng-an Yao}^1$ \footnote{Email: mcsyao@mail.sysu.edu.cn}\\
			\quad
			$^1\mbox{School}$ of Mathematics, Sun Yat-sen University,\\
			Guangzhou 510275, China\\
			$^2\mbox{Department}$ of Applied Mathematics, 
            The Hong Kong Polytechnic University,\\
			Hong Kong 999077, China\\
			$^3\mbox{Department}$ of Mathematics,
			University of Notre Dame, \\
			Notre Dame 46556, USA\\
		}
		
		\date{}
		
		\maketitle

		\begin{abstract}
			{This paper resolves the global regularity problem for the three-dimensional incompressible magnetohydrodynamics (MHD) equations in the upper half-space with slip boundary conditions, in the presence of a background magnetic field. Motivated by geophysical applications, we consider an anisotropic MHD system with weak dissipation in the $x_2$ and $x_3$ directions and small vertical magnetic diffusion.
				By exploiting the stabilizing effect induced by the background magnetic field and constructing a hierarchy of four energy functionals, we establish global-in-time uniform bounds that are independent of the viscosity in the $x_2$ and $x_3$ directions and the vertical resistivity. A key innovation in our analysis is the development of a two-tier energy method, which couples the boundedness of conormal derivatives with the decay of tangential derivatives. These global conormal regularity estimates, together with sharp decay rates, enable us to rigorously justify the vanishing dissipation limit and derive explicit long-time convergence rates to the MHD system with vanishing dissipation in the $x_2$ and $x_3$ directions and no vertical magnetic diffusion. In the absence of a magnetic field, the global-in-time vanishing viscosity limit for the 3D incompressible Navier-Stokes equations with anisotropic dissipation remains a challenging open problem. This work reveals the mechanism by which the magnetic field enhances dissipation and stabilizes the fluid dynamics in the vanishing viscosity limit.
			}
			
		\end{abstract}

		\tableofcontents

		\section{Introduction}
		This paper deals with the global-in-time vanishing viscosity limit
		problem for the 3D anisotropic MHD equations near a background magnetic field.
		To shed light on the main challenges of this problem, we begin by briefly reviewing several important facts about the behavior of solutions to the Euler equations and the vanishing viscosity limit of the Navier-Stokes equations. As the viscosity coefficient $\varepsilon$
		tends to zero, the incompressible Navier-Stokes equations
		\beqq
		\partial_t u + u \cdot \nabla u - \varepsilon \Delta u + \nabla p = 0, \quad
		\nabla \cdot u = 0
		\deqq
		formally converge to the incompressible Euler equations
		\beqq
		\partial_t u + u \cdot \nabla u + \nabla p = 0, \quad
		\nabla \cdot u = 0.
		\deqq
		However, the solutions to the incompressible Euler equations can grow rapidly in time (cf. \cite{{Kiselev-Sverak2014}, {Zlatos2015}, {Choi-Jeong-2021}}), and classical solutions to the 3D Euler equations may develop finite-time singularities \cite{{Chen-Hou-2021}, {Elgindi-2021}}. A more extensive discussion can be found in the review by Drivas and Elgindi \cite{Drivas-Elgindi-2022}. In particular, perturbations governed by the Euler equations near the trivial solution are generally unstable. Consequently, the vanishing viscosity limit of the Navier-Stokes equations can typically be justified only on short time intervals.
		
		In geophysical contexts such as ocean dynamics, the dissipation is highly anisotropic. The vertical viscosity coefficient ranges from $1$ to $10^3 {\rm cm}^2/{\rm sec}$ while the horizontal viscosity ranges from $10^5$ to $10^8 {\rm cm}^2/{\rm sec}$ (cf. [Chapter 4, \cite{1987Geophysical}]). Motivated by this, we consider the anisotropic incompressible Navier-Stokes equations:
		\beq\label{01}
		\partial_t u + u \cdot \nabla u - \partial_{11} u -  \partial_{22} u - \varepsilon \partial_{33} u + \nabla p = 0, \quad
		\nabla \cdot u = 0,
		\deq
		and investigate their behavior as the vertical viscosity $\ep$ tends to zero. In this regime, the system becomes
		\beq\label{02}
		\partial_t u + u \cdot \nabla u - \partial_{11} u -  \partial_{22} u + \nabla p = 0, \quad
		\nabla \cdot u = 0.
		\deq
		Due to their physical relevance and special mathematical structure, these anisotropic models have received considerable attention. A variety of global-in-time small data well-posedness results and large-time decay estimates have been established (cf.\cite{Chemin2006, Chemin2007, Iftimie2002, Paicu2005, Liu-Zhang-2020, Paicu2005-2, Zhang-2009}).
		When $\var$ tends to zero, it has been rigorously justified that equation \eqref{01} converges to the equation \eqref{02} on short time intervals (see, absence of a physical boundary \cite{ani-CDGG2000}, Navier friction condition \cite{ani-Iftimie-Planas2016} and no-slip boundary condition \cite{ani-Liu-Wang2013,ani-Tao2018}).
		Recently, it is verified that the solution of equation \eqref{01} will converge
		globally in time  to the  equation \eqref{02} under
		the slip boundary condition as $\var$ tends to zero(cf.\cite{Gao2024}).

        The anisotropic diffusion is a common physical
		phenomenon and describes processes where the diffusion is directionally dependent.
		Anisotropic diffusive processes occurs in Darcay's flow for porous media,
		large scale turbulence where turbulence scales are anisotropic in size.
		Many mathematically rigorous studies have
		been devoted to understanding such anisotropic flows.
		There is a very large literature on the primitive and the Boussinesq equations with anisotropic dissipation. Various partial dissipation cases on the primitive equations have been
		examined by Cao, Li and Titi(cf.\cite{{Cao-Li-Titi-2016},{Cao-Li-Titi-2017},{Cao-Li-Titi-2020}}).
		Their main focus has been on the global existence and regularity problem.
        Here, we consider the anisotropic incompressible Navier-Stokes equation
        with small dissipative structure in the $x_2$ and $x_3$ directions.
        Thus, as the coefficient $\ep$ tends to zero, the anisotropic
        incompressible Navier-Stokes equation
        \beq\label{01-1}
		\p_t u + u \cdot \nabla u-\p_{11}u-\ep \p_{22}u-\ep \p_{33}u+ \nabla p =0,\quad
		\nabla \cdot u=0
		\deq
        will converge to the anisotropic Navier-Stokes equation
		\beq\label{03}
		\p_t u + u \cdot \nabla u-\p_{11}u+ \nabla p =0,\quad
		\nabla \cdot u= 0.
		\deq
		Since global well-posedness of equation \eqref{03} are not well-understood,
		thus it is difficult to verify whether  the solution of
        equation \eqref{01-1} will converge
		globally in time  to the solution of equation \eqref{03} or not.

		It is well-known that the stabilizing phenomenon observed in physical experiments involving electrically conducting fluids.
		The experiments exhibit a remarkable phenomenon: a background magnetic field actually stabilizes and damps turbulent MHD fluids(cf.\cite{{Alexakis-2011},{A-1942},{Alemany-Moreau-Sulem-Frisch-1979},
			{D1995},{D1997},{D2001},{Gallet-Berhanu-Mordant-2009},{Gallet-Doering-2015}}).
		Thus, we hope to establish this phenomenon as a mathematically rigorous fact on
		the following anisotropic MHD equations
		\begin{equation}\label{eq3}
			\left\{\begin{array}{*{4}{ll}}
				\p_t u^\var + u^\var \cdot \nabla u^\ep
				-\p_1^2 u^\ep -\ep \p_2^2 u^\ep -\ep \p_3^2 u^\ep+ \nabla p^\ep =
				B^\ep \cdot \nabla B^\ep &{\rm in} ~~ \mathbb{R}_+^3,\\
				\partial_t B^\ep + u^\ep \cdot \nabla B^\ep
				-\Delta_h B^\ep -\var  \p_3^2 B^\ep
				= B^\ep\cdot \nabla u^\ep &{\rm in} ~~ \mathbb{R}_+^3,\\
				\nabla \cdot u^\ep = 0,\quad
				\nabla \cdot B^\ep= 0 \quad & {\rm in} ~~ \mathbb{R}_+^3,
			\end{array}\right.
		\end{equation}
		with the boundary condition
		\beq
		(u_3^\ep, B_3^\ep, \p_3 u_h^\ep, \p_3 B_h^\ep)|_{x_3=0}=0,
		\deq
		and the initial data
		\beq
		(u^\ep, B^\ep)|_{t=0}=(u_0, B_0).
		\deq
		Here $u^\ep$, $p^\ep$ and $B^\ep$ represent the velocity,
		total pressure and magnetic field, respectively.
		To understand the stabilizing mechanism of a background magnetic field
		$u_* \equiv 0$ and $B_* \equiv e_2:= (0,1,0)$,
		which is obviously a steady-state of \eqref{eq3}, we study the dynamics
		of the perturbation $(u^\var, b^\var)$ with $b^\var=B^\var- B_*$.
		Clearly, $(u^\var, b^\var)$ satisfies the equations
		\begin{equation}\label{eqr}
			\left\{\begin{array}{*{4}{ll}}
				\partial_t u^\ep + u^\ep \cdot \nabla u^\ep
				-\p_1^2 u^\ep -\ep \p_2^2 u^\ep -\ep \p_3^2 u^\ep
				+ \nabla p^\ep = b^\ep\cdot \nabla b^\ep + \p_2 b^\ep
				\quad & {\rm in} ~~\mathbb{R}_+^3,\\
				\p_t b^\ep + u^\ep \cdot \nabla b^\ep
				-\Delta_h b^\ep -\var  \p_3^2 b^\ep
				=b^\ep\cdot \nabla u^\ep + \p_2 u^\ep
				\quad & {\rm in} ~~ \mathbb{R}_+^3,\\
				\nabla \cdot u^\ep = 0,\quad
				\nabla \cdot b^\ep = 0 \quad & {\rm in} ~~ \mathbb{R}_+^3,\\
				u_3^\ep = 0,\quad \partial_3 u_h^\ep=0, \quad b_3^\ep = 0,\quad \partial_3 b_h^\ep=0\quad & {\rm on} ~~ \mathbb{R}^2 \times \{ x_3=0 \},
			\end{array}\right.
		\end{equation}	
		with the initial data
		\beq\label{eqr-i}
		(u^\ep, b^\ep)|_{t=0}=(u_0, b_0):=(u_0, B_0-e_2).
		\deq
		As the parameter $\ep$ tends to zero, the equation \eqref{eqr}
		will transform as
		\begin{equation}\label{eqr0}
			\left\{\begin{array}{*{4}{ll}}
				\partial_t u^0 + u^0 \cdot \nabla u^0
				-\p_1^2 u^0+ \nabla p^0= b^0 \cdot \nabla b^0 + \p_2 b^0
				\quad & {\rm in} ~~\mathbb{R}_+^3,\\
				\p_t b^0 + u^0 \cdot \nabla b^0
				-\Delta_h b^0
				=b^0\cdot \nabla u^0 + \p_2 u^0
				\quad & {\rm in} ~~ \mathbb{R}_+^3,\\
				\nabla \cdot u^0 = 0,\quad
				\nabla \cdot b^0 = 0 \quad & {\rm in} ~~ \mathbb{R}_+^3,\\
				u_3^0 = 0, \quad b_3^0 = 0,
				\quad & {\rm on} ~~ \mathbb{R}^2 \times \{ x_3=0 \},
			\end{array}\right.
		\end{equation}
		with the initial data
		\beq\label{eqr0-i}
		(u^0, b^0)|_{t=0}=(u_0, b_0):=(u_0, B_0-e_2).
		\deq
		{Thus, our target in this paper is to establish the global in time uniform estimate
			for the anisotropic incompressible MHD equation \eqref{eqr} and
			establish the uniform explicit convergence rate of solutions
			between \eqref{eqr} and \eqref{eqr0}.
			This reveals the mechanism of how the magnetic field
			generates the enhance dissipation and helps to stabilize the fluid
			in the process of vanishing viscosity limit.}
		
		Due to the smallness of parameter $\ep$ in \eqref{eqr}, it is challenging to study the
		global stability problem.
		To this end, let us introduce the recent developments on the stability and
		decay of the anisotropic incompressible MHD equations.
		There are considerable progress on the stability of the background magnetic field for the ideal MHD equations and the MHD equations with partial dissipative structure or damping term.
		On the one hand, the nonlinear stability for the ideal MHD equations can be found in \cite{Longtime1988,CL2018,HXY2018,PZZ2018,WZ2017}.
On the other hand, there have been substantial recent developments on the well-posedness and stability problems for the MHD equations with partial dissipative structure and damping term, and significant progress has been made
(cf.\cite{LXZ2015,DZ2018,HW2010,PZZ2018,RWXZ2014,RXZ2016,TW2018,WW2017,WWX2015,CWY2014,DJJW2018,DLW2019,Fefferman2014,HL2013,JNWXY2015,WuJiahong2018,Wu2021Advance,new-Abidi-Zhang2017,new-Boardman-Lin-Wu2020,new-Cao-Regmi-Wu2013,new-Cao-Wu2011,new-Chen-Zhang-Zhou2022,new-Du-Zhou2015,new-Fefferman-McCormick-Robinson-Rodrigo2017,new-Feng-Hafeez-Wu2021,new-Jiang-Jiang-Zhao2022,new-Lai-Wu-Zhang2021,new-Lai-Wu-Zhang2022,new-Li-Yang-2024,new-Lin-Ji-Wu2020,new-Lin-Zhang2014,new-Lin-Zhang2015,new-Sermange-Temam1983,new-Wei-Zhang2020,new-Xie-Jiu-Liu2024,new-Xu-Zhang2015,new-Zhang2016,new-Hu-Lin-arxiv}).
		Recently, Lin, Wu and Zhu \cite{Lin-Wu-Zhu-2025} successfully applied
		the decay rate estimate of linear system and estimate of nonlinear system to
		establish the global stability for the incompressible MHD equation \eqref{eqr0} under small perturbations near a background magnetic field.
		However, due to the appearance of small viscous term in \eqref{eqr}
		and domain boundary, it is tricky problem for us to establish the global uniform energy estimate
		under the framework of classical Sobolev space as in \cite{Lin-Wu-Zhu-2025}.
		The reason is that we can't establish similar decay rate estimate and
		higher order normal derivative as the case of three-dimensional
		whole space.
		
		In order to find the suitable energy framework for the equation \eqref{eqr},
		let us introduce the progress about the full vanishing viscosity limit of
		the incompressible Navier-Stokes equation.
		The vanishing viscosity limit of the incompressible Navier-Stokes is one of the most fundamental problems in fluid mechanics. This inviscid limit in the whole space case has been examined
		by many authors, see, for instance, \cite{Constantin1986,Constantin1988,Kato1972,Masmoudi2007CMP, ConstantinWu1995,ConstantinWu1996}.
		However, in the presence of physical boundary, the problem
		becomes much more complicated due to the
		possible appearance of boundary layer.
		Considerable progress has also been made for the incomprssible Navier-Stokes
		equation supplemented with Navier-slip boundary condition.
		Indeed, for some special types of Navier boundary conditions, the main part of the
		boundary layer vanishes and, as a consequence,
		uniform bounds and uniform existence time (independent of viscosity)
		can be established (cf.\cite{BeiraodaVeiga2010,BeiraodaVeiga2011,Xiao2007}).
		However, as pointed in \cite{Iftimie2011ARMA},  uniform regularity estimates
		in Sobolev spaces are not expected for general curved boundaries.
		Masmoudi and Rousset \cite{Masmoudi2012ARMA} successfully established uniform estimates in conormal Sobolev spaces
		for the 3D general smooth domains with the Navier-slip
		boundary conditions. The convergence
		of the viscous solutions to the inviscid ones is then obtained by a compactness argument.
		Furthermore, based on these uniform estimates, \cite{Xiao2013,Gie2012} provided some better convergence rates.
		Finally, the three dimensional viscous MHD equation will converge to the
		ideal MHD equation under the slip boundary condition(cf.\cite{{Xiao2009JFA},{MR3472518}}).
		The above full vanishing viscosity limit progress and convergence rate of solution
		between viscous equation and inviscid equation hold only in a small time interval.
		
		Due to the smallness of parameter $\ep$ in \eqref{eqr}
		and appearance of domain boundary, we hope to establish the energy estimate
		under the framework of conormal Sobolev space rather than the classical
		Sobolev space as in \cite{Lin-Wu-Zhu-2025}.
			To define the conormal Sobolev spaces, we consider $(Z_k)_{1 \le k \le 3}$, a finite set of generators of vector fields tangent to the boundary $\mathbb{R}^2 \times \{x_3=0\}$, with $Z_k = \partial_k$, $k=1,2$ and $Z_3= \varphi (x_3) \partial_3$, where $\varphi(x_3)$ is any smooth bounded function such that $\varphi(0)=0$, $\varphi'(0) \neq 0$ and $\varphi(x_3)>0$ for every $x_3>0$ (for example, $\varphi(x_3)=x_3(1+x_3)^{-1}$ fits). We set
			\begin{equation*}
				\begin{aligned}
					H_{co}^m(\mathbb{R}_+^3) :=
					\{ f \in
					L^2(\mathbb{R}_+^3)~|~Z^\alpha f \in L^2(\mathbb{R}_+^3),~~\forall 0\le |\alpha| \le m \},
				\end{aligned}
			\end{equation*}
			where $Z^\alpha := Z^{\alpha_1}_1 Z^{\alpha_2}_2 Z^{\alpha_3}_3$.
			To give a precise account of our main results,
			we define $\nabla_h:=(\p_1, \p_2, 0)$ and $\alpha_h:=(\alpha_1, \alpha_2, 0)$.
			We also use the following notations, for every $m \in \mathbb{N}$:
			\begin{equation*}
				\begin{aligned}
					&\Vert f \Vert_{H^m_{tan}}^2 :=
					\sum\limits_{0\le |\alpha_h| \le m} \Vert Z^{\alpha_h} f \Vert_{L^2}^2,
					\quad
					\Vert f \Vert_{H^m_{co}}^2 :=
					\sum\limits_{0\le|\alpha| \le m} \Vert Z^\alpha f \Vert_{L^2}^2.
				\end{aligned}
			\end{equation*}
			Finally, we also define the operator $\Lambda^{s}_h$, $s \in \mathbb{R}$ by
			\begin{align*}
				\Lambda_h^s f(x_h) :=
				\int_{\mathbb{R}^2} |\xi_h|^s \hat f(\xi_h) {\rm e}^{2\pi i x_h \cdot \xi_h} d\xi_h,
			\end{align*}
			where $x_h:=(x_1, x_2, 0)$, $\xi_h:=(\xi_1,\xi_2, 0)$ and $\hat{f}$ is the Fourier transform of $f$ in $\mathbb{R}^2$.
			Let us define
			\beqq
				\begin{aligned}
					\mathcal{E}(u^\ep, b^\ep)(t):=
					&\| (u^\ep, b^\ep)(t)\|_{H^m_{co}}^2
					+\|(\p_3 u^\ep, \p_3 b^\ep)(t)\|_{H^{m-1}_{co}}^2
					+\|(\p_{33} u^\ep, \p_{33} b^\ep)(t)\Vert_{H^{m-2}_{co}}^2 \\
					&+\| \Lambda^{-s}_h(u^\ep, b^\ep)(t)\|_{H^{m-1}_{tan}}^2
					+\| \Lambda^{-s}_h(\p_3 u^\ep, \p_3 b^\ep)(t)\|_{H^{m-2}_{tan}}^2,
				\end{aligned}
			\deqq
			and the dissipation norm
			\beqq
			\begin{aligned}
				\mathcal{D}(u^\ep, b^\ep, \ep)(t)
				:=&\|\p_1 u^\ep(t)\|_{H^m_{co}}^2
				+\|\p_{13} u^\ep(t)\|_{H^{m-1}_{co}}^2
				+\|\p_{133} u^\ep(t)\|_{H^{m-2}_{co}}^2
				+\|\nabla_h b^\ep(t)\|_{H^m_{co}}^2
				+\|\nabla_h \p_{3} b^\ep(t)\|_{H^{m-1}_{co}}^2\\
				&+\|\nabla_h \p_{33} b^\ep(t)\|_{H^{m-2}_{co}}^2
                +\|\p_2 u^\ep(t)\|_{H^{m-1}_{co}}^2
				+\|\p_{23} u^\ep(t)\|_{H^{m-2}_{co}}^2
				+\|\p_{233} u^\ep(t)\|_{H^{m-3}_{co}}^2\\
				&+\ep (\|(\p_2 u^\ep, \p_3u^\ep)(t)\|_{H^m_{co}}^2
				+\|(\p_{23}u^\ep, \p_{33}u^\ep)(t)\|_{H^{m-1}_{co}}^2
                +\|(\p_{233}u^\ep, \p_{333}u^\ep)(t)\|_{H^{m-2}_{co}}^2
				)\\
				&+\ep(\|\p_3 b^\ep(t)\|_{H^m_{co}}^2
				+\|\p_{33} b^\ep(t)\|_{H^{m-1}_{co}}^2
				+\|\p_{333} b^\ep(t)\|_{H^{m-2}_{co}}^2).
			\end{aligned}
			\deqq
            Due to the dissipative structure in the $x_1$ and $x_2$ directions
            both in \eqref{eqr} and \eqref{eqr0}, it is easy to verify that
            the solution $(u^\var, b^\var)$ of \eqref{eqr} will converge
            to the solution $(u^0, b^0)$ of \eqref{eqr0} in a short time
            as $\var$ tends to zero.
            Thus, our target in this paper is to prove that
            $(u^\var, b^\var)$ converges to $(u^0, b^0)$ globally as
            $\var$ tends to zero.
            To this end, the global in time uniform (with respect to $\var$)
            estimate plays a decisive role in this vanishing dissipation limit
            progress.
			Now we state our first result concerning on the global-in-time uniform
			regularity and time decay rate as follows.
			\begin{theo}\label{main_result_one}
				For every integer $m \ge 4$ and constants
				$(\sigma, s)$ satisfying $\frac{9}{10}<\sigma<s<1$, assume the initial data $(u_0, b_0)$ satisfy
				$\nabla \cdot u_0=\nabla \cdot b_0=0$ and there exists a small positive constant $\delta_0$
				such that
				\begin{equation}\label{condition-one}
					\mathcal{E}(u^\ep, b^\ep)(0) \le \delta_0,
				\end{equation}
				then the equation \eqref{eqr} has a unique global solution $(u^\ep, b^\ep)$ satisfying
				\begin{equation}\label{uniform_estimate}
					\mathcal{E}(u^\ep, b^\ep)(t)+\int_0^t \mathcal{D}(u^\ep, b^\ep, \ep)(\tau) d\tau
					\le C\mathcal{E}(u^\ep, b^\ep)(0),
				\end{equation}
				and the decay estimate
				\begin{equation*}
					\begin{aligned}
						&(1+t)^s(\|(u^\ep, b^\ep)(t)\|_{H^{m-1}_{tan}}^2
						+\|(\p_3 u^\ep, \p_3 b^\ep)(t)\|_{H^{m-2}_{tan}}^2)
						+\int_0^t (1+\tau)^{\sigma}\|\p_{23}u^\ep(\tau)\|_{H^{m-3}_{tan}}^2 d\tau\\
						&+\int_0^t (1+\tau)^{\sigma}(\|(\p_1 u^\ep, \nabla_h b^\ep)(\tau)\|_{H^{m-1}_{tan}}^2
						+\|(\p_{13}u^\ep, \nabla_h \p_3 b^\ep,\p_2 u^\ep)(\tau)\|_{H^{m-2}_{tan}}^2)d\tau\\
						&+\ep \int_0^t(1+\tau)^{\sigma}
						(\|(\p_2 u^\ep, \p_3 u^\ep, \p_3 b^\ep)(\tau)\|_{H^{m-1}_{tan}}^2
						+\|(\p_{23}u^\ep,\p_{33}u^\ep,\p_{33}b^\ep)(\tau)\|_{H^{m-2}_{tan}}^2)d\tau
						\le C\delta_0,
					\end{aligned}
				\end{equation*}
				where $C$ is a positive constant independent of $\ep$ and time $t$.
			\end{theo}
			
			Since the uniform estimate and decay rate in Theorem \ref{main_result_one}
			are independent of $\ep$, we can apply this method
			to the equation \eqref{eqr0} and establish the following result.
			\begin{theo}\label{main-coro}
				For every integer $m \ge 4$ and constants
				$(\sigma, s)$ satisfying $\frac{9}{10}<\sigma<s<1$, assume the initial data $(u_0, b_0)$ satisfy
				$\nabla \cdot u_0=\nabla \cdot b_0=0$ and there exists a small positive constant $\delta_0$
				such that
				\begin{equation*}
					\mathcal{E}(u^0, b^0)(0) \le \delta_0,
				\end{equation*}
				then the equation \eqref{eqr0} has a unique global solution $(u^0, b^0)$ satisfying
				\begin{equation*}
					\mathcal{E}(u^0, b^0)(t)+\int_0^t \mathcal{D}(u^0, b^0, 0)(\tau) d\tau
					\le C\mathcal{E}(u^0, b^0)(0),
				\end{equation*}
				and the decay estimate
				\begin{equation*}
					\begin{aligned}
						&(1+t)^s(\|(u^0, b^0)(t)\|_{H^{m-1}_{tan}}^2
						+\|(\p_3 u^0, \p_3 b^0)(t)\|_{H^{m-2}_{tan}}^2)
						+\int_0^t (1+\tau)^{\sigma}\|\p_{23}u^0(\tau)\|_{H^{m-3}_{tan}}^2 d\tau\\
						&+\int_0^t (1+\tau)^{\sigma}(\|(\p_1 u^0, \nabla_h b^0)(\tau)\|_{H^{m-1}_{tan}}^2
						+\|(\p_{13}u^0, \nabla_h \p_3 b^0,\p_2 u^0)(\tau)\|_{H^{m-2}_{tan}}^2)d\tau
						\le C\delta_0,
					\end{aligned}
				\end{equation*}
				where $C$ is a positive constant independent of time $t$.
			\end{theo}
			
			\begin{rema}
			Recently, Lin, Wu and Zhu \cite{Lin-Wu-Zhu-2025} successfully established the global stability of the limit equation \eqref{eqr0} in the three-dimensional whole space $\mathbb{R}^3$, assuming the initial data is a small perturbation of the steady-state solution induced by a background magnetic field.
			Our results in Theorems \ref{main_result_one} and  \ref{main-coro} provide a global well-posedness framework for both the viscous equation \eqref{eqr} and the limit equation \eqref{eqr0} respectively. In comparison to \cite{Lin-Wu-Zhu-2025}, the energy estimate \eqref{uniform_estimate} involves only second-order normal derivatives, reflecting the impact of the boundary layer.
			Furthermore, we develop a two-tier energy method that couples the boundedness
			of conormal derivative to the decay of tangential derivative, the latter of which is necessary to balance out the growth of conormal derivative.
		\end{rema}
		
		\begin{rema}
			To the authors' best knowledge,  the global well-posedness and large time behavior of the anisotropic incompressible Navier-Stokes equation \eqref{03} is still not well-understood.
			However, due to the stability induced by a background magnetic field, we can
			establish the global uniform regularity of anisotropic MHD equation \eqref{eqr}.
			This stability induced by a background magnetic field plays an important role
			in the progress of global in time vanishing viscosity limit.
			Given the exclusively horizontal dissipative structure of the magnetic field in \eqref{eqr}, we leverage the enhanced dissipation induced by the background magnetic field and introduce four layers of energy functionals.
		\end{rema}	
			
			
           The local in time solution $(u^\var, b^\var)$ of \eqref{eqr}
           and solution $(u^0, b^0)$ of \eqref{eqr0} have been extended to be
           global one in Theorems \ref{main_result_one} and \ref{main-coro}
           respectively.
           Thus, our final task is to verify that solution $(u^\var, b^\var)$
           will converge to $(u^0, b^0)$ globally.
			To this end, we will establish the explicit convergence rate of solutions
			between \eqref{eqr} and \eqref{eqr0}.
			The advantage of this rate is not only to show that
		the solution $(u^\ep, b^\ep)$ of \eqref{eqr} converges to $(u^0, b^0)$ of \eqref{eqr0}
			as $\ep$ tends to zero, but also is independent of time $t$.
			
			\begin{theo}\label{main_result_two}
				For every integer $m \ge 5$ and constants
	$(\sigma, s)$ satisfying $\frac{9}{10}<\sigma<s<1$, assume the initial data $(u_0, b_0)$ satisfy
	$\nabla \cdot u_0=\nabla \cdot b_0=0$ and there exists a small positive constant $\delta_*$
				such that
				\begin{equation}\label{condition-two}
			\mathcal{E}(u^\ep, b^\ep)(0)+\|\Lambda_h^{-s}\p_3^2(u_0, b_0)\|_{L^2}^2 \le \delta_*.
				\end{equation}
				Then, the global solution $(u^\ep, b^\ep)$ of \eqref{eqr}
				will converge to $(u^0, b^0)$ of \eqref{eqr0} with the convergence rate
				\begin{equation}\label{deo}
					\|(u^\ep-u^0, b^\ep-b^0)(t)\|_{L^2}
					\le C \ep^{\frac14},
				\end{equation}
				and
				\begin{equation}\label{det}
					\|(u^\ep-u^0, b^\ep-b^0)(t)\|_{L^\infty}
					\le C \ep^{\frac{1}{8}},
				\end{equation}
				where $C$ is a positive constant independent of $\ep$ and time $t$.
			\end{theo}
			
			\begin{rema}
				In order to make sure that the convergence rates \eqref{deo} and \eqref{det}
				are independent of time, we need to establish the decay in time estimate
				for the second order normal derivative of solution.
				Thus, we need to require the regularity index $m\ge 5$
				and smallness of initial data $\Lambda_h^{-s}\p_3^2(u_0, b_0)$ in $L^2-$norm.
                Due to the estimates \eqref{deo} and \eqref{det} independent of time, we can obtain the convergence as follows
                \beqq
                \|(u^\var-u^0, b^\var-b^0)(t)\|_{L^\infty([0,+\infty);L^2(\mathbb{R}^3_+))}
                +\|(u^\var-u^0, b^\var-b^0)(t)\|_{L^\infty([0,+\infty);L^\infty(\mathbb{R}^3_+))}
                \rightarrow 0
                \deqq
                as $\var$ tends to zero.
			\end{rema}

			Throughout this paper, we use symbol $A \lesssim B$ for $ A\le C B$ where $C>0$ is a constant which may change from line to line and independent of
			time $t$ and  $\varepsilon$ for $0<\varepsilon <1$.
			
			The rest of the paper is organized as follows.
			In Section \ref{section-approach}, we explain the difficulties and our approach to establish
			the global uniform estimate and time decay rate estimate
			for the system  \eqref{eqr} and \eqref{eqr0} respectively.
			In Section  \ref{global-estimate}, we apply the four layers of energy functionals
			to establish the global uniform and decay rate estimate
			for system  \eqref{eqr} under the condition of small initial data.
			In Section  \ref{asymptotic-behavior}, we will establish the time decay
			rate for the second order normal derivative of solution for
			system \eqref{eqr} under the addition condition \eqref{condition-two}.
			This decay rate provides us a specific convergence rate for the solutions
			between \eqref{eqr} and \eqref{eqr0} as the parameter $\ep$ tends to zero.
			Finally, we introduce some useful inequalities in Appendix \ref{usefull-inequality}
			and provide the detailed proofs for the claimed technical estimates in Appendix
			\ref{claim-estimates}.

			\section{Difficulties and outline of our approach}\label{section-approach}
			In this section, we will explain the main difficulties of proving Theorems
			\ref{main_result_one} and \ref{main_result_two} as well as our strategies for overcoming them.
			Indeed, we will establish the global uniform  estimate under the framework
			of conormal Sobolev space due to the effect of boundary layer.
			Due to the lower dissipative structure of velocity in the $x_2$ direction,
			we will encounter the difficult term
			$\i \p_2 u_2 Z^{\alpha}\p_3 \wu \cdot \p_3 Z^{\alpha} \wu dx$
			(see \eqref{3404}) as we hope to establish the uniform in time energy estimate.
			Thus, we handle with this term as follows
			\beq\label{201}
			\begin{aligned}
				&\int_0^t \i \p_2 u_2 Z^{\alpha}\p_3 \wu \cdot \p_3 Z^{\alpha} \wu dx d\tau\\
				\lesssim
				&\int_0^t \|(\p_2 u_2, \p_{23} u_2)\|_{H^1_{tan}}
				\|\p_{13} Z^{\alpha} \wu\|_{L^2}^{\frac{1}{2}}
				\|\p_3 \wu\|_{H^{m-2}_{co}}^{\frac32}d\tau\\
				\lesssim
				&(\underset{0\le \tau \le t}{\sup}\|\p_3 \wu\|_{H^{m-2}_{co}})^{\frac32}
				\left\{\int_0^t   \|(\p_2 u_2, \p_{23} u_2)\|_{H^1_{tan}}^2
				(1+\tau)^{\sigma}d\tau\right\}^\frac{1}{2} \\
				&\times \left\{\int_0^t   \|\p_{13} Z^\al \wu\|_{L^2}^2 d\tau\right\}^{\frac{1}{4}}
				\left\{\int_0^t   (1+\tau)^{-2\sigma}d\tau\right\}^{\frac{1}{4}},
			\end{aligned}
			\deq
			where the index $\sigma \in (\frac{9}{10}, 1)$ and $|\alpha|=m-2$
            is the regularity index.
			Thus, these estimates require us not only to establish the global uniform
			estimate under the conormal Sobolev space but also build some suitable
			decay rate for the tangential derivative of velocity.
			Thus, our proof will divide into the following two steps.
			
			\textbf{Step 1: Estimate of tangential derivative and its decay rate.}
			First of all, we establish the estimate for the tangential derivative
			of velocity and magnetic field and control the nonlinear term by the
			tangential derivative of dissipation norm and conormal derivative energy norm.
			Secondly, due to the divergence-free condition, we can establish
			the estimate for the tangential derivatives of vorticity instead of
			the normal derivative of solution.
			This can help us avoid to dealing with the pressure term.
			Thirdly, we will establish the dissipation estimate for
			the derivative of velocity in the $x_2$ direction.
			Due to the good effect of a background magnetic field,
			we employ the good term $\p_2 u$ in the equation $\eqref{eqr}_2$
			to establish the dissipation estimate $\|\p_2 u\|_{H^{m-1}_{co}}^2$.
			Then, we can obtain the following differential inequality
			\beq\label{202}
			\frac{d}{dt}\widehat{\mathcal{E}}^{m-1}_{tan}(t)
			+\kappa \md_{tan}^{m-1}(t)\le 0.
			\deq
			Finally, we will establish the uniform estimate of solution
			and itself vorticity in the negative Sobolev space under
			the time weighted assumption. Then, due to these four layers of energy functionals,
			we rewrite the differential inequality \eqref{202}
			as the following form
			\beqq
			\frac{d}{dt}\widehat{\mathcal{E}}^{m-1}_{tan}(t)
			+\kappa C_0^{-\frac{1}{s}}
			\widehat{\mathcal{E}}^{m-1}_{tan}(t)^{1+\frac{1}{s}}\le 0,
			\deqq
			which yields directly the decay estimate
			\beqq
			\widehat{\mathcal{E}}^{m-1}_{tan}(t)
			\lesssim C_0(1+t)^{-s}.
			\deqq
			This estimate and the differential inequality \eqref{202}
			will help us establish the time weighted estimate
			\beq\label{203}
			(1+t)^{s}{\mathcal{E}}^{m-1}_{tan}(t)
			+\kappa \int_0^t (1+\tau)^{\sigma} \md_{tan}^{m-1}(\tau)d\tau
			\le CC_0.
			\deq
			This estimate will help us give the control of the term
			on the right hand side of inequality \eqref{201}.
			
			\textbf{Step 2: Estimate of second order normal derivative
				and enhanced dissipation.}
			In order to close the energy estimate, we need to establish the estimates for
			normal derivative of velocity and magnetic field.
			In this progress, we will encounter the difficult term
			$\i \p_2 u_2 Z^{\alpha}\p_3 \wu \cdot \p_3 Z^{\alpha} \wu dx$.
			Then, we control the estimate for this term by the method as \eqref{201}.
			At this time, the decay in time estimate of tangetial derivative \eqref{203}
			will return to balance out the growth of conormal derivative estimate.
			In other words, in order to deal with the difficult term $\i \p_2 u_2 Z^{\alpha}\p_3 \wu \cdot \p_3 Z^{\alpha} \wu dx$,
			our method here constructs a two-tier energy method
(motivated by \cite{{Guo-Tice-2013ARMA},{Guo-Tice-2013PDE}})
 that couples the boundedness
			of conormal derivative to the decay of tangential derivative, the latter of which is necessary to balance out the growth of conormal derivative.
			
			Finally, let us focus on the convergence rate of solutions between
			\eqref{eqr} and \eqref{eqr0}.
			Let $(u^\ep, b^\ep)$ and $(u^0, b^0)$ be the global solutions of systems \eqref{eqr}
			and \eqref{eqr0} with the same initial data assumption.
			By energy method, one can establish the following estimate
			\beq\label{204}
			\begin{aligned}
				&\|(u^\ep-u^0, b^\ep-b^0)(t)\|_{H^1}^2\\
				\le
				&C \ep \int_0^t \|(\p_{22}u^\ep, \p_{33} u^\ep, \p_{33}b^\ep)(\tau)
				\|_{H^1}\|(u^\ep-u^0, b^\ep-b^0)(t)\|_{H^1} d\tau\\
				&\times\exp\left\{\int_0^t (\mathcal{{B}}(\tau)\mathcal{{B}}_h(\tau)
				+\|\nabla_h (u^0, u^\ep)(\tau)\|_{H^1}^{\frac{2}{3}}
				\|\nabla_h^2 (u^0, u^\ep)(\tau)\|_{H^1}^{\frac{2}{3}}) d\tau\right\}.
			\end{aligned}
			\deq
			Suppose the initial data satisfies the small condition \eqref{condition-two},
			we can establish the decay rate for the quantity
			$\p_{33}(u^\ep, b^\ep)$ in $L^2-$norm.
			Then, due to the time decay estimate built before, the specific convergence rate \eqref{204} will only depend on the parameter $\ep$ and be independent of time.

			\section{Global in time uniform regularity}\label{global-estimate}
			
			In this section, we will establish global-in-time well-posedness
			in conormal Sobolev space for the equations \eqref{eqr}
			under the small initial data \eqref{condition-one}.
			First of all, similar to the result in \cite{MR3472518},
			one can establish the uniform (with respect to $\ep$)
			local-in-time existence and uniqueness for system \eqref{eqr}.
			Then, we will extend the local-in-time solution to be global one.
			Thus, our target in this section is to establish the global-in-time
			uniform regularity under the condition of small initial data  \eqref{condition-one}.
			For notational convenience, we drop the superscript $\var$ throughout this section.
			Let us denote $\wu:=\nabla \times u$ and $\wb:=\nabla \times b$,
			and we define the energy norms
			\beq\label{e-1}
			\begin{aligned}
				\me_{tan}^{k}(t)
				:=&\|(u, b)(t)\|_{H^{k}_{tan}}^2+\|(\wu, \wb)(t)\|_{H^{k-1}_{tan}}^2,\\
				\mathcal{E}^{k}(t)
				:=&\me_{tan}^{k}(t)+\|\p_3(\wu, \wb)(t)\|_{H^{k-2}_{co}}^2,
			\end{aligned}
			\deq
			and the dissipation norms
			\beq\label{d-1}
			\begin{aligned}
				\mathcal{D}_{tan}^{k}(t)
				:=&\|(\p_1 u, \nabla_h b)(t)\|_{H^k_{tan}}^2
				+\|(\p_1 w^u, \nabla_h w^b)(t)\|_{H^{k-1}_{tan}}^2
				+\|\p_2 u(t)\|_{H^{k-1}_{tan}}^2+\|\p_2 w^u(t)\|_{H^{k-2}_{tan}}^2,\\
				\md^{k}(t)
				:=&
				\mathcal{D}_{tan}^{k}(t)+\|\p_{23} \wu(t)\|_{H^{k-3}_{co}}^2
				+\|(\p_{13} \wu, \nabla_h \p_3 \wb)(t)\|_{H^{k-2}_{co}}^2.
			\end{aligned}
			\deq
			Now, let us state the global uniform estimate as follows.
			\begin{prop}\label{main_pro}
				For every integer $m \ge 4$ and constants
				$(\sigma, s)$ satisfy $\frac{9}{10}<\sigma<s<1$, assume the initial data $(u_0, b_0)$ satisfying
				$\nabla \cdot u_0=\nabla \cdot b_0=0$.
				For the solution $(u, b)$ of equation \eqref{eqr}
				defined on $ [0,T] \times \mathbb{R}^3_+$, assume there exists
				a small positive constant $\delta$ such that
				\begin{equation}\label{assumption}
					\mathcal{E}^m(t)
					+(1+t)^s \mathcal{E}_{tan}^{m-1}(t)
					+\int_0^t (1+\tau)^{\sigma} \md_{tan}^{m-1}(\tau)d\tau
					+\int_0^t \md^{m}(\tau)d\tau \le \delta,
				\end{equation}
				then the  solution of equation \eqref{eqr} has the estimate
				\begin{equation}\label{close_assumption}
					\mathcal{E}^m(t)
					+(1+t)^s \mathcal{E}_{tan}^{m-1}(t)
					+\int_0^t (1+\tau)^{\sigma} \md_{tan}^{m-1}(\tau)d\tau
					+\int_0^t \md^{m}(\tau)d\tau \le \frac{\delta}{2}.
				\end{equation}
				Here the small positive constant $\delta:=8C(\mathcal{E}^m(0)
				+\|(\Lambda_h^{-s}u, \Lambda_h^{-s} b,
				\Lambda_h^{-s}\wu,\Lambda_h^{-s} \wb)(0)\|_{L^2}^2)$
				and $C$ is a positive constant independent of time $t$ and
				parameter $\ep$.
			\end{prop}

			\subsection{Tangential and one order normal derivative estimate}\label{sec-half space}
			
			First of all, we establish the estimate for the tangential derivative
			of velocity and magnetic field.
			\begin{lemm}\label{lemma32}
				For any smooth solution $(u, b)$ of equation \eqref{eqr},
				it holds for any positive integer $k \le m$
				\beq\label{3101}
				\frac{d}{dt}\|(u, b)(t)\|_{H^k_{tan}}^2
				+\|(\p_1 u, \nabla_h b)(t)\|_{H^k_{tan}}^2
				+\ep \|(\p_2 u, \p_3 u, \p_3 b)(t)\|_{H^k_{tan}}^2
				\lesssim  \sqrt{\me^m(t)}\md_{tan}^{k}(t).
				\deq
			\end{lemm}
			\begin{proof}
				For any $|\al_h|=k\le m$, the equation \eqref{eqr} yields directly
				\begin{equation}\label{3102}
					\begin{aligned}
						&\frac{d}{dt}\frac{1}{2}\i(|Z^{\ah}u|^2+|Z^{\ah}b|^2)dx
						+\i(|\p_1 Z^{\ah}u|^2+|\nabla_h Z^{\ah}b|^2)dx\\
						&+\ep \i(|\p_2 Z^{\ah}u|^2+|\p_3 Z^{\ah}u|^2+|\p_3 Z^{\ah}b|^2)dx\\
						=&\i Z^{\ah}(-u\cdot \nabla u+b\cdot \nabla b)\cdot Z^{\ah} u \ dx
						+\i Z^{\ah}(-u\cdot \nabla b+b \cdot \nabla u)\cdot Z^{\ah} b\ dx,
					\end{aligned}
				\end{equation}
				where we have used the basic fact
				\beqq
				\i Z^{\ah} \p_2 b \cdot Z^{\ah} u \ dx
				+\i Z^{\ah} \p_2 u \cdot Z^{\ah} b \ dx=0.
				\deqq
				If $\alpha_h=0$, then we can obtain
				\begin{equation}\label{3103}
					\frac{d}{dt}\frac{1}{2}\i(|u|^2+|b|^2)dx
					+\i(|\p_1 u|^2+|\nabla_h b|^2)dx\\
					+\ep \i(|\p_2 u|^2+|\p_3 u|^2+|\p_3 b|^2)dx=0.
				\end{equation}
				\textbf{Now let us deal with the case $\alpha_h > 0$}.
				Obviously, it holds
				\beqq
				\i Z^{\ah}(-u\cdot \nabla u)\cdot Z^{\ah} u\ dx
				=-\sum_{0\le \beta_h \le \alpha_h}C^{\beta_h}_{\alpha_h}
				\i (Z^{\beta_h} u \cdot Z^{\al_h-\beta_h}\nabla u)\cdot Z^{\al_h} u\ dx.
				\deqq
				If $\beta_h=0$, integrating by part and using the divergence-free condition,
				we conclude
				\beqq
				-\i(u \cdot Z^{\al_h}\nabla u)\cdot Z^{\al_h} u\ dx
				=\frac{1}{2}\i |Z^{\al_h} u|^2 {\rm div} u\ dx=0.
				\deqq
				If $\beta_h=\alpha_h>0$, the anisotropic type inequality \eqref{ie:Sobolev} yields directly
				\beqq
				\begin{aligned}
					&\i (Z^{\al_h} u \cdot \nabla u)\cdot Z^{\al_h} u\ dx\\
					\lesssim
					&\|Z^{\al_h} u\|_{L^2}^{\frac12}\|\p_1 Z^{\al_h} u\|_{L^2}^{\frac12}
					\|\nabla u\|_{L^2}^{\frac14}
					\|\p_2 \nabla u\|_{L^2}^{\frac14}
					\|\p_3 \nabla u\|_{L^2}^{\frac14}
					\|\p_{23}\nabla u\|_{L^2}^{\frac14}
					\|Z^{\al_h} u\|_{L^2}\\
					\lesssim
					&\sqrt{\me^m(t)}\md_{tan}^{k}(t).
				\end{aligned}
				\deqq
				If $1<|\beta_h|<|\alpha_h|$, we apply  the anisotropic
				type inequality \eqref{ie:Sobolev} to obtain
				\beqq
				\begin{aligned}
					&\i (Z^{\beta_h} u \cdot Z^{\al_h-\beta_h}\nabla u)\cdot Z^{\al_h} u dx\\
					\lesssim
					&\|Z^{\beta_h} u\|_{L^2}^{\frac12}\|\p_3 Z^{\beta_h} u\|_{L^2}^{\frac12}
					\|Z^{\al_h-\beta_h}\nabla u\|_{L^2}^{\frac12}
					\|\p_2 Z^{\al_h-\beta_h}\nabla u\|_{L^2}^{\frac12}
					\|Z^{\al_h} u\|_{L^2}^{\frac12}\|\p_1 Z^{\al_h} u\|_{L^2}^{\frac12}\\
					\lesssim
					&\sqrt{\me^m(t)}\md_{tan}^{k}(t),
				\end{aligned}
				\deqq
				and for $|\beta_h|=1$, we have
				\beqq
				\begin{aligned}
					&\i (Z^{\beta_h} u \cdot Z^{\al_h-\beta_h}\nabla u)\cdot Z^{\al_h} u\ dx\\
					\lesssim
					&\|Z^{\beta_h} u\|_{L^2}^{\frac14}\|\p_3 Z^{\beta_h} u\|_{L^2}^{\frac14}
					\|\p_2 Z^{\beta_h} u\|_{L^2}^{\frac14}\|\p_{23} Z^{\beta_h} u\|_{L^2}^{\frac14}
					\|Z^{\al_h-\beta_h}\nabla u\|_{L^2}
					\|Z^{\al_h} u\|_{L^2}^{\frac12}\|\p_1 Z^{\al_h} u\|_{L^2}^{\frac12}\\
					\lesssim
					&\sqrt{\me^m(t)}\md_{tan}^{k}(t).
				\end{aligned}
				\deqq
				Thus, the combination of above estimates yields directly
				\beq\label{3104}
				\i Z^{\ah}(-u\cdot \nabla u)\cdot Z^{\ah} u\ dx
				\lesssim \sqrt{\me^m(t)}\md_{tan}^{k}(t).
				\deq
				Next, integrating by part and using the
				boundary condition $\eqref{eqr}_4$ give directly
				\beqq
				\i (b\cdot \nabla Z^{\ah}b) \cdot Z^{\ah}u\ dx
				+\i (b \cdot \nabla Z^{\ah}u) \cdot Z^{\ah}b\ dx=0,
				\deqq
				then we achieve
				\beq\label{3105}
				\begin{aligned}
					&\i Z^{\ah}(b\cdot \nabla b)\cdot Z^{\ah} u \ dx
					+\i Z^{\ah}(b \cdot \nabla u)\cdot Z^{\ah} b\ dx\\
					=&
					\sum_{0< \beta_h \le \alpha_h}C^{\beta_h}_{\alpha_h}
					\i (Z^{\beta_h}b\cdot \nabla Z^{\ah-\beta_h}b)
					\cdot Z^{\ah}u\ dx
					+\sum_{0< \beta_h \le \alpha_h}C^{\beta_h}_{\alpha_h}
					\i (Z^{\beta_h}b \cdot \nabla Z^{\ah-\beta_h}u)
					\cdot Z^{\ah}b\ dx\\
					\lesssim
					&\sum_{0< \beta_h \le \alpha_h}C^{\beta_h}_{\alpha_h}
					\|Z^{\beta_h}b\|_{L^2}^{\frac12}
					\|\p_3 Z^{\beta_h}b\|_{L^2}^{\frac12}
					\|\nabla Z^{\ah-\beta_h}b\|_{L^2}^{\frac12}
					\|\p_2 \nabla Z^{\ah-\beta_h}b\|_{L^2}^{\frac12}
					\|Z^{\ah}u\|_{L^2}^{\frac12}\| \p_1 Z^{\ah}u\|_{L^2}^{\frac12}\\
					&+\sum_{0< \beta_h \le \alpha_h}C^{\beta_h}_{\alpha_h}
					\|Z^{\beta_h}b\|_{L^2}^{\frac12}
					\|\p_2 Z^{\beta_h}b\|_{L^2}^{\frac12}
					\|\nabla Z^{\ah-\beta_h}u\|_{L^2}^{\frac12}
					\|\p_1\nabla Z^{\ah-\beta_h}u\|_{L^2}^{\frac12}
					\|Z^{\ah}b\|_{L^2}^{\frac12}
					\|\p_3 Z^{\ah}b\|_{L^2}^{\frac12}\\
					\lesssim
					&\sqrt{\me^m(t)}\md_{tan}^{k}(t).
				\end{aligned}
				\deq
				Similarly, it is easy to check that
				\beq\label{3106}
				\begin{aligned}
					&\i Z^{\ah}(-u\cdot \nabla b)\cdot Z^{\ah} b\ dx\\
					=
					&\sum_{0< \beta_h \le \alpha_h}C^{\beta_h}_{\alpha_h}
					\i (Z^{\beta_h}u \cdot \nabla Z^{\ah-\beta_h}b)
					\cdot Z^{\ah}b\ dx
					-\frac{1}{2} \i |Z^{\ah}b|^2 {\rm div}u\ dx\\
					\lesssim
					&\sum_{0< \beta_h \le \alpha_h}C^{\beta_h}_{\alpha_h}
					\|Z^{\beta_h}u\|_{L^2}^{\frac12}
					\|\p_1 Z^{\beta_h}u\|_{L^2}^{\frac12}
					\|\nabla Z^{\ah-\beta_h}b\|_{L^2}^{\frac12}
					\|\p_2\nabla Z^{\ah-\beta_h}b\|_{L^2}^{\frac12}
					\|Z^{\ah}b\|_{L^2}^{\frac12}
					\|\p_3 Z^{\ah}b\|_{L^2}^{\frac12}\\
					\lesssim
					&\sqrt{\me^m(t)}\md_{tan}^{k}(t).
				\end{aligned}
				\deq
				Thus, substituting the estimates \eqref{3104},
				\eqref{3105} and \eqref{3106} into \eqref{3102},
				we have
				\beqq
				\frac{d}{dt}\|(u, b)\|_{H^k_{tan}}^2
				+\|(\p_1 u, \nabla_h b)\|_{H^k_{tan}}^2
				+\ep \|(\p_2 u, \p_3 u, \p_3 b)\|_{H^k_{tan}}^2
				\lesssim \sqrt{\me^m(t)}\md_{tan}^{k}(t).
				\deqq
				Therefore, we complete the proof of this lemma.
			\end{proof}
			
			Next, we will establish the estimate for the normal derivative
			of velocity and magnetic field.
			Take the vorticity operator to the equation \eqref{eqr}, we can write	
			\beq\label{eqwu}
			\p_t \wu-\p_1^2 \wu-\ep \p_2^2 \wu-\ep \p_3^2 \wu-\p_2 \wb
			=-u\cdot \nabla \wu+\wu\cdot \nabla u+b \cdot \nabla \wb-\wb\cdot \nabla b,
			\deq
			and
			\beq\label{eqwb}
			\p_t \wb-\Delta_h \wb-\ep \p_3^2 \wb-\p_2 \wu
			=-u\cdot \nabla \wb-\nabla(u\cdot \nabla)\times b
			+b\cdot \nabla \wu+\nabla(b\cdot \nabla)\times u.
			\deq
			Due to the boundary condition $\eqref{eqr}_4$, we get
			\beq\label{bd-wuwb}
			(\wu_h, \wb_h, \p_3 \wu_3, \p_3 \wb_3)|_{x_3=0}=0.
			\deq
			Now, we will establish the estimate for the
			quantity $(\wu, \wb)$ as follows.
			\begin{lemm}\label{lemma33}
				For any smooth solution $(u, b)$ of equation \eqref{eqr},
				it holds for any positive integer $k\le m-1$
				\beq\label{3201}
				\frac{d}{dt}\|(\wu, \wb)\|_{H^{k}_{tan}}^2
				+\|(\p_1 \wu, \nabla_h \wb)\|_{H^{k}_{tan}}^2
				+\ep\|(\p_2 \wu, \p_3 \wu, \p_3 \wb)\|_{H^{k}_{tan}}^2
				\lesssim \sqrt{\me^m(t)}\md_{tan}^{k+1}(t).
				\deq
			\end{lemm}
			
			\begin{proof}
				For any $|\ah|=k\le m-1$, the equations \eqref{eqwu}-\eqref{eqwb} yield directly
				\beq\label{3202}
				\begin{aligned}
					&\frac{d}{dt}\frac{1}{2}\i (|Z^{\ah} \wu|^2+|Z^{\ah} \wb|^2) dx
					+\i(|\p_1 Z^{\ah} \wu|^2+|\nabla_h Z^{\ah} \wb|^2)dx\\
					&+\ep\i(|\p_2 Z^{\ah} \wu|^2+|\p_3 Z^{\ah} \wu|^2+|\p_3 Z^{\ah} \wb|^2)dx\\
					=&\i Z^{\ah}(-u\cdot \nabla \wu)
					\cdot Z^{\ah} \wu dx
					+\i Z^{\ah}(\wu\cdot \nabla u)
					\cdot Z^{\ah} \wu dx\\
					&+\i Z^{\ah}(b \cdot \nabla \wb)
					\cdot Z^{\ah} \wu dx
					+\i Z^{\ah}(-\wb\cdot \nabla b)
					\cdot Z^{\ah} \wu dx\\
					&+\i Z^{\ah}(-u\cdot \nabla \wb)\cdot Z^{\ah} \wb dx
					+\i Z^{\ah}(-\nabla(u\cdot \nabla)\times b)\cdot Z^{\ah} \wb dx\\
					&+\i Z^{\ah}(b\cdot \nabla \wu)\cdot Z^{\ah} \wb dx\
					+\i Z^{\ah}(\nabla(b\cdot \nabla)\times u)\cdot Z^{\ah} \wb dx
					:=\sum_{i=1}^8 I_i.
				\end{aligned}
				\deq
				If $\ah=0$, then integrating by part and
                using the divergence-free condition, we have
				\beqq
				\i (u\cdot \nabla \wu) \cdot \wu dx
				=\i (u\cdot \nabla \wb) \cdot \wb dx
				=\i (b \cdot \nabla \wu) \cdot \wb dx
				+\i (b \cdot \nabla \wb) \cdot \wu dx =0.
				\deqq
				Using the Sobolev inequality \eqref{ie:Sobolev}, we have
				\beqq
				\begin{aligned}
					\i (\wb \cdot \nabla b) \cdot \wu dx
					\lesssim
					&\|\wb_h\|_{L^2}^{\frac12}
					\|\p_2 \wb_h\|_{L^2}^{\frac12}
					\|\nabla_h b\|_{L^2}^{\frac12}
					\|\p_3 \nabla_h b\|_{L^2}^{\frac12}
					\|\wu\|_{L^2}^{\frac12}
					\|\p_1 \wu\|_{L^2}^{\frac12}\\
					&+\|\wb_3\|_{L^2}^{\frac12}
					\|\p_3 \wb_3\|_{L^2}^{\frac12}
					\|\p_3 b\|_{L^2}^{\frac12}
					\|\p_{23} b\|_{L^2}^{\frac12}
					\|\wu\|_{L^2}^{\frac12}
					\|\p_1 \wu\|_{L^2}^{\frac12}\\
					\lesssim
					&\sqrt{\me^m(t)}\md_{tan}^{1}(t),
				\end{aligned}
				\deqq
				and
				\beqq
				\begin{aligned}
					\i \nabla u_i \times \p_i b \cdot \wb dx
					\lesssim
					&\|\nabla u_h\|_{L^2}^{\frac12}
					\|\p_2 \nabla u_h\|_{L^2}^{\frac12}
					\|\nabla_h b\|_{L^2}^{\frac12}
					\|\p_3 \nabla_h b\|_{L^2}^{\frac12}
					\|\wb\|_{L^2}^{\frac12}
					\|\p_1 \wb\|_{L^2}^{\frac12}\\
					&+\|\nabla u_3\|_{L^2}^{\frac12}
					\|\p_3 \nabla u_3\|_{L^2}^{\frac12}
					\|\p_3 b\|_{L^2}^{\frac12}
					\|\p_2 \p_3 b\|_{L^2}^{\frac12}
					\|\wb\|_{L^2}^{\frac12}
					\|\p_1 \wb\|_{L^2}^{\frac12}\\
					\lesssim
					&\sqrt{\me^m(t)}\md_{tan}^{1}(t).
				\end{aligned}
				\deqq
				Similarly, it is easy to check that
				\beqq
				|\i (\wu \cdot \nabla u) \cdot \wu dx|,\,
				|\i \nabla b_i \times \p_i u \cdot \wb dx|
				\lesssim \sqrt{\me^m(t)}\md_{tan}^{1}(t).
				\deqq
				Thus, we can get the estimate
				\beq\label{3203}
				\frac{d}{dt}\|(\wu, \wb)\|_{L^2}^2
				+\|(\p_1 \wu, \nabla_h \wb)\|_{L^2}^2
				+\ep\|(\p_2 \wu, \p_3 \wu, \p_3 \wb)\|_{L^2}^2
				\le C\sqrt{\me^m(t)}\md_{tan}^{1}(t).
				\deq
				\textbf{Now let us deal with the case $\ah>0$}.
				Obviously, it holds
				\beqq
					I_1=
					-\sum_{0\le \beta_h \le \alpha_h}C^{\beta_h}_{\alpha_h}
					\i (Z^{\beta_h}u \cdot \nabla Z^{\ah-\beta_h}\wu)\cdot Z^{\ah} \wu dx.
				\deqq
				If $\beta_h=0$, we can write
				\beqq
				\i (u \cdot \nabla) Z^{\ah}\wu \cdot Z^{\ah} \wu dx=
				-\frac12\i |Z^{\ah} \wu|^2 {\rm div}u dx=0.
				\deqq
				If $\beta_h=\alpha_h$, we use the anisotropic type inequality \eqref{ie:Sobolev}
				to obtain
				\beqq
				\begin{aligned}
					&\i Z^{\ah}u \cdot \nabla \wu \cdot Z^{\ah} \wu dx\\
					\lesssim
					&\|Z^{\ah}u\|_{L^2}^{\frac12}\|\p_3 Z^{\ah}u\|_{L^2}^{\frac12}
					\|\nabla \wu\|_{L^2}^{\frac12}\|\p_2 \nabla \wu\|_{L^2}^{\frac12}
					\|Z^{\ah} \wu\|_{L^2}^{\frac12}\|\p_1 Z^{\ah} \wu\|_{L^2}^{\frac12}\\
					\lesssim
					&\sqrt{\me^m(t)}\md_{tan}^{k+1}(t).
				\end{aligned}
				\deqq
				If $0<\beta_h<\alpha_h$,  we apply the anisotropic type inequality \eqref{ie:Sobolev}
				to deduce
				\beqq
				\begin{aligned}
					&\i (Z^{\beta_h}u \cdot \nabla Z^{\ah-\beta_h}\wu)\cdot Z^{\ah} \wu dx\\
					\lesssim
					&\|Z^{\beta_h}u\|_{L^2}^{\frac{1}{2}}
					\|\p_3 Z^{\beta_h}u\|_{L^2}^{\frac{1}{2}}
					\|\nabla Z^{\ah-\beta_h}\wu\|_{L^2}^{\frac{1}{2}}
					\|\p_2 \nabla Z^{\ah-\beta_h}\wu\|_{L^2}^{\frac{1}{2}}
					\|Z^{\ah} \wu\|_{L^2}^{\frac{1}{2}}
					\|\p_1 Z^{\ah} \wu\|_{L^2}^{\frac{1}{2}}\\
					\lesssim
					&\sqrt{\me^m(t)}\md_{tan}^{k+1}(t).
				\end{aligned}
				\deqq
				Thus, the combination of the above estimates yields directly
				\beqq
				I_1 \lesssim \sqrt{\me^m(t)}\md_{tan}^{k+1}(t).
				\deqq
				Similarly, we can conclude
				\beqq
				I_2, I_4, I_5\lesssim \sqrt{\me^m(t)}\md_{tan}^{k+1}(t).
				\deqq
				Integrating by part and using the divergence-free condition
				$\eqref{eqr}_3$, we conclude
				\beqq
				\begin{aligned}
					&\i (b\cdot \nabla) Z^{\ah}\wb \cdot Z^{\ah} \wu dx
					+\i (b\cdot \nabla) Z^{\ah} \wu \cdot  Z^{\ah}\wb dx\\
					=&-\i (Z^{\ah}\wb \cdot Z^{\ah} \wu)\nabla \cdot b \ dx=0,
				\end{aligned}
				\deqq
				and hence, we may write
				\beqq
				\begin{aligned}
					I_3+I_7
					=&\sum_{0< \beta_h \le \alpha_h}C^{\beta_h}_{\alpha_h}
					\i (Z^{\beta_h}b \cdot \nabla Z^{\ah-\beta_h}\wb)\cdot Z^{\ah} \wu dx\\
					&+\sum_{0< \beta_h \le \alpha_h}C^{\beta_h}_{\alpha_h}
					\i (Z^{\beta_h}b \cdot \nabla Z^{\ah-\beta_h}\wu)\cdot Z^{\ah} \wb dx\\
					\lesssim
					&\sum_{0< \beta_h \le \alpha_h}\|Z^{\beta_h}b\|_{L^2}^{\frac14}
					\|\p_3 Z^{\beta_h}b\|_{L^2}^{\frac14}
					\|\p_2 Z^{\beta_h}b\|_{L^2}^{\frac14}
					\|\p_{23} Z^{\beta_h}b\|_{L^2}^{\frac14}
					\|\nabla Z^{\ah-\beta_h}\wb\|_{L^2}
					\|Z^{\ah} \wu\|_{L^2}^{\frac12}
					\|\p_1 Z^{\ah} \wu\|_{L^2}^{\frac12}\\
					&+\sum_{0< \beta_h \le \alpha_h}\|Z^{\beta_h}b\|_{L^2}^{\frac14}
					\|\p_3 Z^{\beta_h}b\|_{L^2}^{\frac14}
					\|\p_2 Z^{\beta_h}b\|_{L^2}^{\frac14}
					\|\p_{23} Z^{\beta_h}b\|_{L^2}^{\frac14}
					\|\nabla Z^{\ah-\beta_h}\wu\|_{L^2}
					\|Z^{\ah} \wb\|_{L^2}^{\frac12}
					\|\p_1 Z^{\ah} \wb\|_{L^2}^{\frac12}\\
					\lesssim
					&\sqrt{\me^m(t)}\md_{tan}^{k+1}(t).
				\end{aligned}
				\deqq
				Similarly, the terms $I_7$ and $I_8$ can be bounded by
				\beqq
				\begin{aligned}
					I_6, I_8\lesssim
					&\sum_{0\le \beta_h < \alpha_h}C^{\beta_h}_{\alpha_h}
					\|Z^{\beta_h} \nabla u\|_{L^2}^{\frac12}
					\|\p_3 Z^{\beta_h} \nabla u\|_{L^2}^{\frac12}
					\|\nabla Z^{\ah-\beta_h} b\|_{L^2}^{\frac12}
					\|\p_1 \nabla Z^{\ah-\beta_h} b\|_{L^2}^{\frac12}
					\|Z^{\ah} \wb\|_{L^2}^{\frac12}
					\|\p_2 Z^{\ah} \wb\|_{L^2}^{\frac12}\\
					&\quad+\|Z^{\ah} \nabla u\|_{L^2}^{\frac12}
					\|\p_1 Z^{\ah} \nabla u\|_{L^2}^{\frac12}
					\|\nabla b\|_{L^2}^{\frac12}
					\|\p_3 \nabla  b\|_{L^2}^{\frac12}
					\|Z^{\ah} \wb\|_{L^2}^{\frac12}
					\|\p_2 Z^{\ah} \wb\|_{L^2}^{\frac12}\\
					\lesssim
					&\sqrt{\me^m(t)}\md_{tan}^{k+1}(t).
				\end{aligned}
				\deqq
				Thus, collecting all estimates for $I_1$ through $I_8$
				into \eqref{3202}, we conclude
				\beqq
				\frac{d}{dt}\|(\wu, \wb)\|_{H^{k}_{tan}}^2
				+\|(\p_1 \wu, \nabla_h \wb)\|_{H^{k}_{tan}}^2
				+\ep\|(\p_2 \wu, \p_3 \wu, \p_3 \wb)\|_{H^{k}_{tan}}^2
				\lesssim
				\sqrt{\me^m(t)}\md_{tan}^{k+1}(t).
				\deqq
				Therefore, we complete the proof of this lemma.
			\end{proof}

			\subsection{Second order normal derivative estimate} \label{subsec-priori-half}
			In this subsection, we will establish the estimate for
			the second order normal derivative of velocity and magnetic field.
			Indeed, due to the boundary condition $\eqref{eqr}_4$
			and \eqref{bd-wuwb}, it is easy to check that
			\beqq
			(\nabla(u\cdot \nabla)\times b)_h|_{x_3=0}
			=(\nabla(b\cdot \nabla)\times u)_h|_{x_3=0}=0,
			\deqq
			and hence, we can obtain
			\beq\label{bd01}
			(-u\cdot \nabla \wu+\wu\cdot \nabla u+b \cdot \nabla \wb-\wb\cdot \nabla b)_h|_{x_3=0}=0,
			\deq
			and
			\beq\label{bd02}
			(-u\cdot \nabla \wb-\nabla(u\cdot \nabla)\times b
			+b\cdot \nabla \wu+\nabla(b\cdot \nabla)\times u)_h|_{x_3=0}=0.
			\deq	
			Now, we establish the estimate for the second order
			normal derivative of velocity and magnetic field.
			\begin{lemm}
				For any smooth solution $(u, b)$ of equation \eqref{eqr},
				it holds
				\beq\label{3301}
				\frac{d}{dt}\|(\p_3 \wu, \p_3 \wb)\|_{L^2}^2
				+\|(\p_{13} \wu, \nabla_h \p_3 \wb)\|_{L^2}^2
				+\ep \|(\p_{23} \wu, \p_{33} \wu, \p_{33} \wb)\|_{L^2}^2
				\lesssim \sqrt{\me^m(t)}\md^{m}(t).
				\deq
			\end{lemm}
			\begin{proof}
				Due to the boundary conditions \eqref{bd01} and \eqref{bd02}, we conclude
				\beq\label{3302}
				\begin{aligned}
					&\frac{d}{dt}\frac{1}{2}\i (|\p_3 \wu|^2+|\p_3 \wb|^2) dx
					+\i (|\p_{13} \wu|^2+|\nabla_h \p_3 \wb|^2) dx\\
					&+\ep \i(|\p_{23} \wu|^2+|\p_{33} \wu|^2+|\p_{33} \wb|^2)dx\\
					=&\i \p_3(-u\cdot \nabla \wu+\wu\cdot \nabla u+b \cdot \nabla \wb-\wb\cdot \nabla b)
					\cdot \p_3 \wu dx\\
					&+\i \p_3 (-u\cdot \nabla \wb-\nabla(u\cdot \nabla)\times b
					+b\cdot \nabla \wu+\nabla(b\cdot \nabla)\times u)
					\cdot \p_3 \wb dx.
				\end{aligned}
				\deq
				Integrating by part and using the anisotropic
				type inequality \eqref{ie:Sobolev}, we can write
				\beq\label{3303}
				\begin{aligned}
					&\i \p_3(-u\cdot \nabla \wu) \cdot \p_3 \wu dx\\
					=&-\i (\p_3 u_h \cdot \nabla_h \wu+\p_3 u_3 \p_3 \wu)\cdot \p_3 \wu dx
					+\frac12\i |\p_3 \wu|^2 {\rm div}u \ dx\\
					\lesssim
					&\|\p_3 u_h\|_{L^2}^{\frac12}\|\p_{23} u_h\|_{L^2}^{\frac12}
					\|\nabla_h \wu\|_{L^2}^{\frac12}\|\p_3 \nabla_h \wu\|_{L^2}^{\frac12}
					\|\p_3 \wu\|_{L^2}^{\frac12}\|\p_{13} \wu\|_{L^2}^{\frac12}\\
					&+\|\p_3 u_3\|_{L^2}^{\frac12}\|\p_{33} u_3\|_{L^2}^{\frac12}
					\|\p_3 \wu\|_{L^2}^{\frac12}\|\p_{23} \wu\|_{L^2}^{\frac12}
					\|\p_3 \wu\|_{L^2}^{\frac12}\|\p_{13} \wu\|_{L^2}^{\frac12}\\
					\lesssim
					&\sqrt{\me^m(t)}\md^{m}(t).
				\end{aligned}
				\deq
				Similarly, it is easy to check that
				\beq\label{3304}
				\i \p_3(\wu\cdot \nabla u-\wb\cdot \nabla b)\cdot \p_3 \wu dx
				+\i \p_3 (-u\cdot \nabla \wb)\cdot \p_3 \wb dx
				\lesssim \sqrt{\me^m(t)}\md^{m}(t).
				\deq
				Integrating by part and using the anisotropic
				type inequality \eqref{ie:Sobolev}, we conclude
				\beq\label{3305}
				\begin{aligned}
					&\i \p_3(b \cdot \nabla \wb)\cdot \p_3 \wu dx
					+\i \p_3 (b\cdot \nabla \wu)\cdot \p_3 \wb dx\\
					=
					&\i (\p_3 b_h \cdot \nabla_h \wb+\p_3 b_3 \p_3 \wb)\p_3 \wu dx
					+\i (\p_3 b_h \cdot \nabla_h \wu+\p_3 b_3 \p_3 \wu)\p_3 \wb dx\\
					&+\i (b \cdot \nabla )(\p_3 \wb \cdot \p_3 \wu) dx +\i (b \cdot \nabla )(\p_3 \wu \cdot \p_3 \wb) dx\\
					\lesssim
					&\|\p_3 b_h\|_{L^2}^{\frac12}\|\p_{23} b_h\|_{L^2}^{\frac12}
					\|\nabla_h(\wb, \wu)\|_{L^2}^{\frac12}\|\p_{3} \nabla_h(\wb, \wu)\|_{L^2}^{\frac12}
					\|\p_3 (\wu, \wb)\|_{L^2}^{\frac12}\|\p_{13} (\wu, \wb)\|_{L^2}^{\frac12}\\
					&+\|\p_3 b_3\|_{L^2}^{\frac12}\|\p_{33} b_3\|_{L^2}^{\frac12}
					\|\p_{3} (\wb, \wu)\|_{L^2}
					\|\p_{23} \wb\|_{L^2}^{\frac12}
					\|\p_{13} \wu\|_{L^2}^{\frac12}\\
					\lesssim
					&\sqrt{\me^m(t)}\md^{m}(t).
				\end{aligned}
				\deq
				Direct computation gives
				\beq\label{3306}
				\begin{aligned}
					\nabla(u\cdot \nabla)\times b
					&=\nabla u_1 \times \p_1 b
					+\nabla u_2 \times \p_2 b
					+\nabla u_3 \times \p_3 b,\\
					\nabla(b\cdot \nabla)\times u
					&=\nabla b_1 \times \p_1 u
					+\nabla b_2 \times \p_2 u
					+\nabla b_3 \times \p_3 u,
				\end{aligned}
				\deq
				then we apply conditions $\nabla \cdot u=0$ and $\nabla \cdot b=0$ to get
				\beq\label{3307}
				\begin{aligned}
					\nabla u_3 \times \p_3 b
					&=(\p_2 u_3\p_3 b_3-\p_3 u_3\p_3 b_2,
					\p_3 u_3\p_3 b_1-\p_1 u_3\p_3 b_3,
					\p_1 u_3\p_3 b_2-\p_2 u_3\p_3 b_1)^T\\
					&=(-\p_2 u_3 \nabla_h\cdot b_h+\nabla_h \cdot u_h\p_3 b_2,
					-\nabla_h\cdot u_h\p_3 b_1+\p_1 u_3 \nabla_h\cdot b_h,
					\p_1 u_3\p_3 b_2-\p_2 u_3\p_3 b_1)^T,
				\end{aligned}
				\deq
				and
				\beq\label{3308}
				\begin{aligned}
					\nabla b_3 \times \p_3 u
					&=(\p_2 b_3\p_3 u_3-\p_3 b_3\p_3 u_2,
					\p_3 b_3\p_3 u_1-\p_1 b_3\p_3 u_3,
					\p_1 b_3\p_3 u_2-\p_2 b_3\p_3 u_1)^T\\
					&=(-\p_2 b_3 \nabla_h\cdot u_h+\nabla_h \cdot b_h\p_3 u_2,
					-\nabla_h\cdot b_h\p_3 u_1+\p_1 b_3 \nabla_h\cdot u_h,
					\p_1 b_3\p_3 u_2-\p_2 b_3\p_3 u_1)^T.
				\end{aligned}
				\deq
				Thus, the combination of relations \eqref{3306}-\eqref{3308}
				and anisotropic type inequality \eqref{ie:Sobolev} yields
				\beq\label{3309}
				\i \p_3 (-\nabla(u\cdot \nabla)\times b+\nabla(b\cdot \nabla)\times u)
				\cdot \p_3 \wb dx\lesssim \sqrt{\me^m(t)}\md^{m}(t).
				\deq
				Substituting the estimates \eqref{3303},
				\eqref{3304}, \eqref{3305} and \eqref{3309}
				into \eqref{3302},  we complete the proof of this lemma.
			\end{proof}
			
			Next, we will establish the estimate for the second order
			normal derivative of velocity and magnetic field in higher
			conormal Sobolev space.
			
			\begin{lemm}
				For any smooth solution $(u, b)$ of equation \eqref{eqr},
				it holds for any $|\alpha|\le m-2$
				\beq\label{3401}
				\begin{aligned}
					&\|\p_3 (\wu, \wb)(t)\|_{H^{m-2}_{co}}^2
					+\int_0^t \|(\p_{13} \wu, \nabla_h \p_3 \wb)\|_{H^{m-2}_{co}}^2 d \tau
					+\ep \int_0^t \|(\p_{23}  \wu, \p_{33} \wu, \p_{33} \wb)\|_{H^{m-2}_{co}}^2d \tau \\
					\lesssim
					&\|\p_3 (\wu_0, \wb_0)\|_{H^{m-2}_{co}}^2
					+\int_0^t \sqrt{\me^m(\tau)}\md^{m}(\tau) d\tau
					+\int_0^t \|(\nabla_h u_h, \nabla_h \p_3 u_h)\|_{H^1_{tan}}
					\|\p_3 \wu\|_{H^{m-2}_{co}}^{\frac32}\|\p_{13} Z^\al \wu\|_{L^2}^{\frac12} d\tau.
				\end{aligned}
				\deq
				Furthermore, under the assumption \eqref{assumption}, it holds
				\beq\label{3423}
				\begin{aligned}
					&\|\p_3 (\wu, \wb)(t)\|_{H^{m-2}_{co}}^2
					+\int_0^t \|(\p_{13} \wu, \nabla_h \p_3 \wb)\|_{H^{m-2}_{co}}^2 d \tau
					+\ep \int_0^t \|(\p_{23}  \wu, \p_{33} \wu, \p_{33} \wb)\|_{H^{m-2}_{co}}^2d \tau \\
					\lesssim
					&\|\p_3 (\wu_0, \wb_0)\|_{H^{m-2}_{co}}^2
					+\delta^{\frac32}
					+\int_0^t \sqrt{\me^m(\tau)}\md^{m}(\tau) d\tau.
				\end{aligned}
				\deq
			\end{lemm}
			\begin{proof}
				For any $|\alpha|\le m-2$, the equations \eqref{eqwu}
				and \eqref{eqwb} yield directly
				\beq\label{3402}
				\begin{aligned}
					&\frac{d}{dt}\frac12 \i (|\p_3 Z^\al \wu|^2+|\p_3 Z^\al \wb|^2)dx
					+\i(|\p_{13} Z^\al \wu|^2+|\nabla_h \p_3 Z^\al \wb|^2)dx\\
					&+\ep \i(|\p_{23} Z^\al \wu|^2+|\p_{33} Z^\al \wu|^2+|\p_{33} Z^\al \wb|^2)dx\\
					=&\i \p_3Z^\al(-u\cdot \nabla \wu)\cdot \p_3 Z^\al \wu dx
					+\i \p_3Z^\al(\wu\cdot \nabla u)\cdot \p_3 Z^\al \wu dx\\
					&+\i \p_3Z^\al(b \cdot \nabla \wb)\cdot \p_3 Z^\al \wu dx
					+\i \p_3Z^\al(-\wb\cdot \nabla b)\cdot \p_3 Z^\al \wu dx\\
					&+\i \p_3 Z^\al(-u\cdot \nabla \wb)\cdot \p_3 Z^\al \wb dx
					+\i \p_3 Z^\al(-\nabla(u\cdot \nabla)\times b)\cdot \p_3 Z^\al \wb dx\\
					&+\i \p_3 Z^\al(b\cdot \nabla \wu)\cdot \p_3 Z^\al \wb dx
					+\i \p_3 Z^\al(\nabla(b\cdot \nabla)\times u)\cdot \p_3 Z^\al \wb dx\\
					&+(-\ep) \i Z^\al \p_{33} \wu \cdot \p_{33} Z^\al \wu dx
					+(-\ep) \i Z^\al \p_{33} \wb \cdot \p_{33} Z^\al \wb dx
					:=\sum_{i=1}^{10}II_i,
				\end{aligned}
				\deq
				where we have used the basic fact that
				\beqq
				\begin{aligned}
					&\i Z^\al \p_2 \wb \cdot \p_{33} Z^\al \wu dx
					+\i Z^\al \p_2 \wu \cdot \p_{33} Z^\al \wb dx \\
					=&-\i \p_{23} Z^\al  \wb \cdot \p_3  Z^\al \wu dx
					-\i \p_{23} Z^\al  \wu \cdot \p_3 Z^\al \wb dx=0.
				\end{aligned}
				\deqq
				Then, we can establish the following estimates
				\beq\label{claim-estimate}
				\begin{aligned}
					&II_1, II_2
					\lesssim
					\sqrt{\me^m(t)}\md^{m}(t)
					+\|(\nabla_h u_h, \nabla_h \p_3 u_h)\|_{H^1_{tan}}
					\|\p_3 \wu\|_{H^{m-2}_{co}}^{\frac32}\|\p_{13} Z^\al \wu\|_{L^2}^{\frac12},\\
					&II_3+II_6, II_4, II_5, II_7, II_8 \lesssim \sqrt{\me^m(t)}\md^{m}(t),\\
					&II_9 +\ep \i |\p_{33} Z^\al \wu|^2 dx \lesssim
					\ep (\|\p_3 \wu\|_{H^{|\alpha|-1}_{co}}^2+\|\p_{33} \wu\|_{H^{|\alpha|-1}_{co}}^2),\\
					&II_{10} +\ep \i |\p_{33} Z^\al \wb|^2 dx \lesssim
					\ep (\|\p_3 \wb\|_{H^{|\alpha|-1}_{co}}^2+\|\p_{33} \wb\|_{H^{|\alpha|-1}_{co}}^2).
				\end{aligned}
				\deq
				The proof of estimate \eqref{claim-estimate} is complicated and long,
				the proof in detail can be found in Appendix \ref{claim-estimates}.
				Substituting \eqref{claim-estimate} into \eqref{3402}
				and using \eqref{3301} after integrating over $[0, t]$,
				we can apply the induction to obtain
				\beq\label{3424}
				\begin{aligned}
					&\|\p_3 (\wu, \wb)(t)\|_{H^{m-2}_{co}}^2
					+\int_0^t \|(\p_{13} \wu, \nabla_h \p_3 \wb)\|_{H^{m-2}_{co}}^2 d \tau
					+\ep \int_0^t \|(\p_{23}  \wu, \p_{33} \wu, \p_{33} \wb)\|_{H^{m-2}_{co}}^2d \tau \\
					\lesssim
					&\|\p_3 (\wu_0, \wb_0)\|_{H^{m-2}_{co}}^2
					+\!\!\int_0^t \!\sqrt{\me^m(\tau)}\md^{m}(\tau) d\tau
					+\!\!\int_0^t \!\|(\nabla_h u_h, \nabla_h \p_3 u_h)\|_{H^1_{tan}}
					\|\p_3 \wu\|_{H^{m-2}_{co}}^{\frac32}\|\p_{13} Z^\al \wu\|_{L^2}^{\frac12} d\tau.
				\end{aligned}
				\deq
				Finally, under the assumption estimate \eqref{assumption}, we can obtain
				\beq\label{3425}
				\begin{aligned}
					&\int_0^t  \|(\nabla_h u_h, \nabla_h \p_3 u_h)\|_{H^1_{tan}}
					\|\p_3 \wu\|_{H^{m-2}_{co}}^{\frac32}\|\p_{13} Z^\al \wu\|_{L^2}^{\frac12} d\tau\\
					\lesssim
					&(\underset{0\le \tau \le t}{\sup}\|\p_3 \wu\|_{H^{m-2}_{co}}^2)^{\frac34}
					\left\{\int_0^t \|(\nabla_h u_h, \nabla_h \p_3 u_h)\|_{H^1_{tan}}^2
					(1+\tau)^{\sigma}d\tau\right\}^\frac{1}{2}\\
					&\times \left\{\int_0^t \|\p_{13} Z^\al \wu\|_{L^2}^2 d\tau\right\}^{\frac{1}{4}}
					\left\{\int_0^t (1+\tau)^{-2\sigma}d\tau\right\}^{\frac{1}{4}}
					\lesssim
					\delta^{\frac{3}{2}}.
				\end{aligned}
				\deq
				The combination of estimates \eqref{3424} and \eqref{3425} implies
				the estimates \eqref{3401} and \eqref{3423}.
				Therefore, we complete the proof of this lemma.
			\end{proof}
			
			\subsection{Enhanced dissipative estimate of velocity field}
			
			In this subsection, we will establish the dissipative
			estimate of velocity field in the $x_2$ direction.
			This enhanced dissipative estimate arises from the
			good stability effect of magnetic field near a
			background magnetic field.
			\begin{lemm}
				For any smooth solution $(u, b)$ of equation \eqref{eqr},
				it holds for all $|\ah|=k \le m-1$
				\beq\label{3501}
				\begin{aligned}
					&
					\frac{d}{dt}\i \p_2 Z^{\ah}b \cdot  Z^{\ah}u \ dx
					+\frac12\i |\p_2 Z^{\ah} u|^2dx\\
					\lesssim
					&\|\nabla_h b\|_{H^{k+1}_{tan}}^2+\|\p_1 u\|_{H^{k}_{tan}}^2
					+\ep \|(\p_3 Z^{\ah} b,\p_{23} Z^{\ah} u)\|_{L^2}^2
					+\sqrt{\me^m(t)}\md_{tan}^{k+1}(t).
				\end{aligned}
				\deq
			\end{lemm}
			\begin{proof}
				For any $|\ah|=k\le m-1$, then the equation $\eqref{eqr}_2$ yields directly
				\beq\label{3502}
				\begin{aligned}
					\i |\p_2 Z^{\ah} u|^2dx
					=&\i \p_t Z^{\ah} b \cdot \p_2 Z^{\ah} u\  dx
					-\i (\Delta_h Z^{\ah} b +\ep \p_{33} Z^{\ah} b)\cdot \p_2 Z^{\ah} u\ dx\\
					&+\i Z^{\ah}(u\cdot \nabla b-b\cdot \nabla u)\cdot \p_2 Z^{\ah} u\ dx.
				\end{aligned}
				\deq
				On the other hand, using the velocity equation $\eqref{eqr}_1$, we have
				\beq\label{3503}
				\begin{aligned}
					\i \p_t Z^{\ah}b \cdot \p_2 Z^{\ah}u \ dx
					=&\frac{d}{dt}\i Z^{\ah}b \cdot \p_2 Z^{\ah}u \ dx
					-\i Z^{\ah}b \cdot \p_2 Z^{\ah} \p_t u \ dx\\
					=&\frac{d}{dt}\i Z^{\ah}b \cdot \p_2 Z^{\ah}u \ dx
					+\i |\p_2 Z^{\ah} b|^2 dx\\
					&-\i Z^{\ah}(u\cdot \nabla u -\p_1^2 u-\ep \p_2^2 u-\ep \p_3^2 u
					+\nabla p-b\cdot \nabla b)\p_2 Z^{\ah}b \ dx.
				\end{aligned}
				\deq
				The combination of equations \eqref{3502} and \eqref{3503}  yields directly
				\beq\label{3504}
				\begin{aligned}
					\i |\p_2 Z^{\ah} u|^2dx
					=&\frac{d}{dt}\i Z^{\ah}b \cdot \p_2 Z^{\ah}u \ dx
					+\i |\p_2 Z^{\ah} b|^2 dx\\
					&-\i (\Delta_h Z^{\ah} b +\ep \p_{33} Z^{\ah} b)\cdot \p_2 Z^{\ah} u\ dx
					+\i Z^{\ah}(u\cdot \nabla b-b\cdot \nabla u)\cdot \p_2 Z^{\ah} u\ dx\\
					&-\i Z^{\ah} (u\cdot \nabla u -\p_1^2 u-\ep \p_2^2 u-\ep \p_3^2 u
					+\nabla p-b\cdot \nabla b)\p_2 Z^{\ah}b \ dx.
				\end{aligned}
				\deq
				Using Cauchy inequality and condition $\nabla \cdot b=0$, we have
				\beq\label{3505}
				\begin{aligned}
					&|\i (\Delta_h Z^{\ah} b +\ep \p_3^2 Z^{\ah} b)\cdot \p_2 Z^{\ah} u\ dx|\\
					\le &\frac{1}{4}\|\p_2 Z^{\ah} u\|_{L^2}^2
					+\|\Delta_h Z^{\ah} b\|_{L^2}^2
					+\ep \|(\p_3 Z^{\ah} b, \p_3 \p_2 Z^{\ah} u)\|_{L^2}^2,
				\end{aligned}
				\deq
				and
				\beq\label{3506}
				\begin{aligned}
					&|\i Z^{\ah} (\p_1^2 u+\ep \p_2^2 u+\ep \p_3^2 u-\nabla p)\p_2 Z^{\ah}b \ dx|\\
					\le &\|\p_1 Z^{\ah} u\|_{L^2}\|\p_{12} Z^{\ah}b\|_{L^2}
					+\ep\|\p_2 Z^{\ah} u\|_{L^2}\|\p_{22} Z^{\ah}b\|_{L^2}\\
					&+\ep\|(\p_{23} Z^{\ah} u, \p_3 Z^{\ah}b)\|_{L^2}^2.
				\end{aligned}
				\deq
				Using the anisotropic type inequality \eqref{ie:Sobolev}, it holds
				\beq\label{3507}
				\begin{aligned}
					&-\i Z^{\ah}(b\cdot \nabla u)\cdot \p_2 Z^{\ah} u\ dx\\
					=&-\sum_{0\le \beta_h \le \alpha_h}C^{\beta_h}_{\alpha_h}
					\i Z^{\beta_h} b \cdot Z^{\alpha_h-\beta_h}\nabla u
					\cdot \p_2 Z^{\ah}u \ dx\\
					\lesssim
					&\|Z^{\beta_h} b\|_{L^2}^{\frac14}
					\|\p_3 Z^{\beta_h} b\|_{L^2}^{\frac14}
					\|\p_{2} Z^{\beta_h} b\|_{L^2}^{\frac14}
					\|\p_{23} Z^{\beta_h} b\|_{L^2}^{\frac14}
					\|Z^{\alpha_h-\beta_h}\nabla u\|_{L^2}^{\frac12}\\
					&\times\|\p_1 Z^{\alpha_h-\beta_h}\nabla u\|_{L^2}^{\frac12}
					\|\p_2 Z^{\ah}u\|_{L^2}\\
					\lesssim
					&\sqrt{\me^m(t)}\md_{tan}^{k+1}(t).
				\end{aligned}
				\deq
				Similarly, it is easy to check that
				\beq\label{3508}
				\begin{aligned}
					|\i Z^{\ah}(u\cdot \nabla b)\cdot \p_2 Z^{\ah} u\ dx|,
					|\i Z^{\ah} (u\cdot \nabla u -b\cdot \nabla b)\p_2 Z^{\ah}b \ dx|
					\lesssim \sqrt{\me^m(t)}\md_{tan}^{k+1}(t).
				\end{aligned}
				\deq
				Thus, substituting the estimates \eqref{3505},\eqref{3506},
				\eqref{3507} and \eqref{3508} into \eqref{3504}, we can write
				\beqq
				\begin{aligned}
					&\frac{d}{dt}\i \p_2 Z^{\ah}b \cdot  Z^{\ah}u \ dx
					+\frac12\i |\p_2 Z^{\ah} u|^2dx\\
					\lesssim
					&\|\nabla_h b\|_{H^{k+1}_{tan}}^2+\|\p_1 u\|_{H^{k}_{tan}}^2
					+\ep \|(\p_3 Z^{\ah} b,\p_{23} Z^{\ah} u)\|_{L^2}^2
					+\sqrt{\me^m(t)}\md_{tan}^{k+1}(t).
				\end{aligned}
				\deqq
				Therefore, we complete the proof of this lemma.
			\end{proof}
			
			Next, we establish the dissipative estimate for normal derivative of
			velocity in the $x_2$ direction as follows.
			\begin{lemm}
				For any smooth solution $(u, b)$ of equation \eqref{eqr},
				it holds for all $|\ah|=k\le m-3$
				\beq\label{3601}
				\begin{aligned}
					&\frac{d}{dt}\i \p_2 Z^{\ah} \wb \cdot Z^{\ah} \wu dx
					+\|\p_2 Z^{\ah} \wu\|_{L^2}^2 \\
					\lesssim &\|(\nabla_h \wb, \nabla_h^2 \wb,\p_{11} \wu)\|_{H^{k}_{tan}}^2
					+\ep \|(\p_{22} \wu,\p_{23} \wu,  \p_3  \wb)\|_{H^{k}_{tan}}^2
					+\sqrt{\me^m(t)}\md_{tan}^{k+2}(t),
				\end{aligned}
				\deq
				and for $|\ah|=m-2$, it holds
				\beq\label{3613}
				\begin{aligned}
					&\frac{d}{dt}\i \p_2 Z^{\ah} \wb \cdot Z^{\ah} \wu dx
					+\|\p_2 Z^{\ah} \wu\|_{L^2}^2\\
					\lesssim &\|(\nabla_h \wb, \nabla_h^2 \wb,\p_{11} \wu)\|_{H^{m-2}_{tan}}^2
					+\ep \|(\p_{22} \wu,\p_{23} \wu,  \p_3  \wb)\|_{H^{m-2}_{tan}}^2
					+\sqrt{\me^m(t)}\md^{m}(t).
				\end{aligned}
				\deq
			\end{lemm}
			\begin{proof}
				For any $|\alpha_h|=k\le m-2$, the equation \eqref{eqwb} yields directly
				\beq\label{3602}
				\begin{aligned}
					\i |\p_2 Z^{\ah} \wu|^2 dx
					=&\i Z^{\ah} \p_t \wb \cdot \p_2 Z^{\ah} \wu dx
					-\i (\Delta_h Z^{\ah} \wb+\ep \p_{33} Z^{\ah} \wb)\cdot \p_2 Z^{\ah} \wu dx\\
					&+\i Z^{\ah} (u\cdot \nabla \wb+\nabla(u\cdot \nabla)\times b
					-b\cdot \nabla \wu-\nabla(b\cdot \nabla)\times u)\cdot \p_2 Z^{\ah} \wu dx.
				\end{aligned}
				\deq
				On the other hand, using the vorticity equation \eqref{eqwu}, we conclude
				\beq\label{3603}
				\begin{aligned}
					&\i Z^{\ah} \p_t \wb \cdot \p_2 Z^{\ah} \wu dx\\
					=&-\frac{d}{dt}\i \p_2 Z^{\ah} \wb \cdot Z^{\ah} \wu dx
					+\i \p_2  Z^{\ah} \wb \cdot Z^{\ah} \p_t \wu dx\\
					=&-\frac{d}{dt}\i \p_2 Z^{\ah} \wb \cdot Z^{\ah} \wu dx
					+\i \p_2  Z^{\ah} \wb \cdot (\p_{11} Z^{\ah}\wu
					+\ep \p_{22} Z^{\ah}\wu
					+\ep \p_{33} Z^{\ah} \wu) dx\\
					&+\i \p_2  Z^{\ah} \wb \cdot
					Z^{\ah}(-u\cdot \nabla \wu+\wu\cdot \nabla u+b \cdot \nabla \wb-\wb\cdot \nabla b) dx
					+\i |\p_2  Z^{\ah} \wb|^2 dx.
				\end{aligned}
				\deq
				Thus, the combination of \eqref{3602} and \eqref{3603} yields directly
				\beq\label{3604}
				\begin{aligned}
					&\frac{d}{dt}\i \p_2 Z^{\ah} \wb \cdot Z^{\ah} \wu dx
					+\i |\p_2 Z^{\ah} \wu|^2 dx\\
					=&\i |\p_2  Z^{\ah} \wb|^2 dx
					+\i (-\Delta_h Z^{\ah} \wb-\ep \p_3^2 Z^{\ah} \wb)\cdot \p_2 Z^{\ah} \wu dx\\
					&+\i (\p_1^2 Z^{\ah}\wu+\ep \p_2^2 Z^{\ah}\wu
					+\ep \p_3^2 Z^{\ah} \wu)\cdot  \p_2  Z^{\ah} \wb dx\\
					&+\i Z^{\ah}(-u\cdot \nabla \wu)\cdot \p_2  Z^{\ah} \wb dx
					+\i Z^{\ah}(\wu\cdot \nabla u)\cdot \p_2  Z^{\ah} \wb dx\\
					&+\i Z^{\ah}(b \cdot \nabla \wb)\cdot \p_2  Z^{\ah} \wb dx
					+\i Z^{\ah}(-\wb\cdot \nabla b)\cdot \p_2  Z^{\ah} \wb dx\\
					&+\i Z^{\ah} (u\cdot \nabla \wb)\cdot \p_2 Z^{\ah} \wu dx
					+\i Z^{\ah} (\nabla(u\cdot \nabla)\times b)\cdot \p_2 Z^{\ah} \wu dx\\
					&+\i Z^{\ah} (-b\cdot \nabla \wu)\cdot \p_2 Z^{\ah} \wu dx
					+\i Z^{\ah} (-\nabla(b\cdot \nabla)\times u)\cdot \p_2 Z^{\ah} \wu dx\\
					:=&\i |\p_2  Z^{\ah} \wb|^2 dx
					+\sum_{i=1}^{10} III_i.
				\end{aligned}
				\deq
				Integrating by part and using Cauchy inequality, we have
				\beq\label{3605}
				III_1
				\lesssim \frac14\|\p_2 Z^{\ah} \wu\|_{L^2}^2+\|\Delta_h Z^{\ah} \wb\|_{L^2}^2
				+\ep \|(\p_{23} Z^{\ah} \wu, \p_3 Z^{\ah} \wb)\|_{L^2}^2,
				\deq
				and
				\beq\label{3606}
				III_2
				\lesssim \|(\p_{11} Z^{\ah}\wu, \p_2  Z^{\ah} \wb)\|_{L^2}^2
				+\ep \|(\p_{22} Z^{\ah}\wu,
				\p_{23} Z^{\ah} \wu,  \p_3  Z^{\ah} \wb)\|_{L^2}^2.
				\deq
				Now let us deal with the term $III_3$. It is to check that
				\beqq
				\begin{aligned}
					III_3=&-\sum_{0\le \beta_h \le \alpha_h}C^{\beta_h}_{\alpha_h}
					\i Z^{\beta_h} u_h\cdot Z^{\ah-\beta_h} \nabla_h \wu  \cdot \p_2  Z^{\ah} \wb dx\\
					&-\sum_{0< \beta_h \le \alpha_h}C^{\beta_h}_{\alpha_h}
					\i Z^{\beta_h} u_3 Z^{\ah-\beta_h} \p_3 \wu  \cdot \p_2  Z^{\ah} \wb dx
					-\i  u_3 \p_3 Z^{\ah} \wu  \cdot \p_2  Z^{\ah} \wb dx\\
					:=&III_{31}+III_{32}+III_{33}.
				\end{aligned}
				\deqq
				Using the anisotropic type inequality \eqref{ie:Sobolev}, we obtain
				\beqq
				\begin{aligned}
					III_{31}
					&\lesssim \sum_{0\le \beta_h \le \alpha_h}
					\|Z^{\beta_h} u_h\|_{L^2}^{\frac12}
					\|\p_3 Z^{\beta_h} u_h\|_{L^2}^{\frac12}
					\|Z^{\ah-\beta_h} \nabla_h \wu\|_{L^2}^{\frac12}
					\|\p_1 Z^{\ah-\beta_h} \nabla_h \wu\|_{L^2}^{\frac12}\\
					&\quad\quad\quad \quad\quad\times \|\p_2  Z^{\ah} \wb\|_{L^2}^{\frac12}
					\|\p_{22}  Z^{\ah} \wb\|_{L^2}^{\frac12}\\
					&\lesssim \sqrt{\me^m(t)}\md_{tan}^{k+2}(t),
				\end{aligned}
				\deqq
				and
				\beqq
				\begin{aligned}
					III_{32}
					&\lesssim \sum_{0< \beta_h \le \alpha_h}
					\|Z^{\beta_h} u_3\|_{L^2}^{\frac12}
					\|\p_3 Z^{\beta_h} u_3\|_{L^2}^{\frac12}
					\|\p_3 Z^{\ah-\beta_h} \wu\|_{L^2}^{\frac12}
					\|\p_3 Z^{\ah-\beta_h} \wu\|_{L^2}^{\frac12}\\
					&\quad\quad\quad\quad\quad \times \|\p_2  Z^{\ah} \wb\|_{L^2}^{\frac12}
					\|\p_{12}  Z^{\ah} \wb\|_{L^2}^{\frac12}\\
					&\lesssim \sqrt{\me^m(t)}\md_{tan}^{k+2}(t).
				\end{aligned}
				\deqq
				The term $III_{33}$ is somewhat complicated.
				If $\alpha_h=0$, we apply the divergence-free condition to get
				\beqq
				\begin{aligned}
					III_{33}
					=&-\i  u_3 \p_3 \wu  \cdot \p_2 \wb dx\\
					\lesssim
					&\|(u_3, \p_3 u_3)\|_{L^2}^{\frac14}
					\|\p_3(u_3, \p_3 u_3)\|_{L^2}^{\frac14}
					\|\p_2(u_3, \p_3 u_3)\|_{L^2}^{\frac14}
					\|\p_{23}(u_3, \p_3 u_3)\|_{L^2}^{\frac14}\\
					&\times \|Z_3 \wu\|_{L^2}^{\frac12}\|\p_1 Z_3 \wu\|_{L^2}^{\frac12}
					\| \p_2 \wb\|_{L^2}\\
					\lesssim
					&\|(u_3, \p_3 u_3)\|_{L^2}^{\frac14}
					\|\p_3(u_3, \p_3 u_3)\|_{L^2}^{\frac14}
					\|\p_2(u_3, \p_3 u_3)\|_{L^2}^{\frac14}
					\|\p_{23}(u_3, \p_3 u_3)\|_{L^2}^{\frac14}\\
					&\times \|Z_3 \wu\|_{L^2}^{\frac12}
					\|\p_1 \wu\|_{L^2}^{\frac14}
					\|\p_1 Z_3^2 \wu\|_{L^2}^{\frac14}
					\| \p_2 \wb\|_{L^2}\\
					\lesssim
					&\sqrt{\me^m(t)}\md_{tan}^{k+2}(t).
				\end{aligned}
				\deqq
				If $\alpha_h>0$,  we can conclude for all $0<|\alpha_h|=k\le m-3$
				\beqq
				\begin{aligned}
					III_{33}
					=&-\i  u_3 \p_3 Z^{\ah} \wu  \cdot \p_2  Z^{\ah} \wb dx\\
					\lesssim
					&\|(u_3, \p_3 u_3)\|_{L^2}^{\frac14}
					\|\p_3(u_3, \p_3 u_3)\|_{L^2}^{\frac14}
					\|\p_2(u_3, \p_3 u_3)\|_{L^2}^{\frac14}
					\|\p_{23}(u_3, \p_3 u_3)\|_{L^2}^{\frac14}\\
					&\times \|Z_3 Z^{\ah} \wu\|_{L^2}
					\|\p_2  Z^{\ah} \wb\|_{L^2}^{\frac12}
					\| \p_{12}  Z^{\ah} \wb\|_{L^2}^{\frac12}\\
					\lesssim
					&\|(u_3, \p_3 u_3)\|_{L^2}^{\frac14}
					\|\p_3(u_3, \p_3 u_3)\|_{L^2}^{\frac14}
					\|\p_2(u_3, \p_3 u_3)\|_{L^2}^{\frac14}
					\|\p_{23}(u_3, \p_3 u_3)\|_{L^2}^{\frac14}\\
					&\times \| Z^{\ah} \wu\|_{L^2}^{\frac12}
					\|Z_3^2 Z^{\ah} \wu\|_{L^2}^{\frac12}
					\|\p_2  Z^{\ah} \wb\|_{L^2}^{\frac12}
					\| \p_{12}  Z^{\ah} \wb\|_{L^2}^{\frac12}\\
					\lesssim
					&\sqrt{\me^m(t)}\md_{tan}^{k+2}(t).
				\end{aligned}
				\deqq
				For $0<|\alpha_h|=m-2$, we have
				\beqq
				\begin{aligned}
					III_{33}
					=&-\i  u_3 \p_3 Z^{\ah} \wu  \cdot \p_2  Z^{\ah} \wb dx\\
					\lesssim
					&\|(u_3, \p_3 u_3)\|_{L^2}^{\frac14}
					\|\p_3(u_3, \p_3 u_3)\|_{L^2}^{\frac14}
					\|\p_2(u_3, \p_3 u_3)\|_{L^2}^{\frac14}
					\|\p_{23}(u_3, \p_3 u_3)\|_{L^2}^{\frac14}\\
					&\times \|Z_3 Z^{\ah} \wu\|_{L^2}^{\frac12}
					\|\p_1 Z_3 Z^{\ah} \wu\|_{L^2}^{\frac12}
					\| \p_{2}  Z^{\ah} \wb\|_{L^2}\\
					\lesssim
					&\sqrt{\me^m(t)}\md^m(t).
				\end{aligned}
				\deqq
				Thus, we can obtain the estimate for all $|\alpha_h|=k \le m-3$
				\beq\label{3607}
				III_3 \lesssim \sqrt{\me^m(t)}\md_{tan}^{k+2}(t),
				\deq
				and for all $|\alpha_h|=m-2$
				\beq\label{3608}
				III_3 \lesssim \sqrt{\me^m(t)}\md^{m}(t).
				\deq
				Similarly, we have for all $|\alpha_h|=k\le m-3$
				\beq\label{3609}
				III_5, III_7, III_9  \lesssim \sqrt{\me^m(t)}\md_{tan}^{k+2}(t),
				\deq
				and for all $|\alpha_h|=m-2$
				\beq\label{3610}
				III_5, III_7, III_9 \lesssim \sqrt{\me^m(t)}\md^{m}(t).
				\deq
				Using the anisotropic type inequality \eqref{ie:Sobolev}, we can get
				\beq\label{3611}
				\begin{aligned}
					III_4
					=&\sum_{0\le \beta_h \le \alpha_h}C^{\beta_h}_{\alpha_h}
					\i (Z^{\beta_h} \wu_h\cdot Z^{\ah-\beta_h} \nabla_h u
					+Z^{\beta_h} \wu_3 \cdot Z^{\ah-\beta_h} \p_3 u)
					\cdot \p_2  Z^{\ah} \wb dx\\
					\lesssim
					&\sum_{0\le \beta_h \le \alpha_h}
					\|Z^{\beta_h} \wu_h\|_{L^2}^{\frac12}
					\|\p_2 Z^{\beta_h} \wu_h\|_{L^2}^{\frac12}
					\|Z^{\ah-\beta_h} \nabla_h u\|_{L^2}^{\frac12}
					\|\p_3 Z^{\ah-\beta_h} \nabla_h u\|_{L^2}^{\frac12}\\
					&\quad \quad \quad \times \|\p_2  Z^{\ah} \wb\|_{L^2}^{\frac12}
					\|\p_{12}  Z^{\ah} \wb\|_{L^2}^{\frac12}\\
					&+\sum_{0\le \beta_h \le \alpha_h}
					\|Z^{\beta_h} \wu_3\|_{L^2}^{\frac12}
					\|\p_3 Z^{\beta_h} \wu_3\|_{L^2}^{\frac12}
					\|Z^{\ah-\beta_h} \p_3 u\|_{L^2}^{\frac12}
					\|\p_2 Z^{\ah-\beta_h} \p_3 u\|_{L^2}^{\frac12}\\
					&\quad \quad \quad \times \|\p_2  Z^{\ah} \wb \|_{L^2}^{\frac12}
					\|\p_{12}  Z^{\ah} \wb \|_{L^2}^{\frac12}\\
					\lesssim
					&\sqrt{\me^m(t)}\md_{tan}^{k+2}(t).
				\end{aligned}
				\deq
				Similarly, due to the relations \eqref{3306},
				\eqref{3307} and \eqref{3308}, we may deduce that
				\beq\label{3612}
				III_6, III_8, III_{10}\lesssim \sqrt{\me^m(t)}\md_{tan}^{k+2}(t).
				\deq
				Taking all estimates for $III_1$ through $III_{10}$ into account,
				we complete the proof of this lemma.
			\end{proof}
			
			Finally, we will establish the dissipative estimate for normal derivative of
			velocity under the higher conormal Sobolev space in the $x_2$ direction as follows.
			\begin{lemm}
				For any smooth solution $(u, b)$ of equation \eqref{eqr},
				it holds for all $|\alpha|\le m-3$
				\beq\label{3701}
				\begin{aligned}
					&-\frac{d}{dt}\i \p_3 Z^\alpha \wb \cdot \p_{23} Z^\alpha \wu dx
					+\i |\p_{23} Z^\alpha \wu|^2 dx\\
					\lesssim
					&\|(\p_{13} \wu, \nabla_h \p_3 \wb)\|_{H^{m-2}_{co}}^2
					+\ep \|(\p_{23} \wu, \p_{33} \wu, \p_{33} \wb )\|_{H^{m-2}_{co}}^2
					+\sqrt{\me^m(t)}\md^{m}(t).
				\end{aligned}
				\deq
			\end{lemm}
			\begin{proof}
				For any $|\alpha|\le m-3$,  then the equation \eqref{eqwb} yields directly
				\beq\label{3702}
				\begin{aligned}
					&-\i Z^\alpha \p_2 \wu \cdot \p_{233} Z^\alpha \wu dx\\
					=&-\i \p_t Z^\alpha \wb \cdot \p_{233} Z^\alpha \wu dx
					+\i Z^\alpha (\Delta_h \wb+\ep \p_3^2 \wb)\cdot \p_{233} Z^\alpha \wu dx\\
					&+\i Z^{\al} (-u\cdot \nabla \wb-\nabla(u\cdot \nabla)\times b
					+b\cdot \nabla \wu+\nabla(b\cdot \nabla)\times u)
					\cdot \p_{233} Z^\alpha \wu dx.
				\end{aligned}
				\deq
				Integrating by part and using the boundary condition \eqref{bd-wuwb}, we have
				\beq\label{3703}
				-\i Z^\alpha \p_2 \wu \cdot \p_{233} Z^\alpha \wu dx
				=\i |\p_{23} Z^\alpha \wu|^2 dx.
				\deq
				Integrating by part and using the equation \eqref{eqwu}, it is easy to check that
				\beq\label{3704}
				\begin{aligned}
					&-\i \p_t Z^\alpha \wb \cdot \p_{233} Z^\alpha \wu dx
					=\i \p_t \p_3 Z^\alpha \wb \cdot \p_{23} Z^\alpha \wu dx\\
					=&\frac{d}{dt}\i \p_3 Z^\alpha \wb \cdot \p_{23} Z^\alpha \wu dx
					-\i \p_3 Z^\alpha \wb \cdot \p_{23} Z^\alpha \p_t  \wu dx\\
					=&\frac{d}{dt}\i \p_3 Z^\alpha \wb \cdot \p_{23} Z^\alpha \wu dx
					-\i \p_{233} Z^\alpha \wb \cdot  Z^\alpha \p_t  \wu dx\\
					=&\frac{d}{dt}\i \p_3 Z^\alpha \wb \cdot \p_{23} Z^\alpha \wu dx
					-\i \p_{233} Z^\alpha \wb \cdot
					Z^\alpha(\p_{11} \wu+\ep \p_{22} \wu+\ep \p_{33} \wu+\p_2 \wb)dx\\
					&-\i \p_{233} Z^\alpha \wb
					\cdot  Z^\alpha(-u\cdot \nabla \wu+\wu\cdot \nabla u+b \cdot \nabla \wb-\wb\cdot \nabla b)dx.
				\end{aligned}
				\deq
				Then, the combination of equations \eqref{3702}, \eqref{3703}
				and \eqref{3704} yields directly
				\beqq
				\begin{aligned}
					&-\frac{d}{dt}\i \p_3 \wb \cdot \p_{23} \wu dx+\i |\p_{23}  \wu|^2 dx\\
					=&-
					\i  Z^\alpha(\p_1^2 \wu+\ep \p_2^2 \wu+\ep \p_3^2 \wu+\p_2 \wb)
					\cdot \p_{233} Z^\alpha \wb dx\\
					&+\i Z^\alpha (\Delta_h \wb+\ep \p_3^2 \wb)\cdot \p_{233} Z^\alpha \wu dx
					+\i Z^\alpha(u\cdot \nabla \wu )\cdot  \p_{233} Z^\alpha \wb dx\\
					&+\i Z^\alpha(-\wu\cdot \nabla u)\cdot  \p_{233} Z^\alpha \wb dx
					+\i Z^\alpha(-b \cdot \nabla \wb)\cdot  \p_{233} Z^\alpha \wb dx\\
					&+\i Z^\alpha(\wb\cdot \nabla b)\cdot  \p_{233} Z^\alpha \wb dx
					+\i Z^{\al} (-u\cdot \nabla \wb)\cdot \p_{233} Z^\alpha \wu dx\\
					&+\i Z^{\al} (-\nabla(u\cdot \nabla)\times b)\cdot \p_{233} Z^\alpha \wu dx
					+\i Z^{\al} (b\cdot \nabla \wu) \cdot \p_{233} Z^\alpha \wu dx\\
					&+\i Z^{\al} (\nabla(b\cdot \nabla)\times u)\cdot \p_{233} Z^\alpha \wu dx
					:=\sum_{i=1}^{10} IV_i.
				\end{aligned}
				\deqq
				Integrating by part, we may write
				\beq\label{3705}
				\begin{aligned}
					IV_1
					=&-\i  Z^\alpha(\p_1^2 \wu+\ep \p_2^2 \wu+\ep \p_3^2 \wu+\p_2 \wb)
					\cdot \p_{233} Z^\alpha \wb dx\\
					=&\i  \p_{113} Z^\alpha \wu \cdot \p_{23} Z^\alpha \wb dx
					+\ep \i  \p_{223} Z^\alpha  \wu \cdot \p_{23} Z^\alpha \wb dx\\
					& -\ep \i  Z^\alpha \p_3^2 \wu \cdot \p_{233} Z^\alpha \wb dx
					+\i \p_{23} Z^\alpha  \wb \cdot \p_{23} Z^\alpha \wb dx\\
					\lesssim
					&\|(\p_{113} Z^\alpha \wu, \p_{23} Z^\alpha \wb)\|_{L^2}^2
					+\ep \|(\p_{223} Z^\alpha  \wu, \p_{23} Z^\alpha \wb)\|_{L^2}^2\\
					&+\ep \|(Z^\alpha \p_{33} \wu,\p_{233} Z^\alpha \wb)\|_{L^2}^2
					+\|\p_{23} Z^\alpha \wb\|_{L^2}^2,
				\end{aligned}
				\deq
				and
				\beq\label{3706}
				\begin{aligned}
					IV_2
					=&\i Z^\alpha (\p_1^2 \wb+\p_2^2 \wb+\ep \p_3^2 \wb)\cdot \p_{233} Z^\alpha \wu dx\\
					=&\i   \p_{13}  Z^\alpha \wb \cdot \p_{123} Z^\alpha \wu dx
					-\i \p_{223} Z^\alpha  \wb \cdot \p_{23} Z^\alpha \wu dx\\
					&+\ep \i \p_{33} Z^\alpha   \wb \cdot \p_{233} Z^\alpha \wu dx\\
					\lesssim
					&\frac14 \|\p_{23} Z^\alpha \wu\|_{L^2}^2
					+\|(\p_{13}  Z^\alpha \wb, \p_{123} Z^\alpha \wu,
					\p_{223} Z^\alpha  \wb)\|_{L^2}^2\\
					&+\ep\|(\p_{33} Z^\alpha   \wb, \p_{233} Z^\alpha \wu)\|_{L^2}^2.
				\end{aligned}
				\deq
				Integrating by part, it is easy to check that
				\beqq
				\begin{aligned}
					IV_9
					=
					-\i \p_{3}Z^{\al} (b_h\cdot \nabla_h \wu +b_3 \p_3 \wu )\cdot \p_{23} Z^\alpha \wu dx
					:=IV_{91}+IV_{92}+IV_{93}+IV_{94},
				\end{aligned}
				\deqq
				where the terms $IV_{9i}(i=1,2,3,4)$ are defined below
				\beqq
				\begin{aligned}
					IV_{91}:=&-\sum_{0\le \beta \le \alpha}C^{\beta}_{\alpha}
					\i \p_3 Z^\beta b_h Z^{\alpha-\beta} \nabla_h \wu
					\cdot \p_{23} Z^\alpha \wu dx,\\
					IV_{92}:=&-\sum_{0\le \beta \le \alpha}C^{\beta}_{\alpha}
					\i Z^\beta b_h \p_3 Z^{\alpha-\beta} \nabla_h \wu
					\cdot \p_{23} Z^\alpha \wu dx,\\
				\end{aligned}
				\deqq
				and
				\beqq
				\begin{aligned}
					IV_{93}:=&-\sum_{0\le \beta \le \alpha}C^{\beta}_{\alpha}
					\i \p_3 Z^\beta b_3 Z^{\alpha-\beta}\p_3 \wu
					\cdot \p_{23} Z^\alpha \wu dx,\\
					IV_{94}:=&-\sum_{0\le \beta \le \alpha}C^{\beta}_{\alpha}
					\i Z^\beta b_3 \p_3 Z^{\alpha-\beta}\p_3 \wu
					\cdot \p_{23} Z^\alpha \wu dx.
				\end{aligned}
				\deqq
				Using the anisotropic type inequality \eqref{ie:Sobolev}, we may write
				\beqq
				\begin{aligned}
					IV_{93}
					\lesssim
					&\sum_{0\le \beta \le \alpha}
					\|\p_3 Z^\beta b_3\|_{L^2}^{\frac14}
					\|\p_{33}Z^\beta b_3\|_{L^2}^{\frac14}
					\|\p_{23} Z^\beta b_3\|_{L^2}^{\frac14}
					\|\p_{233}Z^\beta b_3\|_{L^2}^{\frac14}\\
					&\times
					\|Z^{\alpha-\beta}\p_3 \wu\|_{L^2}^{\frac12}
					\|\p_1 Z^{\alpha-\beta}\p_3 \wu\|_{L^2}^{\frac12}
					\|\p_{23} Z^\alpha \wu\|_{L^2}\\
					\lesssim
					&\sqrt{\me^m(t)}\md^{m}(t),
				\end{aligned}
				\deqq
				and
				\beqq
				\begin{aligned}
					IV_{94}
					\lesssim
					&\sum_{0\le \beta \le \alpha}
					\|(Z^\beta b_3, \p_3 Z^\beta b_3)\|_{L^2}^{\frac14}
					\|(\p_3 Z^\beta b_3, \p_{33}Z^\beta b_3)\|_{L^2}^{\frac14}
					\|(\p_2 Z^\beta b_3, \p_{23} Z^\beta b_3)\|_{L^2}^{\frac14}\\
					&\times
					\|(\p_{23}Z^\beta b_3, \p_{233}Z^\beta b_3)\|_{L^2}^{\frac14}
					\|Z^{\alpha-\beta+e_3}\p_3 \wu\|_{L^2}^{\frac12}
					\|\p_1 Z^{\alpha-\beta+e_3}\p_3 \wu\|_{L^2}^{\frac12}
					\|\p_{23} Z^\alpha \wu\|_{L^2}\\
					\lesssim
					&\sqrt{\me^m(t)}\md^{m}(t).
				\end{aligned}
				\deqq
				Similarly, it is easy to check that
				\beqq
				IV_{91},IV_{92} \lesssim \sqrt{\me^m(t)}\md^{m}(t).
				\deqq
				Thus, we can obtain the estimate for $IV_{9}$ as follows
				\beq\label{3707}
				IV_{9} \lesssim \sqrt{\me^m(t)}\md^{m}(t).
				\deq
				Similarly, we can get that
				\beq\label{3708}
				IV_{3}, IV_{5},IV_{7} \lesssim \sqrt{\me^m(t)}\md^{m}(t).
				\deq
				Integrating by part, we have
				\beqq
				\begin{aligned}
					IV_4
					=
					&\i \p_{3}Z^\alpha(\wu\cdot \nabla u)\cdot  \p_{23} Z^\alpha \wb dx\\
					=
					&\sum_{0\le \beta \le \alpha}C^{\beta}_{\alpha}
					\i (\p_{3}Z^{\beta} \wu_h \cdot Z^{\alpha-\beta} \nabla_h u
					+\p_{3}Z^{\beta} \wu_3 \cdot Z^{\alpha-\beta} \p_3 u )
					\cdot  \p_{23} Z^\alpha \wb dx\\
					&+\sum_{0\le \beta \le \alpha}C^{\beta}_{\alpha}
					\i (\p_{3}Z^{\beta} \wu_h \cdot Z^{\alpha-\beta} \nabla_h u
					+Z^{\beta} \wu_3 \p_{3} Z^{\alpha-\beta} \p_3 u)
					\cdot  \p_{23} Z^\alpha \wb dx\\
					:=&IV_{41}+IV_{42}.
				\end{aligned}
				\deqq
				Using the anisotropic type inequality \eqref{ie:Sobolev}, we may write
				\beqq
				\begin{aligned}
					IV_{41}
					\lesssim
					&\sum_{0\le \beta \le \alpha}
					\|\p_{3}Z^{\beta} \wu_h\|_{L^2}^{\frac12}
					\|\p_1 \p_{3}Z^{\beta} \wu_h\|_{L^2}^{\frac12}
					\|Z^{\alpha-\beta} \nabla_h u \|_{L^2}^{\frac12}\\
					&\quad \quad \times \|\p_3 Z^{\alpha-\beta} \nabla_h u \|_{L^2}^{\frac12}
					\|\p_{23} Z^\alpha \wb\|_{L^2}^{\frac12}
					\|\p_{223} Z^\alpha \wb\|_{L^2}^{\frac12}\\
					&+\sum_{0\le \beta \le \alpha}
					\|\p_{3}Z^{\beta} \wu_3\|_{L^2}^{\frac12}
					\|\p_3\p_{3}Z^{\beta} \wu_3\|_{L^2}^{\frac12}
					\|Z^{\alpha-\beta} \p_3 u\|_{L^2}^{\frac12}\\
					&\quad \quad \times \|\p_1 Z^{\alpha-\beta} \p_3 u\|_{L^2}^{\frac12}
					\|\p_{23} Z^\alpha \wb\|_{L^2}^{\frac12}
					\|\p_{223} Z^\alpha \wb\|_{L^2}^{\frac12}\\
					\lesssim
					&\sqrt{\me^m(t)}\md^{m}(t),
				\end{aligned}
				\deqq
				and
				\beqq
				\begin{aligned}
					IV_{42}
					\lesssim
					&\sum_{0\le \beta \le \alpha}
					\|\p_{3}Z^{\beta} \wu_h\|_{L^2}^{\frac12}
					\|\p_1 \p_{3}Z^{\beta} \wu_h\|_{L^2}^{\frac12}
					\|Z^{\alpha-\beta} \nabla_h u\|_{L^2}^{\frac12}\\
					&\quad \quad \times
					\|\p_3 Z^{\alpha-\beta} \nabla_h u\|_{L^2}^{\frac12}
					\|\p_{23} Z^\alpha \wb\|_{L^2}^{\frac12}
					\|\p_{223} Z^\alpha \wb\|_{L^2}^{\frac12}\\
					&+\sum_{0\le \beta \le \alpha}
					\|Z^{\beta} \wu_3\|_{L^2}^{\frac12}
					\|\p_2Z^{\beta} \wu_3\|_{L^2}^{\frac12}
					\|\p_{3} Z^{\alpha-\beta} \p_3 u\|_{L^2}^{\frac12}\\
					&\quad \quad \times \|\p_{13} Z^{\alpha-\beta} \p_3 u\|_{L^2}^{\frac12}
					\|\p_{23} Z^\alpha \wb\|_{L^2}^{\frac12}
					\|\p_{223} Z^\alpha \wb\|_{L^2}^{\frac12}\\
					\lesssim
					&\sqrt{\me^m(t)}\md^{m}(t).
				\end{aligned}
				\deqq
				Thus, we can obtain the estimate
				\beq\label{3708}
				IV_{4} \lesssim \sqrt{\me^m(t)}\md^{m}(t).
				\deq
				Similarly, we can obtain
				\beq\label{3709}
				IV_{6}, IV_{8}, IV_{10} \lesssim \sqrt{\me^m(t)}\md^{m}(t).
				\deq
				Thus, the combination of estimates \eqref{3705}-\eqref{3709}  yields directly
				\beqq
				\begin{aligned}
					&-\frac{d}{dt}\i \p_3 Z^\alpha \wb \cdot \p_{23} Z^\alpha \wu dx
					+\i |\p_{23} Z^\alpha \wu|^2 dx\\
					\lesssim
					&\|(\p_{13} \wu, \nabla_h \p_3 \wb)\|_{H^{m-2}_{co}}^2
					+\ep \|(\p_{23} \wu, \p_{33} \wu, \p_{33} \wb )\|_{H^{m-2}_{co}}^2
					+\sqrt{\me^m(t)}\md^{m}(t).
				\end{aligned}
				\deqq
				Therefore, we complete the proof of this lemma.
			\end{proof}

			\subsection{Negative derivative estimate}
			
			In this subsection, we will establish the estimate in negative
			Sobolev space that will play an important role in establishing
			the decay rate estimate.
			
			\begin{lemm}
				Under the assumption \eqref{assumption},
				the smooth solution $(u, b)$ of equation \eqref{eqr} has the estimate
				\beq\label{3801}
				\begin{aligned}
					&\|(\Lambda_h^{-s}u, \Lambda_h^{-s} b)(t)\|_{L^2}^2
					+\ep \int_0^t \|(\p_2 \Lambda_h^{-s} u, \p_3 \Lambda_h^{-s} u,
					\p_3 \Lambda_h^{-s} b)\|_{L^2}^2 d\tau\\
					&+\int_0^t \|(\p_1 \Lambda_h^{-s} u, \nabla_h \Lambda_h^{-s} b)\|_{L^2}^2 d\tau
					\lesssim \|(\Lambda_h^{-s}u_0, \Lambda_h^{-s} b_0)\|_{L^2}^2+\delta^2.
				\end{aligned}
				\deq
			\end{lemm}
			\begin{proof}
				The equation \eqref{eqr} yields directly
				\beq\label{3802}
				\begin{aligned}
					&\frac{d}{dt}\frac{1}{2}\i(|\Lambda_h^{-s}u|^2+|\Lambda_h^{-s} b|^2)dx
					+\ep \i (|\p_2 \Lambda_h^{-s} u|^2+|\p_3 \Lambda_h^{-s} u|^2) dx\\
					&+\i |\p_1 \Lambda_h^{-s} u|^2 dx
					+\i |\nabla_h \Lambda_h^{-s} b|^2 dx+\ep \i |\p_3 \Lambda_h^{-s} b|^2 dx\\
					=&\i \Lambda_h^{-s}(-u \cdot \nabla u) \cdot \Lambda_h^{-s} u \ dx
					+\i \Lambda_h^{-s}(b \cdot \nabla b) \cdot \Lambda_h^{-s} u \ dx\\
					&+\i \Lambda_h^{-s}(-u \cdot \nabla b) \cdot \Lambda_h^{-s} b \ dx
					+\i \Lambda_h^{-s}(b \cdot \nabla u) \cdot \Lambda_h^{-s} b \ dx\\
					:=&J_1+J_2+J_3+J_4.
				\end{aligned}
				\deq
				Using the inequalities \eqref{ie:Sobolev}  and \eqref{a16},
				we may deduce that
				\beqq
				\begin{aligned}
					J_1
					=&-\int_0^{+\infty}\int_{\mathbb{R}^2}
					\Lambda_h^{-s}(u_h \cdot \nabla_h u+u_3 \p_3 u)\cdot \Lambda_h^{-s}u dx_h dx_3\\
					\lesssim
					&\int_0^{+\infty}(\|\Lambda_h^{-s}(u_h \cdot \nabla_h u)\|_{L^2(\mathbb{R}^2)}
					+\|\Lambda_h^{-s}(u_3 \p_3 u)\|_{L^2(\mathbb{R}^2)})
					\|\Lambda_h^{-s}u\|_{L^2(\mathbb{R}^2)}dx_3\\
					\lesssim
					&\int_0^{+\infty}(\|u_h \cdot \nabla_h u\|_{L^{\frac{1}{\frac{1}{2}+\frac{s}{2}}}(\mathbb{R}^2)}
					+\|u_3 \p_3 u\|_{L^{\frac{1}{\frac{1}{2}+\frac{s}{2}}}(\mathbb{R}^2)})
					\|\Lambda_h^{-s}u\|_{L^2(\mathbb{R}^2)}dx_3\\
					\lesssim
					&\int_0^{+\infty}(\|u_h\|_{L^{\frac{2}{s}}(\mathbb{R}^2)}
					\|\nabla_h u\|_{L^2(\mathbb{R}^2)}
					+\|u_3\|_{L^{\frac{2}{s}}(\mathbb{R}^2)}
					\|\p_3 u\|_{L^2(\mathbb{R}^2)})
					\|\Lambda_h^{-s}u\|_{L^2(\mathbb{R}^2)}dx_3\\
					\lesssim
					&\left.\|\|u_h\|_{L^\infty(\mathbb{R}^+)}\right\|_{L^{\frac{2}{s}}(\mathbb{R}^2)}
					\|\nabla_h u\|_{L^2(\mathbb{R}^3_+)}
					\|\Lambda_h^{-s}u\|_{L^2(\mathbb{R}^3_+)}\\
					&+\left.\|\|u_3\|_{L^\infty(\mathbb{R}^+)}\right\|_{L^{\frac{2}{s}}(\mathbb{R}^2)}
					\|\p_3 u\|_{L^2(\mathbb{R}^3_+)}
					\|\Lambda_h^{-s}u\|_{L^2(\mathbb{R}^3_+)}\\
					\lesssim
					&(\|u_h\|_{L^2}\|\p_2 u_h\|_{L^2}
					+\|\p_1 u_h\|_{L^2}\|\p_{12} u_h\|_{L^2})^{\frac{1-s}{2}}
					\|u_h\|_{L^2}^{\frac{2s-1}{2}}\|\p_3 u_h\|_{L^2}^{\frac{1}{2}}
					\|\nabla_h u\|_{L^2}\|\Lambda_h^{-s}u\|_{L^2}\\
					&+(\|u_3\|_{L^2}\|\p_2 u_3\|_{L^2}
					+\|\p_1 u_3\|_{L^2}\|\p_{12} u_3\|_{L^2})^{\frac{1-s}{2}}
					\|u_3\|_{L^2}^{\frac{2s-1}{2}}\|\p_3 u_3\|_{L^2}^{\frac{1}{2}}
					\|\p_3 u\|_{L^2}\|\Lambda_h^{-s}u\|_{L^2}\\
					\lesssim
					&\sqrt{\mathcal{E}_{tan}^{m-1}(t)}\sqrt{\md_{tan}^{m-1}(t)}
					\|\Lambda_h^{-s}u\|_{L^2}
					+\mathcal{E}_{tan}^{m-1}(t)^{\frac34}
					\md_{tan}^{m-1}(t)^{\frac14}
					\|\Lambda_h^{-s}u\|_{L^2}.
				\end{aligned}
				\deqq
				Similarly, we can obtain the estimate
				\beqq
				\begin{aligned}
					J_2
					\lesssim
					&\sqrt{\mathcal{E}_{tan}^{m-1}(t)}\sqrt{\md_{tan}^{m-1}(t)}
					\|\Lambda_h^{-s}u\|_{L^2}
					+\mathcal{E}_{tan}^{m-1}(t)^{\frac34}
					\md_{tan}^{m-1}(t)^{\frac14}
					\|\Lambda_h^{-s}u\|_{L^2},\\
					J_3, J_4
					\lesssim
					&\sqrt{\mathcal{E}_{tan}^{m-1}(t)}\sqrt{\md_{tan}^{m-1}(t)}
					\|\Lambda_h^{-s}b\|_{L^2}
					+\mathcal{E}_{tan}^{m-1}(t)^{\frac34}
					\md_{tan}^{m-1}(t)^{\frac14}
					\|\Lambda_h^{-s}b\|_{L^2}.
				\end{aligned}
				\deqq
				Substituting the estimates of terms for $J_1$  through $J_4$
				into \eqref{3802} and integrating over $[0, t]$, we obtain
				\beqq
				\begin{aligned}
					&\frac{1}{2}\|(\Lambda_h^{-s}u, \Lambda_h^{-s} b)(t)\|_{L^2}^2
					+\ep \int_0^t \|(\p_2 \Lambda_h^{-s} u, \p_3 \Lambda_h^{-s} u,
					\p_3 \Lambda_h^{-s} b)\|_{L^2}^2 d\tau
					+\int_0^t \|(\p_1 \Lambda_h^{-s} u, \nabla_h \Lambda_h^{-s} b)\|_{L^2}^2 d\tau\\
					\lesssim
					&\frac{1}{2}\|(\Lambda_h^{-s}u_0, \Lambda_h^{-s} b_0)\|_{L^2}^2
					+\underset{0\le \tau \le t}{\sup}\|(\Lambda_h^{-s}u, \Lambda_h^{-s} b)(\tau)\|_{L^2}
					\int_0^t \sqrt{\mathcal{E}_{tan}^{m-1}(\tau)}\sqrt{\md_{tan}^{m-1}(\tau)}d\tau\\
					&+\underset{0\le \tau \le t}{\sup}\|(\Lambda_h^{-s}u, \Lambda_h^{-s} b)(\tau)\|_{L^2}
					\int_0^t \mathcal{E}_{tan}^{m-1}(\tau)^{\frac34} \md_{tan}^{m-1}(\tau)^{\frac14} d\tau.
				\end{aligned}
				\deqq
				Using the assumption \eqref{assumption} and H\"{o}lder's inequality, we have
				\beq\label{3803}
				\begin{aligned}
					\int_0^t  \sqrt{\mathcal{E}_{tan}^{m-1}(\tau)}\sqrt{\md_{tan}^{m-1}(\tau)} d\tau
					&\le
					\left\{\int_0^t   \mathcal{E}_{tan}^{m-1}(\tau)
					(1+\tau)^{-\sigma} d\tau\right\}^{\frac{1}{2}}
					\left\{\int_0^t  \md_{tan}^{m-1}(\tau) (1+\tau)^{\sigma}
					d\tau\right\}^{\frac{1}{2}}\\
					&\lesssim
					\delta \left\{\int_0^t (1+\tau)^{-(s+\sigma)} d\tau\right\}^{\frac{1}{2}}
					\lesssim \delta,
				\end{aligned}
				\deq
				and
				\beqq
				\begin{aligned}
					\int_0^t  \mathcal{E}_{tan}^{m-1}(\tau)^{\frac34} \md_{tan}^{m-1}(\tau)^{\frac14} d\tau
					\lesssim
					&\left\{\int_0^t  \mathcal{E}_{tan}^{m-1}(\tau)
					(1+\tau)^{-\frac13\sigma} d\tau\right\}^{\frac{3}{4}}
					\left\{\int_0^t  \md_{tan}^{m-1}(\tau) (1+\tau)^{\sigma}
					d\tau\right\}^{\frac{1}{4}}\\
					\lesssim&
					\delta \left\{\int_0^t (1+\tau)^{-(s+\frac13\sigma)} d\tau\right\}^{\frac{3}{4}}
					\lesssim \delta,
				\end{aligned}
				\deqq
				where we require $\frac34<\sigma<s<1$ in the above estimate.
				Thus, we can obtain the following estimate
				\beqq
				\begin{aligned}
					&\|(\Lambda_h^{-s}u, \Lambda_h^{-s} b)(t)\|_{L^2}^2
					+\ep \int_0^t \|(\p_2 \Lambda_h^{-s} u, \p_3 \Lambda_h^{-s} u,
					\p_3 \Lambda_h^{-s} b)\|_{L^2}^2 d\tau\\
					&+\int_0^t \|(\p_1 \Lambda_h^{-s} u, \nabla_h \Lambda_h^{-s} b)\|_{L^2}^2 d\tau
					\lesssim
					\|(\Lambda_h^{-s}u_0, \Lambda_h^{-s} b_0)\|_{L^2}^2
					+\delta^2.
				\end{aligned}
				\deqq
				Therefore, we complete the proof of this lemma.
			\end{proof}
			
			Next, we establish estimate of normal derivative
			of velocity and magnetic field in negative Sobolev space.
			\begin{lemm}
				Under the assumption \eqref{assumption},
				the smooth solution $(u, b)$ of equation \eqref{eqr} has the estimate
				\beq\label{3901}
				\begin{aligned}
					&\|(\Lambda_h^{-s}\wu,\Lambda_h^{-s} \wb)(t)\|_{L^2}^2
					+\int_0^t \|(\p_1 \Lambda_h^{-s} \wu, \nabla_h \Lambda_h^{-s} \wb)\|_{L^2}^2 d\tau\\
					&+\ep \int_0^t  \|(\p_2 \Lambda_h^{-s} \wu, \p_3 \Lambda_h^{-s} \wu,
					\p_3 \Lambda_h^{-s} \wb)\|_{L^2}^2d\tau
					\lesssim
					\|(\Lambda_h^{-s}\wu,\Lambda_h^{-s} \wb)(0)\|_{L^2}^2+\delta^2.
				\end{aligned}
				\deq
			\end{lemm}
			\begin{proof}
				Using the equations \eqref{eqwu} and \eqref{eqwb} and
				integrating by part, it is easy to check that
				\begin{equation}\label{3902}
					\begin{aligned}
						&\frac{d}{dt}\frac{1}{2}\i(|\Lambda_h^{-s}\wu|^2+|\Lambda_h^{-s} \wb|^2)dx
						+\ep \i (|\p_2 \Lambda_h^{-s} \wu|^2+|\p_3 \Lambda_h^{-s} \wu|^2) dx\\
						&+\i |\p_1 \Lambda_h^{-s} \wu|^2 dx
						+\i |\nabla_h \Lambda_h^{-s} \wb|^2 dx+\ep \i |\p_3 \Lambda_h^{-s} \wb|^2 dx\\
						=&\i \Lambda_h^{-s}(-u\cdot \nabla \wu) \cdot \Lambda_h^{-s} \wu \ dx
						+\i \Lambda_h^{-s}(\wu\cdot \nabla u) \cdot \Lambda_h^{-s} \wu \ dx\\
						&+\i \Lambda_h^{-s}(b \cdot \nabla \wb) \cdot \Lambda_h^{-s} \wu \ dx
						+\i \Lambda_h^{-s}(-\wb\cdot \nabla b) \cdot \Lambda_h^{-s} \wu \ dx\\
						&
						+\i \Lambda_h^{-s}(-u\cdot \nabla \wb) \cdot \Lambda_h^{-s} \wb \ dx
						+\i \Lambda_h^{-s}(-\nabla(u\cdot \nabla)\times b) \cdot \Lambda_h^{-s} \wb \ dx\\
						&+\i \Lambda_h^{-s}(b\cdot \nabla \wu) \cdot \Lambda_h^{-s} \wb \ dx
						+\i \Lambda_h^{-s}(\nabla(b\cdot \nabla)\times u) \cdot \Lambda_h^{-s} \wb \ dx
						:=\sum_{i=1}^{8}K_i.
					\end{aligned}
				\end{equation}
				The H\"{o}lder's inequality and  inequality \eqref{a16} in Lemma \ref{H-L} yield directly
				\begin{equation}\label{3903}
					\begin{aligned}
						\int_{\mathbb{R}_+^3} \Lambda_h^{-s}(u\cdot \nabla \omega^u)
						\cdot \Lambda_h^{-s} \omega^u dx
						\lesssim \int_0^{+\infty}
						(\|u_h \partial_h \omega^u\|_{L^{\frac{1}{\frac12+\frac{s}{2}}}(\mathbb{R}^2)}
						+\|u_3 \partial_3 \omega^u\|_{L^{\frac{1}{\frac12+\frac{s}{2}}}(\mathbb{R}^2)})
						\|\Lambda_h^{-s} \omega^u\|_{L^2(\mathbb{R}^2)}dx_3.
					\end{aligned}
				\end{equation}
				Similar to the estimate of term $J_1$, it is easy to check that
				\begin{equation}\label{3904}
					\begin{aligned}
						&~~~~\int_0^{+\infty}
						\|u_h \partial_h \omega^u\|_{L^{\frac{1}{\frac12+\frac{s}{2}}}(\mathbb{R}^2)}
						\|\Lambda_h^{-s} \omega^u\|_{L^2(\mathbb{R}^2)}dx_3\\
						&\lesssim
						\left\|\|u_h \|_{L^\infty(\mathbb{R}_+)}\right\|_{L^{\frac{2}{s}}}
						\|\partial_h \omega^u\|_{L^2}
						\|\Lambda_h^{-s} \omega^u\|_{L^2}\\
						&\lesssim
						\left(\|u_h\|_{L^2}\|\partial_2 u_h\|_{L^2}
						+\|\partial_1 u_h\|_{L^2}\|\partial_{12}u_h\|_{L^2}\right)^{\frac{1-s}{2}}
						\|u_h\|_{L^2}^{\frac{2s-1}{2}}
						\|\partial_3 u_h\|_{L^2}^{\frac12}
						\|\partial_h \omega^u\|_{L^2}
						\|\Lambda_h^{-s} \omega^u\|_{L^2}\\
						&\lesssim \sqrt{\mathcal{E}_{tan}^{m-1}(t)} \sqrt{\md_{tan}^{m-1}(t)}
						\|\Lambda_h^{-s} \omega^u\|_{L^2}.
					\end{aligned}
				\end{equation}
				Now, let us deal with the difficult term $\int_0^{+\infty}
				\|u_3 \partial_3 \omega^u\|_{L^{\frac{1}{\frac12+\frac{s}{2}}}(\mathbb{R}^2)}
				\|\Lambda_h^{-s} \omega^u\|_{L^2(\mathbb{R}^2)}dx_3$.
				Indeed, we may check that
				\begin{equation*}
					\begin{aligned}
						u_3 \partial_3 \omega^u
						&=u_3 \partial_3 \omega^u \chi(x_3)+u_3 \partial_3 \omega^u (1-\chi(x_3))\\
						&=\int_0^{x_3} \partial_3 u_3 d\xi \partial_3 \omega^u \chi(x_3)
						+\varphi^{-1}u_3 \varphi\partial_3 \omega^u (1-\chi(x_3)),
					\end{aligned}
				\end{equation*}
				which yields directly
				\begin{equation*}
					|u_3 \partial_3 \omega^u_h|
					\lesssim \|\partial_3 u_3\|_{L^\infty(\mathbb{R}_+)}|Z_3 \omega^u_h|
					+|u_3||Z_3 \omega^u_h|.
				\end{equation*}
				Here $\chi(x_3)$ is a smooth cutoff function supported
				near the boundary $x_3=0$.
				Then, this decomposition yields directly
				\begin{equation}\label{3905}
					\begin{aligned}
						&~~~~\int_0^{+\infty}
						\|u_3 \partial_3 \omega^u\|_{L^{\frac{1}{\frac12+\frac{s}{2}}}(\mathbb{R}^2)}
						\|\Lambda_h^{-s} \omega^u\|_{L^2(\mathbb{R}^2)}dx_3\\
						&\lesssim\int_0^{+\infty}
						\left(\left\|\|\partial_3 u_3\|_{L^\infty(\mathbb{R}_+)}\right\|_{L^{\frac2s}(\mathbb{R}^2)}
						+\left\|\|u_3\|_{L^\infty(\mathbb{R}_+)}\right\|_{L^{\frac2s}(\mathbb{R}^2)}\right)
						\|Z_3 \omega^u\|_{L^2(\mathbb{R}^2)}
						\|\Lambda_h^{-s} \omega^u\|_{L^2(\mathbb{R}^2)}dx_3\\
						&\lesssim\left(\left\|\|\partial_3 u_3\|_{L^\infty(\mathbb{R}_+)}\right\|_{L^{\frac2s}(\mathbb{R}^2)}
						+\left\|\|u_3\|_{L^\infty(\mathbb{R}_+)}\right\|_{L^{\frac2s}(\mathbb{R}^2)}\right)
						\|Z_3 \omega^u\|_{L^2}\|\Lambda_h^{-s} \omega^u\|_{L^2}\\
						&\lesssim\left(\left\|\|\partial_3 u_3\|_{L^\infty(\mathbb{R}_+)}\right\|_{L^{\frac2s}(\mathbb{R}^2)}
						+\left\|\|u_3\|_{L^\infty(\mathbb{R}_+)}\right\|_{L^{\frac2s}(\mathbb{R}^2)}\right)
						\|\Lambda_h^{-s} \omega^u\|_{L^2}\\
						&\quad \quad \times \left(\|\omega^u\|_{L^2}
						+\|\omega^u\|_{L^2}^{\frac34}\|Z_3^3 \omega^u\|_{L^2}^{\frac14}
						+\|\omega^u\|_{L^2}^{\frac23}\|Z_3^3 \omega^u\|_{L^2}^{\frac13} \right)\\
						&\lesssim \|(\omega^u, Z_3^3 \omega^u)\|_{L^2}^{\frac13}
						\mathcal{E}_{tan}^{m-1}(t)^{\frac{7}{12}}
						\md_{tan}^{m-1}(t)^{\frac14}
						\|\Lambda_h^{-s} \omega^u\|_{L^2},
					\end{aligned}
				\end{equation}
				where we have used the estimate
				\begin{equation*}
					\begin{aligned}
						&~~~~\left\|\|\partial_3 u_3\|_{L^\infty(\mathbb{R}_+)}\right\|_{L^{\frac2s}(\mathbb{R}^2)}
						+\left\|\|u_3\|_{L^\infty(\mathbb{R}_+)}\right\|_{L^{\frac2s}(\mathbb{R}^2)}\\
						&\lesssim \left(\|\partial_3 u_3\|_{L^2}\|\partial_{23} u_3\|_{L^2}
						+\|\partial_{13} u_3\|_{L^2}
						\|\partial_{123} u_3\|_{L^2}\right)^{\frac{1-s}{2}}
						\|\partial_3 u_3\|_{L^2}^{\frac{2s-1}{2}}
						\|\partial_{33} u_3\|_{L^2}^{\frac12}\\
						&~~~~+(\|u_3\|_{L^2}\|\partial_2 u_3\|_{L^2}
						+\|\partial_1 u_3\|_{L^2}\|\partial_{12}u_3\|_{L^2})^{\frac{1-s}{2}}
						\|u_3\|_{L^2}^{\frac{2s-1}{2}}
						\|\partial_3 u_3\|_{L^2}^{\frac12}\\
						&\lesssim \mathcal{E}_{tan}^{m-1}(t)^{\frac14}
						\md_{tan}^{m-1}(t)^{\frac14}.
					\end{aligned}
				\end{equation*}
				Substituting estimates \eqref{3904} and \eqref{3905}  into \eqref{3903}, we have
				\beq\label{termk1}
				\begin{aligned}
					K_1
					\lesssim
					\sqrt{\mathcal{E}_{tan}^{m-1}(t)} \sqrt{\md_{tan}^{m-1}(t)}
					\|\Lambda_h^{-s} \omega^u\|_{L^2}
					+\|(\omega^u, Z_3^3 \omega^u)\|_{L^2}^{\frac13}
					\mathcal{E}_{tan}^{m-1}(t)^{\frac{7}{12}}
					\md_{tan}^{m-1}(t)^{\frac14}
					\|\Lambda_h^{-s} \omega^u\|_{L^2}.
				\end{aligned}
				\deq
				Similarly, it is easy to check that
				\beqq
				\begin{aligned}
					K_3
					\lesssim
					&\sqrt{\mathcal{E}_{tan}^{m-1}(t)} \sqrt{\md_{tan}^{m-1}(t)}
					\|\Lambda_h^{-s} \omega^u\|_{L^2}
					+\|(\omega^b_h, Z_3^3 \omega^b)\|_{L^2}^{\frac13}
					\mathcal{E}_{tan}^{m-1}(t)^{\frac{7}{12}}
					\md_{tan}^{m-1}(t)^{\frac14}
					\|\Lambda_h^{-s} \omega^u\|_{L^2},\\
					K_5
					\lesssim
					&\sqrt{\mathcal{E}_{tan}^{m-1}(t)} \sqrt{\md_{tan}^{m-1}(t)}
					\|\Lambda_h^{-s} \omega^b\|_{L^2}
					+\|(\omega^b_h, Z_3^3 \omega^b)\|_{L^2}^{\frac13}
					\mathcal{E}_{tan}^{m-1}(t)^{\frac{7}{12}}
					\md_{tan}^{m-1}(t)^{\frac14}
					\|\Lambda_h^{-s} \omega^b\|_{L^2},\\
					K_7
					\lesssim
					&\sqrt{\mathcal{E}_{tan}^{m-1}(t)} \sqrt{\md_{tan}^{m-1}(t)}
					\|\Lambda_h^{-s} \omega^b\|_{L^2}
					+\|(\omega^u, Z_3^3 \omega^u)\|_{L^2}^{\frac13}
					\mathcal{E}_{tan}^{m-1}(t)^{\frac{7}{12}}
					\md_{tan}^{m-1}(t)^{\frac14}
					\|\Lambda_h^{-s} \omega^b\|_{L^2}.
				\end{aligned}
				\deqq
				The H\"{o}lder's inequality and inequality \eqref{a16} in Lemma \ref{H-L} yield directly
				\beqq
				\begin{aligned}
					K_2
					=&\int_{\mathbb{R}_+^3} \Lambda_h^{-s}(w^u\cdot \nabla u)\cdot \Lambda_h^{-s} \omega^u dx\\
					\lesssim
					&\int_0^{+\infty}
					(\|\omega^u_h \cdot \partial_h u\|_{L^{\frac{1}{\frac12+\frac{s}{2}}}(\mathbb{R}^2)}
					+\|\omega_3^u \partial_3 u\|_{L^{\frac{1}{\frac12+\frac{s}{2}}}(\mathbb{R}^2)})
					\|\Lambda_h^{-s} \omega^u\|_{L^2(\mathbb{R}^2)}dx_3\\
					\lesssim
					&
					(\|\partial_h u\|_{L^2}\|\partial_2 \partial_h u\|_{L^2}
					+\|\partial_1 \partial_h u\|_{L^2}
					\|\partial_{12}\partial_h u\|_{L^2})^{\frac{1-s}{2}}
					\|\partial_h u\|_{L^2}^{\frac{2s-1}{2}}
					\|\partial_3 \partial_h u\|_{L^2}^{\frac12}
					\|\omega^u_h\|_{L^2} \|\Lambda_h^{-s} \omega^u\|_{L^2}\\
					&+ (\|\omega_3^u\|_{L^2}\|\partial_2 \omega_3^u\|_{L^2}
					+\|\partial_1 \omega_3^u\|_{L^2}\|\partial_{12}\omega_3^u\|_{L^2})^{\frac{1-s}{2}}
					\|\omega_3^u\|_{L^2}^{\frac{2s-1}{2}}
					\|\partial_3 \omega_3^u\|_{L^2}^{\frac12}
					\|\partial_3 u\|_{L^2}
					\|\Lambda_h^{-s} \omega^u\|_{L^2}\\
					\lesssim
					&  \sqrt{\mathcal{E}_{tan}^{m-1}(t)} \sqrt{\md_{tan}^{m-1}(t)}
					\|\Lambda_h^{-s} \omega^u\|_{L^2}.
				\end{aligned}
				\deqq
				Similarly, it is easy to check that
				\beqq
				K_4
				\lesssim
				\sqrt{\mathcal{E}_{tan}^{m-1}(t)} \sqrt{\md_{tan}^{m-1}(t)}
				\|\Lambda_h^{-s} \omega^u\|_{L^2}.
				\deqq
				Finally, due to the relations \eqref{3306}-\eqref{3308}, we may write
				\beqq
				\begin{aligned}
					K_6
					\le
					&\int_0^{+\infty}\|\Lambda_h^{-s}(\nabla(u\cdot \nabla)\times b)\|_{L^2(\mathbb{R}^2)}
					\|\Lambda_h^{-s} \wb\|_{L^2(\mathbb{R}^2)} dx_3\\
					\lesssim
					&\int_0^{+\infty}\|(\nabla(u\cdot \nabla)\times b)
					\|_{L^{\frac{1}{\frac12+\frac{s}{2}}}(\mathbb{R}^2)})
					\|\Lambda_h^{-s} \wb\|_{L^2(\mathbb{R}^2)} dx_3\\
					\lesssim
					&\int_0^{+\infty}\|\nabla_h u\|_{L^{\frac{2}{s}}(\mathbb{R}^2)}
					\|\nabla_h b\|_{L^2(\mathbb{R}^2)}
					\|\Lambda_h^{-s} \wb\|_{L^2(\mathbb{R}^2)} dx_3\\
					&+\int_0^{+\infty}\|\nabla_h (u, b)\|_{L^{\frac{2}{s}}(\mathbb{R}^2)}
					\|\p_3 (u_h, b_h)\|_{L^2(\mathbb{R}^2)}
					\|\Lambda_h^{-s} \wb\|_{L^2(\mathbb{R}^2)} dx_3\\
					\lesssim
					&\sqrt{\mathcal{E}_{tan}^{m-1}(t)} \sqrt{\md_{tan}^{m-1}(t)}
					\|\Lambda_h^{-s} \wb\|_{L^2}.
				\end{aligned}
				\deqq
				Similarly, it is easy to check that
				\beqq
				K_8
				\lesssim
				\sqrt{\mathcal{E}_{tan}^{m-1}(t)} \sqrt{\md_{tan}^{m-1}(t)}
				\|\Lambda_h^{-s} \wb\|_{L^2}.
				\deqq
				Substituting the estimates of terms for  $K_1$
				through $K_8$ into \eqref{3902}, we have
				\beq\label{3906}
				\begin{aligned}
					&\frac{1}{2}\|(\Lambda_h^{-s}\wu,\Lambda_h^{-s} \wb)(t)\|_{L^2}^2
					+\int_0^t \|(\p_1 \Lambda_h^{-s} \wu, \nabla_h \Lambda_h^{-s} \wb)\|_{L^2}^2 d\tau\\
					&+\ep \int_0^t  \|(\p_2 \Lambda_h^{-s} \wu, \p_3 \Lambda_h^{-s} \wu,
					\p_3 \Lambda_h^{-s} \wb)\|_{L^2}^2d\tau\\
					\lesssim
					&\frac{1}{2}\|(\Lambda_h^{-s}\wu,\Lambda_h^{-s} \wb)(0)\|_{L^2}^2
					+\underset{0\le \tau \le t}{\sup}\|(\Lambda_h^{-s} \omega^u,
					\Lambda_h^{-s} \omega^b)(\tau)\|_{L^2}
					\int_0^t
					\sqrt{\mathcal{E}_{tan}^{m-1}(\tau)}
					\sqrt{\md_{tan}^{m-1}(\tau)}d\tau\\
					&+\underset{0\le \tau \le t}{\sup}
					\|(\Lambda_h^{-s} \omega^u,\Lambda_h^{-s} \omega^b)(\tau)\|_{L^2}
					\underset{0\le \tau \le t}{\sup}
					\|(\omega^u, \omega^b, Z_3^3 \omega^u, Z_3^3 \omega^b)\|_{L^2}^{\frac13}
					\int_0^t  \mathcal{E}_{tan}^{m-1}(\tau)^{\frac{7}{12}}
					\md_{tan}^{m-1}(\tau)^{\frac14} d \tau.
				\end{aligned}
				\deq
				Using the assumption \eqref{assumption}, it is easy to check that
				\begin{equation}\label{3907}
					\begin{aligned}
						\int_0^t \mathcal{E}_{tan}^{m-1}(\tau)^{\frac{7}{12}}
						\md_{tan}^{m-1}(\tau)^{\frac14} d\tau
						&\lesssim
						\bigg\{\int_0^t \md_{tan}^{m-1}(\tau)
						(1+\tau)^{\sigma}d\tau\bigg\}^{\frac{1}{4}}
						\bigg\{\int_0^t \mathcal{E}_{tan}^{m-1}(\tau)^{\frac{7}{9}}(1+\tau)^{-\frac13 \sigma}
						d\tau\bigg\}^{\frac{3}{4}}\\
						&\lesssim
						\delta^{\frac56}\bigg\{\int_0^t (1+\tau)^{-(\frac{7}{9}s+\frac13\sigma)}d\tau
						\bigg\}^{\frac{1}{2}}
						\lesssim \delta^{\frac56},
					\end{aligned}
				\end{equation}
				where we have used the condition $\frac{9}{10}<\sigma<s<1$.
				Substituting \eqref{3907} and \eqref{3803} into \eqref{3906}, then we have
				\beqq
				\begin{aligned}
					&\|(\Lambda_h^{-s}\wu,\Lambda_h^{-s} \wb)(t)\|_{L^2}^2
					+\int_0^t \|(\p_1 \Lambda_h^{-s} \wu, \nabla_h \Lambda_h^{-s} \wb)\|_{L^2}^2 d\tau\\
					&+\ep \int_0^t  \|(\p_2 \Lambda_h^{-s} \wu, \p_3 \Lambda_h^{-s} \wu,
					\p_3 \Lambda_h^{-s} \wb)\|_{L^2}^2d\tau\\
					\lesssim
					&\|(\Lambda_h^{-s}\wu,\Lambda_h^{-s} \wb)(0)\|_{L^2}^2+\delta^2.
				\end{aligned}
				\deqq
				Therefore, we complete the proof of this lemma.
			\end{proof}

			\subsection{Decay in time estimate of velocity and magnetic field}
			In this subsection, we will establish the decay rate estimate
			for velocity and magnetic field. This decay in time estimate
			will help us to close the energy estimate.

			\begin{lemm}
				Under the assumption \eqref{assumption},
				the smooth solution $(u, b)$ of equation \eqref{eqr} has the estimate
				\beq\label{31001}
				(1+t)^s \mathcal{E}^{m-1}_{tan}(t) \lesssim C_0,
				\deq
				where the constant $C_0$ is defined in \eqref{co}.
			\end{lemm}
			\begin{proof}
				The combination of estimates \eqref{3101}, \eqref{3201},
				\eqref{3301}, \eqref{3423}, \eqref{3501}, \eqref{3613} and \eqref{3701} yields directly
				\beq\label{31008-1}
				\mathcal{E}^m(t)+\int_0^t \md^{m}(\tau)  d\tau
				+\int_0^t \md_{\ep}^{m}(\tau)  d\tau
				\le C\mathcal{E}^m(0)+C\delta^{\frac32},
				\deq
				where the dissipation norm $\md_{\ep}^{m}(t)$ is defined by
				\beq
				\begin{aligned}
					\md_{\ep}^{m}(t)
					:=&\ep \|(\p_2 u, \p_3 u, \p_3 b)(t)\|_{H^m_{tan}}^2
					+\ep\|(\p_2 \wu, \p_3 \wu, \p_3 \wb)(t)\|_{H^{m-1}_{tan}}^2\\
					&+\ep \|\p_3 (\p_2  \wu, \p_3 \wu, \p_3 \wb)(t)\|_{H^{m-2}_{co}}^2.
				\end{aligned}
				\deq
				The combination of estimates \eqref{3801} and \eqref{3901} yields directly
				\beqq
				\begin{aligned}
					\|(\Lambda_h^{-s}u, \Lambda_h^{-s} b,\Lambda_h^{-s}\wu,\Lambda_h^{-s} \wb)(t)\|_{L^2}^2
					\le
					C(\|(\Lambda_h^{-s}u, \Lambda_h^{-s} b,
					\Lambda_h^{-s}\wu,\Lambda_h^{-s} \wb)(0)\|_{L^2}^2+\delta^2),
				\end{aligned}
				\deqq
				which, together with estimate \eqref{31008-1}, yields directly
				\beq\label{31009}
				\mathcal{E}^m(t)
				+\|(\Lambda_h^{-s}u, \Lambda_h^{-s} b,\Lambda_h^{-s}\wu,\Lambda_h^{-s} \wb)(t)\|_{L^2}^2
				\le CC_0,
				\deq
				where the constant $C_0$ is defined by
				\beq\label{co}
				C_0:=\mathcal{E}^m(0)
				+\|(\Lambda_h^{-s}u, \Lambda_h^{-s} b,
				\Lambda_h^{-s}\wu,\Lambda_h^{-s} \wb)(0)\|_{L^2}^2+\delta^{\f32}.
				\deq
				
				From estimates \eqref{3101} and \eqref{3201}, we have
				\beq\label{31002}
				\frac{d}{dt}\|(u, b)(t)\|_{H^{m-1}_{tan}}^2
				+\|(\p_1 u, \nabla_h b)(t)\|_{H^{m-1}_{tan}}^2
				+\ep \|(\p_2 u, \p_3 u, \p_3 b)(t)\|_{H^{m-1}_{tan}}^2
				\lesssim  \sqrt{\me^m(t)}\md_{tan}^{{m-1}}(t),
				\deq
				and
				\beq\label{31003}
				\frac{d}{dt}\|(\wu, \wb)\|_{H^{m-2}_{tan}}^2
				+\|(\p_1 \wu, \nabla_h \wb)\|_{H^{m-2}_{tan}}^2
				+\ep\|(\p_2 \wu, \p_3 \wu, \p_3 \wb)\|_{H^{m-2}_{tan}}^2
				\lesssim \sqrt{\me^m(t)}\md_{tan}^{m-1}(t),
				\deq
				From the estimates \eqref{3501} and \eqref{3601}, we have
				\beq\label{31004}
				\begin{aligned}
					&\frac{d}{dt}
					\sum_{0\le |\alpha_h| \le m-2}
					\i \p_2 Z^{\ah}b \cdot  Z^{\ah}u \ dx
					+\|\p_2 u\|_{H^{m-2}_{tan}}^2 \\
					\lesssim
					&\|\nabla_h b\|_{H^{m-1}_{tan}}^2+
					\|\p_1  u\|_{H^{m-2}_{tan}}^2
					+\ep \|(\p_3 b, \p_3 u)\|_{H^{m-1}_{tan}}^2
					+\sqrt{\me^m(t)}\md_{tan}^{m-1}(t),
				\end{aligned}
				\deq
				and
				\beq\label{31005}
				\begin{aligned}
					&\frac{d}{dt}\sum_{0\le |\alpha_h| \le m-3}
					\i \p_2 Z^{\ah} \wb \cdot Z^{\ah} \wu dx
					+\|\p_2 \wu\|_{H^{m-3}_{tan}}^2 \\
					\lesssim
					&\|(\p_1 \wu, \nabla_h  \wb)\|_{H^{m-2}_{tan}}^2
					+\|\p_1^2 \wu\|_{H^{m-3}_{tan}}^2
					+\ep \|(\p_3 \wb, \p_3\wu)\|_{H^{m-2}_{tan}}^2
					+\ep \| \p_2^2 \wu\|_{H^{m-3}_{tan}}^2
					+\sqrt{\me^m(t)}\md_{tan}^{m-1}(t).
				\end{aligned}
				\deq
				Then, for some small positive suitable constant $\kappa>0$,
				we can deduce from estimates \eqref{31002}-\eqref{31005} that
				\beq\label{31010}
				\begin{aligned}
					&\frac{d}{dt}\widetilde{\mathcal{E}}^{m-1}_{tan}(t)
					+2\kappa \md_{tan}^{m-1}(t)
					+\ep \|(\p_2 u, \p_3 u, \p_3 b)(t)\|_{H^{m-1}_{tan}}^2\\
					&+\ep\|(\p_2 \wu, \p_3 \wu, \p_3 \wb)\|_{H^{m-2}_{tan}}^2
					\lesssim \sqrt{\me^m(t)}\md_{tan}^{m-1}(t),
				\end{aligned}
				\deq
				where the quantity $\widetilde{\mathcal{E}}^{m-1}_{tan}(t)$ is defined as
				\beqq
				\begin{aligned}
					\widetilde{\mathcal{E}}^{m-1}_{tan}(t)
					:=&\|(u, b)(t)\|_{H^{m-1}_{tan}}^2+\|(\wu, \wb)(t)\|_{H^{m-2}_{tan}}^2
					+\sum_{0\le |\alpha_h| \le m-2} 2\kappa  \i \p_2 Z^{\ah}b \cdot  Z^{\ah}u \ dx \\
					&+\sum_{0\le |\alpha_h| \le m-3} 2\kappa \i \p_2 Z^{\ah} \wb \cdot Z^{\ah} \wu dx.
				\end{aligned}
				\deqq
				Due to the smallness of $\kappa$, it is easy to check that
				$\widetilde{\mathcal{E}}^{m-1}_{tan}(t)$ is equivalent to $\mathcal{E}^{m-1}_{tan}(t)$.
				Thus, due to the a priori assumption \eqref{assumption}, we have
				\beq\label{31006}
				\frac{d}{dt}\widetilde{\mathcal{E}}^{m-1}_{tan}(t)
				+\kappa \md_{tan}^{m-1}(t)
				+\ep (\|(\p_2 u, \p_3 u, \p_3 b)(t)\|_{H^{m-1}_{tan}}^2
				+\|(\p_2 \wu, \p_3 \wu, \p_3 \wb)\|_{H^{m-2}_{tan}}^2)
				\le 0.
				\deq
				On the other hand, due to the inequality
				\beqq
				\|(u, b, \wu, \wb)\|_{L^2}
				\lesssim \|\Lambda_h^{-s}(u, b, \wu, \wb)\|_{L^2}^{\frac{1}{1+s}}
				\|\nabla_h(u, b, \wu, \wb)\|_{L^2}^{\frac{s}{1+s}}
				\deqq
				it is easy to check that
				\beq\label{31007}
				\begin{aligned}
					\widetilde{\mathcal{E}}^{m-1}_{tan}(t)
					\lesssim
					\me_{tan}^{m-1}(t)
					\lesssim
					&(\|\Lambda_h^{-s}(u, b, \wu, \wb)\|_{L^2}^2
					+\|\nabla_h(u, b)\|_{H^{m-2}_{tan}}^2
					+\|\nabla_h(\wu, \wb)\|_{H^{m-3}_{tan}}^2)^{\frac{1}{1+s}}\\
					& \times(\|\nabla_h(u,b)\|_{H^{m-2}_{tan}}^2
					+\|\nabla_h(\wu,\wb)\|_{H^{m-3}_{tan}}^2)^{\frac{s}{1+s}}\\
					\lesssim
					&C_0^{\frac{1}{1+s}} \md_{tan}^{m-1}(t)^{\frac{s}{1+s}},
				\end{aligned}
				\deq
				where we have used the estimate \eqref{31009} in the last inequality.
				The combination of \eqref{31006} and \eqref{31007} yields
				\beqq
				\frac{d}{dt}\widetilde{\mathcal{E}}^{m-1}_{tan}(t)
				+\kappa C_0^{-\frac{1}{s}}
				\widetilde{\mathcal{E}}^{m-1}_{tan}(t)^{1+\frac{1}{s}}\le 0,
				\deqq
				which yields directly the decay estimate
				\beq\label{31008}
				\mathcal{E}^{m-1}_{tan}(t)
				\lesssim
				\widetilde{\mathcal{E}}^{m-1}_{tan}(t)
				\lesssim C_0(1+t)^{-s}.
				\deq
				Therefore, we complete the proof of this lemma.
			\end{proof}
			
			Finally, we will establish the time integration
			of  $\md_{tan}^{m-1}(t)$ with the suitable weight $(1+\tau)^{\sigma}$.
			\begin{lemm}
				Under the assumption \eqref{assumption},
				the smooth solution $(u, b)$ of equation \eqref{eqr} has the estimate
				\beq\label{31101}
				\begin{aligned}
					&(1+t)^{\sigma}{\mathcal{E}}^{m-1}_{tan}(t)
					+\kappa \int_0^t (1+\tau)^{\sigma} \md_{tan}^{m-1}(\tau)d\tau
					+\ep \int_0^t(1+\tau)^{\sigma}\|(\p_2 u, \p_3 u, \p_3 b)(\tau)\|_{H^{m-1}_{\tan}}^2 d\tau\\
					&+\ep \int_0^t (1+\tau)^{\sigma}\|(\p_2 \wu, \p_3 \wu, \p_3 \wb)(\tau)\|_{H^{m-2}_{tan}}^2 d\tau
					\lesssim C_0.
				\end{aligned}
				\deq
				where the constant $C_0$ is defined in \eqref{co}.
			\end{lemm}
			\begin{proof}
				For $0<\sigma<s$, multiplying \eqref{31006} by $(1+t)^{\sigma}$, we have
				\beqq
				\begin{aligned}
					&\frac{d}{dt}[(1+t)^{\sigma}\widetilde{\mathcal{E}}^{m-1}_{tan}(t)]
					+\kappa (1+t)^{\sigma} \md_{tan}^{m-1}(t)
					+\ep (1+t)^{\sigma}\|(\p_2 u, \p_3 u, \p_3 b)(t)\|_{H^{m-1}_{\tan}}^2\\
					&+\ep (1+t)^{\sigma}\|(\p_2 \wu, \p_3 \wu, \p_3 \wb)(t)\|_{H^{m-2}_{tan}}^2
					\le \sigma(1+t)^{\sigma-1}\widetilde{\mathcal{E}}^{m-1}_{tan}(t).
				\end{aligned}
				\deqq
				Integrating the above inequality over $[0, t]$
				and using the uniform estimate \eqref{31008}, we have
				\beqq
				\begin{aligned}
					&(1+t)^{\sigma}\widetilde{\mathcal{E}}^{m-1}_{tan}(t)
					+\kappa \int_0^t  (1+\tau)^{\sigma} \md_{tan}^{m-1}(\tau)d\tau
					+\ep \int_0^t(1+\tau)^{\sigma}\|(\p_2 u, \p_3 u, \p_3 b)(\tau)\|_{H^{m-1}_{\tan}}^2 d\tau\\
					&+\ep \int_0^t (1+\tau)^{\sigma}\|(\p_2 \wu, \p_3 \wu, \p_3 \wb)(\tau)\|_{H^{m-2}_{tan}}^2 d\tau\\
					\le
					&\widetilde{\mathcal{E}}^{m-1}_{tan}(0)
					+\sigma\int_0^t (1+\tau)^{\sigma-1}
					\widetilde{\mathcal{E}}^{m-1}_{tan}(\tau)d\tau\\
					\le
					&\widetilde{\mathcal{E}}^{m-1}_{tan}(0)
					+\sigma\underset{0\le \tau \le t}{\sup}
					\left[\widetilde{\mathcal{E}}^{m-1}_{tan}(\tau)(1+\tau)^s\right]
					\int_0^t (1+\tau)^{\sigma-1-s}d\tau\\
					\le
					&\widetilde{\mathcal{E}}^{m-1}_{tan}(0)
					+\frac{\sigma}{s-\sigma}
					\left[1-(1+t)^{\sigma-s}\right]
					\underset{0\le \tau \le t}{\sup}
					\left[\widetilde{\mathcal{E}}^{m-1}_{tan}(\tau)(1+\tau)^s\right]\\
					\lesssim
					&C_0.
				\end{aligned}
				\deqq
				Therefore, we complete the proof of this lemma.
			\end{proof}
			
			\subsection{Global in time uniform regularity}
			
			In this subsection, we will give the proof of Proposition \ref{main_pro}.
			Indeed,  the combination of estimates \eqref{31001},\eqref{31008-1} and \eqref{31101}
			yields directly
			\beqq
			(1+t)^{s}{\mathcal{E}}^{m-1}_{tan}(t)
			+\int_0^t (1+\tau)^{\sigma} \md_{tan}^{m-1}(\tau)d\tau
			\le C(\mathcal{E}^m(0)
			+\|(\Lambda_h^{-s}u, \Lambda_h^{-s} b,
			\Lambda_h^{-s}\wu,\Lambda_h^{-s} \wb)(0)\|_{L^2}^2+\delta^{\f32}),
			\deqq
			and
			\beqq
			\mathcal{E}^m(t)+\int_0^t \md^{m}(\tau)  d\tau
			\le C\mathcal{E}^m(0)+C\delta^{\frac32},
			\deqq
			where the constants $(s, \sigma)$ satisfy $\frac{9}{10}<\sigma<s<1$.
			Then, we can obtain the estimate
			\beqq\label{31201}
			\begin{aligned}
				&\mathcal{E}^m(t)+(1+t)^{s}{\mathcal{E}}^{m-1}_{tan}(t)
				+\int_0^t (1+\tau)^{\sigma} \md_{tan}^{m-1}(\tau)d\tau
				+\int_0^t \md^{m}(\tau)  d\tau\\
				\le
				&2C(\mathcal{E}^m(0)
				+\|(\Lambda_h^{-s}u, \Lambda_h^{-s} b,
				\Lambda_h^{-s}\wu,\Lambda_h^{-s} \wb)(0)\|_{L^2}^2)
				+2C\delta^{\frac32}.
			\end{aligned}
			\deqq
			Now choose the small constant
			\beq\label{choose_delta}
			\delta:=8C(\mathcal{E}^m(0)
			+\|(\Lambda_h^{-s}u, \Lambda_h^{-s} b,
			\Lambda_h^{-s}\wu,\Lambda_h^{-s} \wb)(0)\|_{L^2}^2)\le \min\{1,\frac{1}{64C^2}\},
			\deq
			then we have
			\beqq\label{31201}
			\begin{aligned}
				&\mathcal{E}^m(t)+(1+t)^{s}{\mathcal{E}}^{m-1}_{tan}(t)
				+\int_0^t (1+\tau)^{\sigma} \md_{tan}^{m-1}(\tau)d\tau
				+\int_0^t \md^{m}(\tau)  d\tau\\
				\le
				&2C(\mathcal{E}^m(0)
				+\|(\Lambda_h^{-s}u, \Lambda_h^{-s} b,
				\Lambda_h^{-s}\wu,\Lambda_h^{-s} \wb)(0)\|_{L^2}^2)
				+2C\delta^{\frac32}\\
				\le
				&\frac{\delta}{4}+\frac{\delta}{4}=\frac{\delta}{2},
			\end{aligned}
			\deqq
			which implies the estimate \eqref{close_assumption}.
			Therefore, we complete the proof of Proposition \ref{main_pro}.

			\begin{proof}[\textbf{Proof of Theorem \ref{main_result_one}}]
				Suppose the assumptions in Theorem \ref{main_result_one} hold,
				similar to the result in \cite{MR3472518}, one can establish the uniform
				(with respect to $\var$) local-in-time well-posedness for the equation \eqref{eqr}.
				Next, we use the standard continuity argument to show the global well-posedness.
				From the local existence result
				and smallness assumption of initial condition,
				it holds
				\beqq
				\mathcal{E}^m(t)
				+(1+t)^s\mathcal{E}_{tan}^{m-1}(t)
				+\int_0^t (1+\tau)^{\sigma} \md_{tan}^{m-1}(\tau)d\tau
				+\int_0^t \md^{m}(\tau)d\tau \le \delta,
				\deqq
				for all $t\in [0, T_0)$ and $\delta=8C(\mathcal{E}^m(0)
				+\|(\Lambda_h^{-s}u, \Lambda_h^{-s} b,
				\Lambda_h^{-s}\wu,\Lambda_h^{-s} \wb)(0)\|_{L^2}^2)$
				and $C$ is a positive constant independent of time $t$ and
				parameter $\ep$.
				Set
				\beq\label{criterion-half}
				T^*:=\underset{T_0}{\sup}
				\left\{T_0~|
				\mathcal{E}^m(t)
				+(1+t)^s\mathcal{E}_{tan}^{m-1}(t)
				+\int_0^t (1+\tau)^{\sigma} \md_{tan}^{m-1}(\tau)d\tau
				+\int_0^t \md^{m}(\tau)d\tau \le \delta,
				\quad \forall ~ t\in [0, T_0)\right\},
				\deq
				we claim that $T^*=+\infty$. Otherwise, applying the estimate
				\eqref{close_assumption}
				and the local-in-time existence result,
				there exists a positive constant $T^{**}$ such that $T^{**}>T^{*}$,
				it holds that for any $T\in [T^{*}, T^{**})$,
				\beqq
				\mathcal{E}^m(t)
				+(1+t)^s\mathcal{E}_{tan}^{m-1}(t)
				+\int_0^t (1+\tau)^{\sigma} \md_{tan}^{m-1}(\tau)d\tau
				+\int_0^t \md^{m}(\tau)d\tau \le \delta,
				\quad \forall ~ t\in [0, T).
				\deqq
				This contradicts the definition of $T^*$ in \eqref{criterion-half}.
				Therefore, we can deduce that $T^*=+\infty$.
				Therefore, we complete the proof of Theorem \ref{main_result_one}.			
			\end{proof}
			
			\section{Convergence rate of solution}\label{asymptotic-behavior}
			In this section, we will establish the convergence rate
			of the solutions between  equations \eqref{eqr} and \eqref{eqr0}.
			Let the regularity index $m \ge 5$, the Theorem \ref{main_result_one}
			make sure that the solution $(u^\ep, b^\ep)$
			of equation \eqref{eqr} exists globally in time. If the initial data
			$\|\Lambda_h^{-s}\p_3^2(u_0, b_0)\|_{L^2}^2$ is small enough,
			we can establish suitable decay rate for $\p_3^2(u^\ep, b^\ep)$.
			This will help us establish the convergence rate independent of time.
			
			\subsection{Decay rate of second order normal derivative}
			
	For notational convenience, we drop the superscript $\var$ throughout this subsection.
			Let us define the new energy and dissipation norms
			\beq
			\begin{aligned}
				\widehat{\mathcal{E}}_{tan}^{3}(t)
				&:=\mathcal{E}_{tan}^{3}(t)+\|\p_3(\wu, \wb)(t)\|_{H^{1}_{tan}}^2,\\
				\widehat{\mathcal{D}}_{tan}^{3}(t)
				&:=\mathcal{D}_{tan}^{3}(t)+\|\p_{23}\wu(t)\|_{L^2}^2
				+\|(\p_{13}\wu, \nabla_h \p_{3}\wb)(t)\|_{H^1_{tan}}^2.
			\end{aligned}
			\deq
			For any small positive constant $\delta>0$, let us assume
			\begin{equation}\label{decay-assumption-new}
				(1+t)^s \widehat{\mathcal{E}}_{tan}^{3}(t)
				+\int_0^t (1+\tau)^{\sigma}
				\widehat{\mathcal{D}}_{tan}^{3}(\tau)d\tau
				\le \delta,
			\end{equation}
			for all $t \in (0, T]$.
			Then, we will establish some energy estimates under the assumption \eqref{decay-assumption-new}.

			\begin{lemm}
				For any smooth solution $(u, b)$ of equation \eqref{eqr},
				it holds
				\beq\label{4101}
				\frac{d}{dt}\|(\p_3 \wu, \p_3 \wb)\|_{H^1_{tan}}^2
				+\|(\p_{13} \wu, \nabla_h \p_3 \wb)\|_{H^1_{tan}}^2
				+\ep \|(\p_{23} \wu, \p_3^2 \wu, \p_3^2 \wb)\|_{H^1_{tan}}^2
				\lesssim \sqrt{\mathcal{E}^m(t)}\widehat{\mathcal{D}}_{tan}^{3}(t).
				\deq
			\end{lemm}
			\begin{proof}
				First of all, due to the analysis of estimate \eqref{3301}, we can obtain
				\beq
				\frac{d}{dt}\|(\p_3 \wu, \p_3 \wb)\|_{L^2}^2
				+\|(\p_{13} \wu, \nabla_h \p_3 \wb)\|_{L^2}^2
				+\ep \|(\p_{23} \wu, \p_{33} \wu, \p_{33} \wb)\|_{L^2}^2
				\lesssim \sqrt{\mathcal{E}^m(t)}\widehat{\mathcal{D}}_{tan}^{3}(t).
				\deq
				For $i=1,2$, similar to the equality \eqref{3302}, we have
				\beq\label{41022}
				\begin{aligned}
					&\frac{d}{dt}\frac{1}{2}\i (|\p_3 \p_ i \wu|^2+|\p_3 \p_ i \wb|^2) dx
					+\i (|\p_{13} \p_ i \wu|^2+|\nabla_h \p_3 \p_ i\wb|^2) dx\\
					&+\ep \i(|\p_{23} \p_ i \wu|^2+|\p_{33}\p_ i  \wu|^2+|\p_{33} \p_ i \wb|^2)dx\\
					=&\i \p_3\p_ i (-u\cdot \nabla \wu+\wu\cdot \nabla u+b \cdot \nabla \wb-\wb\cdot \nabla b)
					\cdot \p_3 \p_ i \wu dx\\
					&+\i \p_3\p_ i  (-u\cdot \nabla \wb-\nabla(u\cdot \nabla)\times b
					+b\cdot \nabla \wu+\nabla(b\cdot \nabla)\times u)
					\cdot \p_3 \p_ i \wb dx\\
				\end{aligned}
				\deq
				Integrating by part, it is easy to check that
				\beq\label{41033}
				\begin{aligned}
					&\i \p_3\p_ i (u\cdot \nabla \wu)\cdot \p_3 \p_ i \wu dx\\
					=&\i ( \p_3 \p_ i u\cdot \nabla \wu
					+\p_3 u\cdot \nabla \p_ i\wu
					+\p_ i u_h \cdot \nabla_h \p_3\wu
					+\p_ i u_3\p_{33}\wu)\cdot \p_3 \p_ i \wu dx,
				\end{aligned}
				\deq
				where we have used the relation
				\beqq
				\i (u\cdot \nabla) \p_3\p_ i \wu \cdot \p_3 \p_ i \wu dx
				=-\frac12 \i |\p_3 \p_ i \wu|^2 {\nabla \cdot }u\  dx=0.
				\deqq
				Using the anisotropic type inequality \eqref{ie:Sobolev},
				then we conclude
				\beq\label{4104}
				\begin{aligned}
					&\i ( \p_3 \p_ i u\cdot \nabla \wu
					+\p_3 u\cdot \nabla \p_ i\wu
					+\p_ i u_h \cdot \nabla_h \p_3\wu)\cdot \p_3 \p_ i \wu dx\\
					\lesssim
					&\|\p_3 \p_ i u\|_{L^2}^{\frac12}\|\p_{33} \p_ i u\|_{L^2}^{\frac12}
					\|\nabla \wu\|_{L^2}^{\frac12}\|\p_2 \nabla \wu\|_{L^2}^{\frac12}
					\|\p_3 \p_ i \wu\|_{L^2}^{\frac12}\|\p_{13} \p_ i \wu\|_{L^2}^{\frac12}\\
					&+\|\nabla \p_ i\wu\|_{L^2}
					\|\p_3 u\|_{L^2}^{\frac14}\|\p_{23} u\|_{L^2}^{\frac14}
					\|\p_{33} u\|_{L^2}^{\frac14}\|\p_{233} u\|_{L^2}^{\frac14}
					\|\p_3 \p_ i \wu\|_{L^2}^{\frac12}\|\p_{13} \p_ i \wu\|_{L^2}^{\frac12}\\
					&+\|\nabla_h \p_3\wu\|_{L^2}
					\|\p_ i u_h\|_{L^2}^{\frac14}\|\p_3 \p_ i u_h\|_{L^2}^{\frac14}
					\|\p_2 \p_ i u_h\|_{L^2}^{\frac14}\|\p_{23}\p_ i u_h\|_{L^2}^{\frac14}
					\|\p_3 \p_ i \wu\|_{L^2}^{\frac12}\|\p_{13} \p_ i \wu\|_{L^2}^{\frac12}\\
					\lesssim
					&\sqrt{\mathcal{E}^m(t)}\widehat{\mathcal{D}}_{tan}^{3}(t).
				\end{aligned}
				\deq
				Similarly, it is easy to check that
				\beq\label{4105}
				\begin{aligned}
					&\i \p_i u_3\p_{33}\wu \cdot \p_3 \p_ i \wu dx\\
					\lesssim
					&\|Z^3 \p_{3}\wu\|_{L^2}
					\|\p_i u_3\|_{L^2}^{\frac14}\|\p_2 \p_i u_3\|_{L^2}^{\frac14}
					\|\p_3 \p_i u_3\|_{L^2}^{\frac14}\|\p_{23}\p_i u_3\|_{L^2}^{\frac14}
					\|\p_3 \p_ i \wu\|_{L^2}^{\frac12}\|\p_{13} \p_ i \wu\|_{L^2}^{\frac12}\\
					&+\|Z^3 \p_{3}\wu\|_{L^2}
					\|\p_i \p_3 u_3\|_{L^2}^{\frac14}\|\p_{23} \p_i u_3\|_{L^2}^{\frac14}
					\|\p_{33} \p_i u_3\|_{L^2}^{\frac14}\|\p_{233}\p_i u_3\|_{L^2}^{\frac14}
					\|\p_3 \p_ i \wu\|_{L^2}^{\frac12}\|\p_{13} \p_ i \wu\|_{L^2}^{\frac12}\\
					\lesssim
					&\sqrt{\mathcal{E}^m(t)}\widehat{\mathcal{D}}_{tan}^{3}(t).
				\end{aligned}
				\deq
				Substituting the estimates \eqref{4104} and \eqref{4105} into \eqref{41033}, we have
				\beq\label{4106}
				\i \p_3\p_ i (u\cdot \nabla \wu)\cdot \p_3 \p_ i \wu dx
				\lesssim \sqrt{\mathcal{E}^m(t)}\widehat{\mathcal{D}}_{tan}^{3}(t).
				\deq
				Similarly, the other terms on the right hand side  of
				\eqref{41022} can be estimated by the same method, thus we have
				\beq\label{4107}
				\begin{aligned}
					&\i \p_3\p_ i (\wu\cdot \nabla u+b \cdot \nabla \wb-\wb\cdot \nabla b)
					\cdot \p_3 \p_ i \wu dx
					\lesssim \sqrt{\mathcal{E}^m(t)}\widehat{\mathcal{D}}_{tan}^{3}(t),\\
					&\i \p_3\p_ i  (-u\cdot \nabla \wb-\nabla(u\cdot \nabla)\times b
					+b\cdot \nabla \wu+\nabla(b\cdot \nabla)\times u)
					\cdot \p_3 \p_ i \wb dx
					\lesssim
					\sqrt{\mathcal{E}^m(t)}\widehat{\mathcal{D}}_{tan}^{3}(t).
				\end{aligned}
				\deq
				Substituting the estimates \eqref{4106} and \eqref{4107}
				into \eqref{41022}, then we complete the proof of this lemma.
			\end{proof}
			
			Next, we establish the dissipation estimate for the quantity
			$\p_{23} \wu$ in $L^2$-norm.
			\begin{lemm}
				For any smooth solution $(u, b)$ of equation \eqref{eqr},
				it holds for $m \ge 5$
				\beq\label{4201}
				\begin{aligned}
					&-\frac{d}{dt}\i \p_3 \wb \cdot \p_{23} \wu dx
					+\i |\p_{23} \wu|^2 dx\\
					\lesssim
					&\sqrt{\mathcal{E}^m(t)}\widehat{\mathcal{D}}_{tan}^{3}(t)
					+\|(\p_{23}\wb, \p_{113} \wu, \Delta_h \p_3 \wb)\|_{L^2}^2
					+\ep \|(\p_{22}\wu, \p_{33}\wu)\|_{L^2} \|\p_{233}  \wb\|_{L^2}.
				\end{aligned}
				\deq
			\end{lemm}
			\begin{proof}
				Using the equations \eqref{eqwu} and \eqref{eqwb}, it is easy to check that
				\beq\label{4202}
				\begin{aligned}
					&-\frac{d}{dt}\i \p_3 \wb \cdot \p_{23} \wu dx
					+\i |\p_{23}  \wu|^2 dx\\
					=&-
					\i  (\p_1^2 \wu+\ep \p_2^2 \wu+\ep \p_3^2 \wu+\p_2 \wb)
					\cdot \p_{233}  \wb dx
					+\i  (\Delta_h \wb+\ep \p_3^2 \wb)\cdot \p_{233} \wu dx\\
					&+\i (u\cdot \nabla \wu )\cdot  \p_{233}  \wb dx
					+\i (-\wu\cdot \nabla u)\cdot  \p_{233}  \wb dx
					+\i (-b \cdot \nabla \wb)\cdot  \p_{233}  \wb dx\\
					&+\i (\wb\cdot \nabla b)\cdot  \p_{233}  \wb dx
					+\i (-u\cdot \nabla \wb)\cdot \p_{233}  \wu dx
					+\i (-\nabla(u\cdot \nabla)\times b)\cdot \p_{233}  \wu dx\\
					&+\i (b\cdot \nabla \wu) \cdot \p_{233}  \wu dx
					+\i (\nabla(b\cdot \nabla)\times u)\cdot \p_{233}  \wu dx
					:=\sum_{i=1}^{10} N_i.
				\end{aligned}
				\deq
				Integrating by part and using the H\"{o}lder inequality, we have
				\beqq
				N_1\le (\|\p_{113} \wu\|_{L^2}+\|\p_{23}\wb\|_{L^2})\|\p_{23}\wb\|_{L^2}
				+\ep \|(\p_{22}\wu, \p_{33}\wu)\|_{L^2} \|\p_{233}  \wb\|_{L^2},
				\deqq
				and
				\beqq
				N_2
				\le \|\Delta_h \p_3 \wb\|_{L^2}\|\p_{23}\wb\|_{L^2}
				+\ep \| \p_{33}\wu\|_{L^2} \|\p_{233}  \wb\|_{L^2}.
				\deqq
				Integrating by part and using the anisotropic type
				inequality \eqref{ie:Sobolev}, we have for $m \ge 5$
				\beqq
				\begin{aligned}
					N_3=
					&-\i (\p_3 u \cdot \nabla \wu+u_h \cdot \nabla_h \p_3 \wu
					+u_3\p_{33} \wu )\cdot  \p_{23}  \wb dx\\
					\lesssim
					&\|\p_3 u\|_{L^2}^{\frac14}\|\p_{13} u\|_{L^2}^{\frac14}
					\|\p_{33} u\|_{L^2}^{\frac14}\|\p_{133} u\|_{L^2}^{\frac14}
					\|\p_{23}  \wb\|_{L^2}
					\|\nabla \wu\|_{L^2}^{\frac12}\|\p_2\nabla \wu\|_{L^2}^{\frac12}\\
					&+\|u_h\|_{L^2}^{\frac12}\|\p_3 u_h\|_{L^2}^{\frac12}
					\|\nabla_h \p_3 \wu \|_{L^2}^{\frac12}\|\p_1 \nabla_h \p_3 \wu \|_{L^2}^{\frac12}
					\|\p_{23}  \wb\|_{L^2}^{\frac12}\|\p_{223}  \wb\|_{L^2}^{\frac12}\\
					&+\|u_3\|_{L^2}^{\frac14}\|\p_2 u_3\|_{L^2}^{\frac14}
					\|\p_3 u_3\|_{L^2}^{\frac14}\|\p_{23} u_3\|_{L^2}^{\frac14}
					\|\p_{23}  \wb\|_{L^2}\|Z_3 \p_{3} \wu\|_{L^2}^{\frac12}
					(\|\p_{13} \wu\|_{L^2}^{\frac14}\|Z_3^2 \p_{13} \wu\|_{L^2}^{\frac14} + \|\p_{13} \wu\|_{L^2}^{\frac12})\\
					&+\|\p_3 u_3\|_{L^2}^{\frac14}\|\p_{23} u_3\|_{L^2}^{\frac14}
					\|\p_{33} u_3\|_{L^2}^{\frac14}\|\p_{233} u_3\|_{L^2}^{\frac14}
					\|\p_{23}  \wb\|_{L^2}\|Z_3 \p_{3} \wu\|_{L^2}^{\frac12}
					\|Z_3 \p_{13} \wu\|_{L^2}^{\frac12}\\
					\lesssim
					&\sqrt{\mathcal{E}^m(t)}\widehat{\mathcal{D}}_{tan}^{3}(t).
				\end{aligned}
				\deqq
				Similarly, it is easy to check that
				\beqq
				N_4, N_5, N_6, N_7, N_9
				\lesssim  \sqrt{\mathcal{E}^m(t)}\widehat{\mathcal{D}}_{tan}^{3}(t).
				\deqq
				Integrating by part and using the anisotropic type
				inequality \eqref{ie:Sobolev}, we have
				\beqq
				\begin{aligned}
					N_8
					=
					&\i (\nabla \p_3 u_i  \times \p_i b
					+\nabla u_i  \times \p_i \p_3 b)\cdot \p_{23}  \wu dx\\
					\lesssim
					&\|\nabla_h \p_3 u\|_{L^2}^{\frac12}\|\nabla_h \p_{33} u\|_{L^2}^{\frac12}
					\|\nabla b\|_{L^2}^{\frac12}\|\p_2 \nabla b\|_{L^2}^{\frac12}
					\|\p_{23}  \wu \|_{L^2}^{\frac12}\|\p_{123}  \wu \|_{L^2}^{\frac12}\\
					&+\|\nabla \p_3 u\|_{L^2}^{\frac12}\|\nabla \p_{23} u\|_{L^2}^{\frac12}
					\|\nabla_h b\|_{L^2}^{\frac12}\|\nabla_h \p_3 b\|_{L^2}^{\frac12}
					\|\p_{23}  \wu \|_{L^2}^{\frac12}\|\p_{123}  \wu \|_{L^2}^{\frac12}\\
					&+\|\nabla_h u\|_{L^2}^{\frac12}\|\nabla_h \p_3 u\|_{L^2}^{\frac12}
					\|\nabla \p_3 b\|_{L^2}^{\frac12}\|\nabla \p_{23 } b\|_{L^2}^{\frac12}
					\|\p_{23}  \wu \|_{L^2}^{\frac12}\|\p_{123}  \wu \|_{L^2}^{\frac12}\\
					&+\|\nabla u\|_{L^2}^{\frac12}\|\nabla \p_2 u\|_{L^2}^{\frac12}
					\|\nabla_h \p_3 b\|_{L^2}^{\frac12}\|\nabla_h \p_{33} b\|_{L^2}^{\frac12}
					\|\p_{23}  \wu \|_{L^2}^{\frac12}\|\p_{123}  \wu \|_{L^2}^{\frac12}\\
					\lesssim
					&\sqrt{\mathcal{E}^m(t)}\widehat{\mathcal{D}}_{tan}^{3}(t).
				\end{aligned}
				\deqq
				Similarly, it is easy to check that
				\beqq
				N_{10}
				\lesssim
				\sqrt{\mathcal{E}^m(t)}\widehat{\mathcal{D}}_{tan}^{3}(t).
				\deqq
				Thus, substituting the estimates for $N_{1}$
				through $N_{10}$ into \eqref{4202}, we complete the proof of this lemma.
			\end{proof}
			
			Finally, we establish the estimate for the second order
			normal derivative of $(u, b)$ in negative Sobolev space.

			\begin{lemm}
				Under the assumption of \eqref{decay-assumption-new}, the smooth solution $(u, b)$ of equation \eqref{eqr} has for all $m \ge 5$
				\beq\label{4301}
				\begin{aligned}
					&\frac{1}{2}\|(\Lambda_h^{-s}\p_3 \wu,\Lambda_h^{-s} \p_3\wb)(t)\|_{L^2}^2
					+\int_0^t \|(\p_1 \Lambda_h^{-s} \p_3 \wu, \nabla_h \Lambda_h^{-s} \p_3 \wb)\|_{L^2}^2 d\tau\\
					&+\ep \int_0^t  \|(\p_2 \Lambda_h^{-s} \p_3 \wu, \p_3 \Lambda_h^{-s} \p_3\wu,
					\p_3 \Lambda_h^{-s} \p_3 \wb)\|_{L^2}^2d\tau\\
					\lesssim
					&\frac{1}{2}\|(\Lambda_h^{-s}\p_3 \wu,\Lambda_h^{-s} \p_3\wb)(0)\|_{L^2}^2
					+\frac12\mathcal{E}^m(u^\ep, b^\ep)(0)^2+\delta^2.
				\end{aligned}
				\deq
			\end{lemm}
			\begin{proof}
				Using the equations \eqref{eqwu} and \eqref{eqwb}, we have
				\beq\label{4302}
				\begin{aligned}
					&\frac{d}{dt}\frac{1}{2}\i(|\Lambda_h^{-s}\p_3 \wu|^2+|\Lambda_h^{-s} \p_3 \wb|^2)dx
					+\ep \i (|\p_2 \Lambda_h^{-s} \p_3 \wu|^2+|\p_3 \Lambda_h^{-s} \p_3\wu|^2) dx\\
					&+\i |\p_1 \Lambda_h^{-s} \p_3\wu|^2 dx
					+\i |\nabla_h \Lambda_h^{-s} \p_3\wb|^2 dx+\ep \i |\p_3 \Lambda_h^{-s} \p_3\wb|^2 dx\\
					=&\i \Lambda_h^{-s}\p_3(-u\cdot \nabla \wu) \cdot \Lambda_h^{-s} \p_3\wu \ dx
					+\i \Lambda_h^{-s}\p_3(\wu\cdot \nabla u
					+b \cdot \nabla \wb-\wb\cdot \nabla b) \cdot \Lambda_h^{-s} \p_3\wu \ dx\\
					&
					+\i \Lambda_h^{-s}\p_3(-u\cdot \nabla \wb-\nabla(u\cdot \nabla)\times b
					+b\cdot \nabla \wu+\nabla(b\cdot \nabla)\times u) \cdot \Lambda_h^{-s} \p_3\wb \ dx\\
					:=&M_1+M_2+M_3.
				\end{aligned}
				\deq
				It is easy to check that
				\beqq
				\begin{aligned}
					M_1
					=
					&\i \Lambda_h^{-s}(-\p_3 u_h \cdot \nabla_h \wu) \cdot \Lambda_h^{-s} \p_3\wu \ dx
					+\i \Lambda_h^{-s}(-\p_3 u_3 \p_3 \wu) \cdot \Lambda_h^{-s} \p_3\wu \ dx\\
					&+\i \Lambda_h^{-s}(-u_h \cdot \nabla_h  \p_3 \wu) \cdot \Lambda_h^{-s} \p_3\wu \ dx
					+\i \Lambda_h^{-s}(-u_3 \p_{33} \wu) \cdot \Lambda_h^{-s} \p_3\wu \ dx\\
					:=
					&M_{11}+M_{12}+M_{13}+M_{14}.
				\end{aligned}
				\deqq
				The H\"{o}lder's inequality and  inequality \eqref{a16} in Lemma \ref{H-L} yield directly
				\beqq
				\begin{aligned}
					M_{11}
					\le &(\|\p_3 u_h\|_{L^2}\|\p_{23} u_h\|_{L^2}
					+\|\p_{13} u_h\|_{L^2}\|\p_{123} u_h\|_{L^2})^{\frac{1-s}{2}}
					\|\p_3 u_h\|_{L^2}^{\frac{2s-1}{2}}
					\|\p_{33}u_h\|_{L^2}^{\frac12}
					\|\nabla_h \wu\|_{L^2}\|\Lambda_h^{-s} \p_3\wu\|_{L^2}\\
					\lesssim &\sqrt{\widehat{\mathcal{E}}_{tan}^{3}(t)}
					\sqrt{\widehat{\mathcal{D}}_{tan}^{3}(t)}
					\|\Lambda_h^{-s} \p_3\wu\|_{L^2},\\
					M_{12}
					\le &(\|\p_3 u_3\|_{L^2}\|\p_{23} u_3\|_{L^2}
					+\|\p_{13} u_3\|_{L^2}\|\p_{123} u_3\|_{L^2})^{\frac{1-s}{2}}
					\|\p_3 u_3\|_{L^2}^{\frac{2s-1}{2}}
					\|\p_{33} u_3\|_{L^2}^{\frac12}
					\|\p_3 \wu\|_{L^2}\|\Lambda_h^{-s} \p_3\wu\|_{L^2}\\
					\lesssim &\sqrt{\widehat{\mathcal{E}}_{tan}^{3}(t)}
					\sqrt{\widehat{\mathcal{D}}_{tan}^{3}(t)}
					\|\Lambda_h^{-s} \p_3\wu\|_{L^2},\\
					M_{13}
					\le &(\|u_h\|_{L^2}\|\p_{2} u_h\|_{L^2}
					+\|\p_{1} u_h\|_{L^2}\|\p_{12} u_h\|_{L^2})^{\frac{1-s}{2}}
					\|u_h\|_{L^2}^{\frac{2s-1}{2}}
					\|\p_{3}u_h\|_{L^2}^{\frac12}
					\|\nabla_h \p_3 \wu\|_{L^2}\|\Lambda_h^{-s} \p_3\wu\|_{L^2}\\
					\lesssim &\sqrt{\widehat{\mathcal{E}}_{tan}^{3}(t)}
					\sqrt{\widehat{\mathcal{D}}_{tan}^{3}(t)}
					\|\Lambda_h^{-s} \p_3\wu\|_{L^2}.
				\end{aligned}
				\deqq
				Similar to the estimate of term $K_1$ in \eqref{termk1}, it is easy to check that
				\beqq
				\begin{aligned}
					M_{14}
					\le
					&  (\|(u_3, \p_3 u_3)\|_{L^2}\|\p_2 (u_3, \p_3 u_3)\|_{L^2}
					+\|\p_1 (u_3, \p_3 u_3)\|_{L^2}\|\p_{12} (u_3, \p_3 u_3)\|_{L^2})^{\frac{1-s}{2}}
					\|(u_3, \p_3 u_3)\|_{L^2}^{\frac{2s-1}{2}}\\
					&\times \|\p_3 (u_3, \p_3 u_3)\|_{L^2}^{\frac12}
					\|(\p_3 \wu, Z_3^3 \p_3 \wu)\|_{L^2}^{\frac13}\|\p_3 \wu\|_{L^2}^{\frac23}
					\|\Lambda_h^{-s} \p_3\wu\|_{L^2}\\
					\lesssim
					&   {\widehat{\mathcal{E}}_{tan}^{3}(t)}^{\frac{7}{12}}
					{\widehat{\mathcal{D}}_{tan}^{3}(t)}^{\frac14}
					\|(\p_3 \wu, Z_3^3 \p_3 \wu)\|_{L^2}^{\frac13}
					\|\Lambda_h^{-s} \p_3\wu\|_{L^2}.
				\end{aligned}
				\deqq
				Thus, we can obtain the estimate
				\beqq
				\begin{aligned}
					M_1
					\lesssim
					\sqrt{\widehat{\mathcal{E}}_{tan}^{3}(t)}
					\sqrt{\widehat{\mathcal{D}}_{tan}^{3}(t)}
					\|\Lambda_h^{-s} \p_3\wu\|_{L^2}
					+{\widehat{\mathcal{E}}_{tan}^{3}(t)}^{\frac{7}{12}}
					{\widehat{\mathcal{D}}_{tan}^{3}(t)}^{\frac14}
					\|(\p_3 \wu, Z_3^3 \p_3 \wu)\|_{L^2}^{\frac13}
					\|\Lambda_h^{-s} \p_3\wu\|_{L^2}.
				\end{aligned}
				\deqq
				Similarly, it is easy to check that
				\beqq
				\begin{aligned}
					M_2
					\lesssim
					&\sqrt{\widehat{\mathcal{E}}_{tan}^{3}(t)}
					\sqrt{\widehat{\mathcal{D}}_{tan}^{3}(t)}
					\|\Lambda_h^{-s} \p_3\wu\|_{L^2}
					+{\widehat{\mathcal{E}}_{tan}^{3}(t)}^{\frac{7}{12}}
					{\widehat{\mathcal{D}}_{tan}^{3}(t)}^{\frac14}
					\|(\p_3 \wb, Z_3^3 \p_3 \wb)\|_{L^2}^{\frac13}
					\|\Lambda_h^{-s} \p_3\wu\|_{L^2},\\
					M_3
					\lesssim
					&\sqrt{\widehat{\mathcal{E}}_{tan}^{3}(t)}
					\sqrt{\widehat{\mathcal{D}}_{tan}^{3}(t)}
					\|\Lambda_h^{-s} \p_3\wb\|_{L^2}
					+{\widehat{\mathcal{E}}_{tan}^{3}(t)}^{\frac{7}{12}}
					{\widehat{\mathcal{D}}_{tan}^{3}(t)}^{\frac14}
					\|(\p_3 \wu, \p_3 \wb, Z_3^3 \p_3 \wu, Z_3^3 \p_3 \wb)\|_{L^2}^{\frac13}
					\|\Lambda_h^{-s} \p_3\wb\|_{L^2}.
				\end{aligned}
				\deqq
				Substituting the estimates of terms for $M_1$ through $M_3$ into \eqref{4302}, we have
				\beqq
				\begin{aligned}
					&\frac{1}{2}\|(\Lambda_h^{-s}\p_3 \wu,\Lambda_h^{-s} \p_3\wb)(t)\|_{L^2}^2
					+\int_0^t \|(\p_1 \Lambda_h^{-s} \p_3 \wu, \nabla_h \Lambda_h^{-s} \p_3 \wb)\|_{L^2}^2 d\tau\\
					&+\ep \int_0^t  \|(\p_2 \Lambda_h^{-s} \p_3 \wu, \p_3 \Lambda_h^{-s} \p_3\wu,
					\p_3 \Lambda_h^{-s} \p_3 \wb)\|_{L^2}^2d\tau\\
					\lesssim
					&\frac{1}{2}\|(\Lambda_h^{-s}\p_3 \wu,\Lambda_h^{-s} \p_3\wb)(0)\|_{L^2}^2
					+\underset{0\le \tau \le t}{\sup}\|(\Lambda_h^{-s} \p_3\omega^u,
					\Lambda_h^{-s} \p_3 \omega^b)(\tau)\|_{L^2}
					\int_0^t\sqrt{\widehat{\mathcal{E}}_{tan}^{3}(\tau)}
					\sqrt{\widehat{\mathcal{D}}_{tan}^{3}(\tau)}d\tau\\
					&+\underset{0\le \tau \le t}{\sup}\|(\Lambda_h^{-s} \p_3\omega^u,
					\Lambda_h^{-s} \p_3\omega^b)(\tau)\|_{L^2}
					\underset{0\le \tau \le t}{\sup}
					\|(\p_3 \wu, \p_3 \wb, Z_3^3 \p_3 \wu, Z_3^3 \p_3 \wb)(\tau)\|_{L^2}^{\frac13}
					\int_0^t  {\widehat{\mathcal{E}}_{tan}^{3}(\tau)}^{\frac{7}{12}}
					{\widehat{\mathcal{D}}_{tan}^{3}(\tau)}^{\frac14} d\tau.
				\end{aligned}
				\deqq
				Under the assumption \eqref{decay-assumption-new}, we have
				\beqq
				\begin{aligned}
					&\int_0^t\sqrt{\widehat{\mathcal{E}}_{tan}^{3}(\tau)}
					\sqrt{\widehat{\mathcal{D}}_{tan}^{3}(\tau)}d\tau\\
					\lesssim
					&\underset{0\le \tau \le t}{\sup}[{\widehat{\mathcal{E}}_{tan}^{3}(\tau)(1+\tau)^s}]^{\frac12}
					\left\{\int_0^t \widehat{\mathcal{D}}_{tan}^{3}(\tau) (1+\tau)^{\sigma}d\tau\right\}^{\frac12}
					\left\{\int_0^t (1+\tau)^{-(\sigma+s)}d\tau\right\}^{\frac12}
					\lesssim \delta,
				\end{aligned}
				\deqq
				and
				\beqq
				\begin{aligned}
					&\int_0^t  {\widehat{\mathcal{E}}_{tan}^{3}(\tau)}^{\frac{7}{12}}
					{\widehat{\mathcal{D}}_{tan}^{3}(\tau)}^{\frac14} d\tau\\
					\lesssim
					&\underset{0\le \tau \le t}{\sup}[{\widehat{\mathcal{E}}_{tan}^{3}(\tau)(1+\tau)^s}]
					^{\frac{7}{12}}
					\left\{\int_0^t \widehat{\mathcal{D}}_{tan}^{3}(\tau) (1+\tau)^{\sigma}d\tau\right\}^{\frac14}
					\left\{\int_0^t (1+\tau)^{-\frac{3\sigma+7s}{9}}d\tau\right\}^{\frac34}
					\lesssim \delta^{\frac56},
				\end{aligned}
				\deqq
				where we have used the condition $\frac{9}{10}<\sigma <s <1$.
				Then, using the H\"{o}lder inequality and estimate \eqref{uniform_estimate},
				we complete the proof of this lemma.
			\end{proof}
			
			Finally, let us establish the decay estimate for the quantity
			$\widehat{\mathcal{E}}^{3}_{tan}(t)$.
			\begin{lemm}
				Under the assumption \eqref{decay-assumption-new},
				the smooth solution $(u, b)$ of equation \eqref{eqr} has the estimate
				\beq\label{4401}
				\begin{aligned}
					(1+t)^{s}\widehat{\mathcal{E}}^{3}_{tan}(t)
					+\kappa \int_0^t (1+\tau)^{\sigma} \widehat{\mathcal{D}}_{tan}^{3}(\tau) d\tau
					+\int_0^t (1+\tau)^{\sigma} \widehat{\mathcal{D}}_{tan, \ep}^{3}(\tau)d\tau
					\le C \widehat{C}_0.
				\end{aligned}
				\deq
				where the constant $C$ independent of parameter $\ep$ and time $t$
				and $\widehat{C}_0$ is defined in \eqref{4103}.
			\end{lemm}
			\begin{proof}
				The combination of estimates \eqref{31010}, \eqref{4101} and \eqref{4201}
                yields directly
				\beq\label{4402}
				\begin{aligned}
					&\frac{d}{dt}\widehat{\mathcal{E}}^{3,q}_{tan}(t)
					+\kappa \widehat{\mathcal{D}}_{tan}^{3}(t)
					+\widehat{\mathcal{D}}_{tan, \ep}^{3}(t)
					\le 0,
				\end{aligned}
				\deq
				where the norms $\widehat{\mathcal{E}}^{3,q}_{tan}(t)$
				and $\widehat{\mathcal{D}}_{tan, \ep}^{3}(t)$ are defined as
				\beqq
				\begin{aligned}
					\widehat{\mathcal{E}}^{3,q}_{tan}(t)
					:=&\|(u, b)(t)\|_{H^{3}_{tan}}^2+\|(\wu, \wb)(t)\|_{H^{2}_{tan}}^2
					+\|(\p_3 \wu, \p_3 \wb)\|_{H^1_{tan}}^2
					+2\kappa \i \p_{23} \wb \cdot \p_3 \wu dx\\
					&+\sum_{0\le |\alpha_h| \le 2} 2\kappa  \i \p_2 Z^{\ah}b \cdot  Z^{\ah}u \ dx
					+\sum_{0\le |\alpha_h| \le 1} 2\kappa \i \p_2 Z^{\ah} \wb \cdot Z^{\ah} \wu dx,
				\end{aligned}
				\deqq
				and
				\beqq
				\begin{aligned}
					\widehat{\mathcal{D}}_{tan, \ep}^{3}(t)
					:=\ep \|(\p_2 u, \p_3 u, \p_3 b)(t)\|_{H^{3}_{tan}}^2
					+\ep\|(\p_2 \wu, \p_3 \wu, \p_3 \wb)\|_{H^{2}_{tan}}^2
					+\ep \|(\p_2 \p_3 \wu, \p_3^2 \wu, \p_3^2 \wb)\|_{H^1_{tan}}^2.
				\end{aligned}
				\deqq
				Here the $\kappa$ is a small positive constant.
				The combination of estimate \eqref{uniform_estimate} and \eqref{4301} yields directly
				\beqq
				\begin{aligned}
					\|\Lambda_h^{-s}(u, b, \wu, \wb, \p_3 \wu, \p_3 \wb)\|_{L^2}^2
					+\|\nabla_h(u, b)\|_{H^{2}_{tan}}^2
					+\|\nabla_h(\wu, \wb)\|_{H^{1}_{tan}}^2
					+\|\nabla_h(\p_3 \wu, \p_3 \wb)\|_{L^2}^2
					\le C \widehat{C}_0,
				\end{aligned}
				\deqq
				where the constant $\widehat{C}_0$ is defined by
				\beq\label{4103}
				\widehat{C}_0:=
				\|(\Lambda_h^{-s}\p_3 \wu,\Lambda_h^{-s} \p_3\wb)(0)\|_{L^2}^2
				+\mathcal{E}(u^\ep, b^\ep)(0)+\delta^2.
				\deq
				Similar to the estimate \eqref{31007}, it is easy to check that
				\beqq
				\begin{aligned}
					&\widehat{\mathcal{E}}^{3,q}_{tan}(t)
					\lesssim
					\widehat{\mathcal{E}}^{3}_{tan}(t)\\
					\lesssim
					&(\|\Lambda_h^{-s}(u, b, \wu, \wb, \p_3 \wu, \p_3 \wb)\|_{L^2}^2
					+\|\nabla_h(u, b)\|_{H^{2}_{tan}}^2
					+\|\nabla_h(\wu, \wb)\|_{H^{1}_{tan}}^2
					+\|\nabla_h(\p_3 \wu, \p_3 \wb)\|_{L^2}^2)^{\frac{1}{1+s}}\\
					& \times(\|\nabla_h(u,b)\|_{H^{2}_{tan}}^2
					+\|\nabla_h(\wu,\wb)\|_{H^{1}_{tan}}^2
					+\|\nabla_h(\p_3 \wu, \p_3 \wb)\|_{L^2}^2)^{\frac{s}{1+s}}\\
					\lesssim
					&\widehat{C}_0^{\frac{1}{1+s}} \widehat{\mathcal{D}}_{tan}^{3}(t)^{\frac{s}{1+s}},
				\end{aligned}
				\deqq
				which, together with \eqref{4402}, we can obtain
				\beqq
				\frac{d}{dt}\widehat{\mathcal{E}}^{3,q}_{tan}(t)
				+\kappa \widehat{C}_0^{-\frac{1}{s}}
				\widehat{\mathcal{E}}^{3,q}_{tan}(t)^{1+\frac{1}{s}}\le 0.
				\deqq
				Thus, we can obtain the decay estimate
				\beqq
				\widehat{\mathcal{E}}^{3}_{tan}(t)
				\lesssim
				\widehat{\mathcal{E}}^{3,q}_{tan}(t)
				\lesssim \widehat{C}_0(1+t)^{-s}.
				\deqq
				For $\frac{9}{10}<\sigma<s<1$, multiplying \eqref{4402} by $(1+t)^{\sigma}$, we have
				\beqq
				\begin{aligned}
					\frac{d}{dt}[(1+t)^{\sigma}\widehat{\mathcal{E}}^{3,q}_{tan}(t)]
					+\kappa (1+t)^{\sigma} \widehat{\mathcal{D}}_{tan}^{3}(t)
					+(1+t)^{\sigma} \widehat{\mathcal{D}}_{tan, \ep}^{3}(t)
					\le \sigma (1+t)^{\sigma-1}\widehat{\mathcal{E}}^{3,q}_{tan}(t).
				\end{aligned}
				\deqq
				Integrating over $[0, t]$, then we have
				\beqq
				\begin{aligned}
					&(1+t)^{\sigma}\widehat{\mathcal{E}}^{3,q}_{tan}(t)
					+\kappa \int_0^t (1+\tau)^{\sigma} \widehat{\mathcal{D}}_{tan}^{3}(\tau) d\tau
					+\int_0^t (1+\tau)^{\sigma} \widehat{\mathcal{D}}_{tan, \ep}^{3}(\tau)d\tau\\
					\le
					&\widehat{\mathcal{E}}^{3,q}_{tan}(0)
					+\sigma \underset{0\le \tau \le t}{\sup}
					[(1+\tau)^{s}\widehat{\mathcal{E}}^{3,q}_{tan}(\tau)]
					\int_0^t (1+\tau)^{\sigma-1-s}d\tau\\
					\le
					&C\widehat{C}_0.
				\end{aligned}
				\deqq
				Therefore, we complete the proof of this lemma.
			\end{proof}

			\begin{proof}[\textbf{Proof of closed energy estimate.}]
				Let us choose
				\beqq
				\delta:=4C(\|(\Lambda_h^{-s}\p_3 \wu,\Lambda_h^{-s} \p_3\wb)(0)\|_{L^2}^2
				+\mathcal{E}(u^\ep, b^\ep)(0))
				\le \min\left\{1, \frac{1}{4C}\right\},
				\deqq
				then the estimate \eqref{4401} implies directly
				\beq\label{decay-assumption-new-close}
				\begin{aligned}
					&(1+t)^{s}\widehat{\mathcal{E}}^{3}_{tan}(t)
					+\kappa \int_0^t (1+\tau)^{\sigma} \widehat{\mathcal{D}}_{tan}^{3}(\tau) d\tau
					+\int_0^t (1+\tau)^{\sigma} \widehat{\mathcal{D}}_{tan, \ep}^{3}(\tau)d\tau\\
					\le
					&C (\|(\Lambda_h^{-s}\p_3 \wu,\Lambda_h^{-s} \p_3\wb)(0)\|_{L^2}^2+\mathcal{E}(u^\ep, b^\ep)(0))
					+ C\delta^2\\
					\le
					&\frac{\delta}{4}+\frac{\delta}{4}=\frac{\delta}{2}.
				\end{aligned}
				\deq
				The combination of estimates \eqref{decay-assumption-new} and
				\eqref{decay-assumption-new-close} help us to establish the closed estimate.
				Thus, we establish the time decay rate estimate \eqref{decay-assumption-new-close}
				for the system \eqref{eqr}.
				Since the estimate \eqref{decay-assumption-new-close} is independent of $\ep$,
			one can apply this method to the system \eqref{eqr0} to establish the decay estimate
				\beq\label{decay-limit}
				\begin{aligned}
					&(1+t)^{s}(\|(u^0, b^0)(t)\|_{H^2}^2+\|\p_{33}\nabla_h(u^0, b^0)(t)\|_{L^2}^2)\\
					&+\int_0^t (1+\tau)^{\sigma}(\|(\nabla_h u^0, \nabla_h b^0)(\tau)\|_{H^2}^2
					+\|(\p_{133}\nabla_h u^0, \p_{33} \nabla_{h}^2 b^0)(\tau)\|_{L^2}^2)d\tau
					\le C,
				\end{aligned}
				\deq
				for all $\frac{9}{10}<\sigma<s<1$ and $C$ is a positive constant independent of time.
			\end{proof}

			\subsection{Convergence rate of solution}
			In this subsection, we will establish the convergence rate
			of the solutions between  equation \eqref{eqr} and \eqref{eqr0}.
			Let us define $(\bu, \bb):=(u^\ep-u^0, b^\ep-b^0)$
			and $(\bwu, \bwb):=(\nabla \times \bu, \nabla \times \bb)$
			and apply the equations \eqref{eqr} and \eqref{eqr0},
			then we can obtain
			\begin{equation}\label{401}
				\left\{
				\begin{array}{*{4}{ll}}
					\p_t \bu-\p_{11} \bu+\nabla (p^\ep-p^0)-\p_2 \bb
					=f,\\
					\p_t \bb-\Delta_h \bb-\p_2 \bu
					=g,\\
					\nabla \cdot \bu= 0,\quad  \nabla \cdot \bb= 0,
				\end{array}
				\right.
			\end{equation}
			and
			\begin{equation}\label{402}
				\left\{
				\begin{array}{*{4}{ll}}
					\p_t \bwu-\p_{11} \bwu-\p_2 \bwb
					=\nabla \times f,\\
					\p_t \bwb-\Delta_h \bwb-\p_2 \bwu
					=\nabla \times g,
				\end{array}
				\right.
			\end{equation}
			where the source terms $f$ and $g$ are defined by
			\beqq
			\begin{aligned}
				f:=&-u^\ep \cdot \nabla \bu-\bu \cdot \nabla u^0+b^\ep\cdot \nabla \bb+\bb\cdot \nabla b^0
				+\ep \p_{22}u^\ep+\ep \p_{33} u^\ep,\\
				g:=&-u^\ep \cdot \nabla \bb-\bu \cdot \nabla b^0+b^\ep \cdot \nabla \bu+\bb\cdot \nabla u^0
				+\ep \p_{33}b^\ep.
			\end{aligned}
			\deqq
			
			Let us define the notations
			\beqq
			\begin{aligned}
				\mathcal{\overline{E}}(t)
				&:=\|(\bu, \bb)(t)\|_{H^1_{tan}}^2+\|(\bwu, \bwb)(t)\|_{L^2}^2,\\
				\mathcal{\overline{D}}(t)
				&:=\|(\p_1 \bu, \nabla_h \bb)(t)\|_{H^1_{tan}}^2
				+\|(\p_2 \bu, \p_1 \bwu, \nabla_h \bwb)(t)\|_{L^2}^2,
			\end{aligned}
			\deqq
			and
			\beqq
			\mathcal{{B}}(t)
			:=\|(u^0, b^0, u^\ep, b^\ep)(t)\|_{H^2}, \quad
			\mathcal{{B}}_h(t)
			:=\|\nabla_h (u^0, b^0, u^\ep, b^\ep)(t)\|_{H^2}.
			\deqq
			First of all, let us establish the estimate for $(\bu, \bb)$ in $L^2$ norm.
			\begin{lemm}\label{lemma32}
				Under the conditions of Theorem \ref{main_result_two},
				then it holds
				\beq\label{4501}
				\begin{aligned}
					\frac{d}{dt}\|(\bu, \bb)\|_{L^2}^2 +\|(\p_1 \bu, \nabla_h \bb)\|_{L^2}^2
					\lesssim
					\nu \mathcal{\overline{D}}(t)
					+ \mathcal{{B}}(t)\mathcal{{B}}_h(t)
					\mathcal{\overline{E}}(t)
					+\ep \|(\p_{22}u^\ep, \p_{33} u^\ep, \p_{33}b^\ep)\|_{H^1}
					\mathcal{\overline{E}}(t)^{\frac12},
				\end{aligned}
				\deq
				where $\nu$ is a small positive constant.
			\end{lemm}
			\begin{proof}
				The equation \eqref{401} yields directly
				\beq\label{4502}
				\begin{aligned}
					&\frac{d}{dt}\frac{1}{2}\i (|\bu|^2+|\bb|^2)dx
					+\i (|\p_1 \bu|^2+|\nabla_h \bb|^2)dx\\
					=&\i (-\bu \cdot \nabla u^0+\bb \cdot \nabla b^0
					+\ep \p_{22}u^\ep+\ep \p_{33} u^\ep)\cdot \bu \ dx
					+\i (-\bu \cdot \nabla b^0+ \bb\cdot \nabla u^0+\ep \p_{33}b^\ep)\cdot \bb \ dx,
				\end{aligned}
				\deq
				where we have used the basic fact
				\beqq
				\i (u^\ep \cdot \nabla) \bu \cdot \bu \ dx
				=\i (u^\ep \cdot \nabla) \bb \cdot \bb \  dx
				=\i (b^\ep \cdot \nabla) \bb \cdot \bu \  dx
				+\i (b^\ep \cdot \nabla) \bu \cdot \bb \  dx
				=0.
				\deqq
				Using the anisotropic type inequality \eqref{ie:Sobolev}
				and divergence-free condition, then we have
				\beq\label{4503}
				\begin{aligned}
					&-\i (\bu \cdot \nabla) u^0\cdot \bu \ dx
					=-\i (\bu_h \cdot \nabla_h) u^0 \cdot \bu \ dx
					-\i \bu_3 \p_3  u^0 \cdot \bu \ dx \\
					\lesssim
					&\|\bu_h\|_{L^2}^{\frac12}\|\p_2 \bu_h\|_{L^2}^{\frac12}
					\|\nabla_h u^0\|_{L^2}^{\frac12}\|\p_3 \nabla_h u^0\|_{L^2}^{\frac12}
					\|\bu\|_{L^2}^{\frac12}\|\p_1 \bu\|_{L^2}^{\frac12}\\
					&+\|\bu_3\|_{L^2}^{\frac12}\|\p_3 \bu_3\|_{L^2}^{\frac12}
					\|\p_3  u^0\|_{L^2}^{\frac12}\|\p_2 \p_3  u^0\|_{L^2}^{\frac12}
					\|\bu\|_{L^2}^{\frac12}\|\p_1 \bu\|_{L^2}^{\frac12}\\
					\lesssim
					&
					\nu\|(\p_1 \bu,\p_2 \bu)\|_{L^2}^2
					+\|\bu \|_{L^2}^2\|\nabla u^0\|_{L^2}\|\nabla_h \p_3 u^0\|_{L^2},
				\end{aligned}
				\deq
				and
				\beq\label{4504}
				\begin{aligned}
					&\i (\bb \cdot \nabla) b^0 \cdot \bu \ dx
					=\i (\bb_h \cdot \nabla_h) b^0 \cdot \bu \ dx
					+\i \bb_3 \p_3 b^0 \cdot \bu \ dx\\
					\lesssim
					&\|\bb_h\|_{L^2}^{\frac12}\|\p_2 \bb_h\|_{L^2}^{\frac12}
					\|\nabla_h b^0\|_{L^2}^{\frac12}\|\p_3 \nabla_h b^0\|_{L^2}^{\frac12}
					\|\bu\|_{L^2}^{\frac12}\|\p_1 \bu\|_{L^2}^{\frac12}\\
					&+\|\bb_3\|_{L^2}^{\frac12}\|\p_3 \bb_3\|_{L^2}^{\frac12}
					\|\p_3 b^0\|_{L^2}^{\frac12}\|\p_{23} b^0\|_{L^2}^{\frac12}
					\|\bu\|_{L^2}^{\frac12}\|\p_1 \bu \|_{L^2}^{\frac12}\\
					\lesssim
					&\nu \|(\p_1 \bu, \nabla_h \bb)\|_{L^2}^2
					+\|(\bu, \bb)\|_{L^2}^2\|\nabla b^0\|_{L^2}
					\|\nabla_h \p_3 b^0\|_{L^2}.
				\end{aligned}
				\deq\label{4505}
				Similarly, we have
				\beq \label{45051}
				\begin{aligned}
					&
					-\i (\bu \cdot \nabla) b^0\cdot \bb \ dx
					\lesssim \nu \|(\p_1 \bu, \nabla_h \bb)\|_{L^2}^2
					+\|(\bu, \bb)\|_{L^2}^2\|\nabla b^0\|_{L^2}
					\|\nabla_h \p_3 b^0\|_{L^2},\\
					&\i (\bb\cdot \nabla) u^0\cdot \bb \ dx
					\lesssim
					\nu \|\nabla_h \bb \|_{L^2}^2
					+\|\bb \|_{L^2}^2\|\nabla u^0\|_{L^2}\|\nabla_h \p_3 u^0\|_{L^2}.
				\end{aligned}
				\deq
				The H\"{o}lder inequality gives directly
				\beq\label{4506}
				\ep \i ( \p_{22}u^\ep+\p_{33} u^\ep)\cdot \bu \ dx
				+\ep\i \p_{33}b^\ep \cdot \bb \ dx
				\le \ep \|(\p_{22}u^\ep, \p_{33} u^\ep, \p_{33}b^\ep)\|_{L^2}\|(\bu, \bb)\|_{L^2}.
				\deq
				Substituting estimates \eqref{4503}-\eqref{4506} into \eqref{4502}, then we complete the proof of this lemma.
			\end{proof}
			
			Next, we will establish the dissipation estimate for the quantity $\p_2 \bu$ in $L^2-$norm.
			\begin{lemm}\label{lemma32}
				Under the conditions of Theorem \ref{main_result_two},
				then it holds
				\beq\label{4601}
				\begin{aligned}
					&\frac{d}{dt}\i  \p_2  \bb \cdot \bu \ dx+\i |\p_2 \bu|^2 dx\\
					\lesssim
					&\|(\p_1 \bu, \nabla_h \bb)\|_{H^1}^2
					+ \mathcal{{B}}(t)\mathcal{{B}}_h(t)
					\mathcal{\overline{E}}(t)
					+\ep \|(\p_{22}u^\ep, \p_{33} u^\ep, \p_{33}b^\ep)\|_{H^1}
					\mathcal{\overline{E}}(t)^{\frac12}.
				\end{aligned}
				\deq
			\end{lemm}
			\begin{proof}
				The equation $\eqref{401}_2$ yields directly
				\beqq
				-\i \p_t \bb \cdot \p_2 \bu \ dx+\i |\p_2 \bu|^2 dx
				=-\i \Delta_h \bb \cdot \p_2 \bu \ dx-\i g \cdot \p_2 \bu \ dx.
				\deqq
				The equation $\eqref{401}_1$ yields directly
				\beqq
				\begin{aligned}
					-\i \p_t \bb \cdot \p_2 \bu \ dx
					&=\frac{d}{dt}\i  \p_2  \bb \cdot \bu \ dx
					-\i  \p_2  \bb \cdot \p_t \bu \ dx\\
					&=\frac{d}{dt}\i  \p_2  \bb \cdot \bu \ dx
					-\i  \p_2  \bb \cdot (f+\p_1^2 \bu+\p_2 \bb) \ dx.
				\end{aligned}
				\deqq
				Thus, we can obtain the equation
				\beq\label{4602}
				\begin{aligned}
					&\frac{d}{dt}\i  \p_2  \bb \cdot \bu \ dx+\i |\p_2 \bu|^2 dx\\
					=&-\i \Delta_h \bb \cdot \p_2 \bu \ dx
					+\i  |\p_2  \bb|^2 \ dx
					+\i  \p_2  \bb \cdot \p_1^2 \bu \ dx\\
					&+\i  \p_2  \bb \cdot (-u^\ep \cdot \nabla \bu-\bu \cdot \nabla u^0
					+b^\ep\cdot \nabla \bb+\bb\cdot \nabla b^0
					+\ep \p_{22}u^\ep+\ep \p_{33} u^\ep) \ dx\\
					&-\i \p_2 \bu \cdot (-u^\ep \cdot \nabla \bb-\bu \cdot \nabla b^0
					+b^\ep \cdot \nabla \bu+\bb\cdot \nabla u^0+\ep \p_{33}b^\ep) \ dx
				\end{aligned}
				\deq
				Using the anisotropic type inequality \eqref{ie:Sobolev}, we have
				\beq\label{4603}
				\begin{aligned}
					\i  \p_2  \bb \cdot (u^\ep \cdot \nabla \bu) dx
					\lesssim
					&\|\p_2  \bb\|_{L^2}^{\frac12}\|\p_{23}  \bb\|_{L^2}^{\frac12}
					\|u^\ep\|_{L^2}^{\frac12}\|\p_2 u^\ep\|_{L^2}^{\frac12}
					\|\nabla \bu\|_{L^2}^{\frac12}\|\p_1 \nabla \bu\|_{L^2}^{\frac12}\\
					\lesssim
					&\frac14\|(\p_1 \nabla \bu, \p_{23}  \bb)\|_{L^2}^2
					+\|(\p_2  \bb, \nabla \bu)\|_{L^2}
					\|u^\ep\|_{L^2}\|\p_2 u^\ep\|_{L^2}.
				\end{aligned}
				\deq
				Similarly, it is easy to check that
				\beq\label{4604}
				\begin{aligned}
					&\i  \p_2  \bb \cdot (-\bu \cdot \nabla u^0
					+b^\ep\cdot \nabla \bb+\bb\cdot \nabla b^0) \ dx\\
					\lesssim
					&\frac14 \|(\p_1 \bu, \nabla_h \bb, \nabla_h \nabla \bb)\|_{L^2}^2
					+\|(\bu, \bb, \nabla \bb)\|_{L^2}^2
					\|(\nabla u^0, \nabla b^0, b^\ep)\|_{L^2}
					\|\nabla_h (\nabla u^0, \nabla b^0,  b^\ep)\|_{L^2},
				\end{aligned}
				\deq
				and
				\beq\label{4605}
				\begin{aligned}
					&\i \p_2 \bu \cdot (-u^\ep \cdot \nabla \bb-\bu \cdot \nabla b^0
					+b^\ep \cdot \nabla \bu+\bb\cdot \nabla u^0) \ dx\\
					\lesssim
				&\frac14(\|\p_2 \bu\|_{L^2}^2+\|(\p_1 \nabla \bu, \nabla_h \p_3 \bb)\|_{L^2}^2)
					+\|(\bu, \bb)\|_{H^1}^2\|(u^\ep, b^\ep, u^0, b^0)\|_{H^1}
					\|\nabla_h(u^\ep, b^\ep, u^0, b^0)\|_{H^1}.
				\end{aligned}
				\deq
				Substituting the estimates \eqref{4603}, \eqref{4604} and \eqref{4605}
				into \eqref{4602}, we complete the proof of this lemma.
			\end{proof}

			Next, we will establish the tangential estimate for the quantity $(\bu, \bb)$ in $L^2-$norm.
			\begin{lemm}\label{lemma32}
				Under the conditions of Theorem \ref{main_result_two},
				then it holds
				\beq\label{4701}
				\begin{aligned}
					\frac{d}{dt}\|\p_i(\bu, \bb)\|_{L^2}^2
					+\|\p_i(\p_1 \bu, \nabla_h \bb)\|_{L^2}^2
					\lesssim
					\nu \mathcal{\overline{D}}(t)
					+ \mathcal{{B}}(t)\mathcal{{B}}_h(t)
					\mathcal{\overline{E}}(t)
					+\ep \|(\p_{22}u^\ep, \p_{33} u^\ep, \p_{33}b^\ep)\|_{H^1}
					\mathcal{\overline{E}}(t)^{\frac12},
				\end{aligned}
				\deq
				where  $i=1,2$, and $\nu$ is a small positive constant.
			\end{lemm}
			\begin{proof}
				For $i=1,2$, the equation \eqref{401} yields directly
				\beq\label{4702}
				\begin{aligned}
					&\frac{d}{dt}\frac{1}{2}\i (|\p_i \bu|^2+|\p_i \bb|^2)dx
					+\i (|\p_1 \p_i \bu|^2+|\nabla_h \p_i \bb|^2)dx\\
					=&\i \p_i (-u^\ep \cdot \nabla \bu-\bu \cdot \nabla u^0+b^\ep\cdot \nabla \bb+\bb\cdot \nabla b^0
					+\ep \p_{22}u^\ep+\ep \p_{33} u^\ep)\cdot \p_i \bu \ dx\\
					&+\i \p_i (-u^\ep \cdot \nabla \bb-\bu \cdot \nabla b^0+b^\ep \cdot \nabla \bu
					+\bb\cdot \nabla u^0+\ep \p_{33}b^\ep)\cdot \p_i  \bb \ dx.
				\end{aligned}
				\deq
				Integrating by part and using the anisotropic type inequality \eqref{ie:Sobolev}, we have
				\beq\label{4703}
				\begin{aligned}
					-\i \p_i (u^\ep \cdot \nabla \bu)\cdot \p_i \bu \ dx
					=&-\i (\p_i u^\ep \cdot \nabla) \bu \cdot \p_i \bu \ dx
					+\frac{1}{2}\i |\p_i \bu|^2 {\rm div}u^\ep \ dx\\
					\lesssim
					&\|\p_i u^\ep\|_{L^2}^{\frac14}\|\p_3 \p_i u^\ep\|_{L^2}^{\frac14}
					\|\p_2 \p_i u^\ep\|_{L^2}^{\frac14}\|\p_{23} \p_i u^\ep\|_{L^2}^{\frac14}
					\|\nabla \bu\|_{L^2}
					\|\p_i \bu\|_{L^2}^{\frac12}\|\p_1 \p_i \bu\|_{L^2}^{\frac12}\\
					\lesssim
					&\nu(\|\p_1 \p_i \bu\|_{L^2}^2+\|\p_i \bu\|_{L^2}^2)
					+\|\bu\|_{H^1}^2 \|\p_i u^\ep\|_{H^1}\|\p_2 \p_i u^\ep\|_{H^1}.
				\end{aligned}
				\deq
				Similarly, it is easy to check that
				\beq\label{4704}
				\begin{aligned}
					&\i \p_i (-\bu \cdot \nabla u^0+\bb\cdot \nabla b^0)\cdot \p_i \bu \ dx\\
					\lesssim
					&\nu(\|\p_1 \p_i \bu\|_{L^2}^2+\|\nabla_h \bu\|_{L^2}^2+\|\p_i \p_3 \bb\|_{L^2}^2)
					+\|(\bu, \bb)\|_{H^1}^2\|(u^0, b^0)\|_{H^2}\|\nabla_h(u^0, b^0)\|_{H^2},
				\end{aligned}
				\deq
				and
				\beq\label{4705}
				\begin{aligned}
					&\i \p_i(-u^\ep \cdot \nabla \bb-\bu \cdot \nabla b^0+\bb\cdot \nabla u^0)\cdot \p_i  \bb \ dx\\
					\lesssim
					&\nu(\|\p_{ii} \bb\|_{L^2}^2+\|\p_2 \nabla \bb\|_{L^2}^2)
					+\|(\bu, \bb)\|_{H^1}^2
					(\|\nabla (u^0, b^0)\|_{H^1_{tan}} \|\nabla_h \nabla(u^0, b^0)\|_{H^1_{tan}}
					+\|u^\ep\|_{H^1}\|\p_1 u^\ep\|_{H^1}).
				\end{aligned}
				\deq
				Integrating by part and using the condition ${\nabla \cdot} b^\ep=0$, we have
				\beq\label{4706}
				\begin{aligned}
					&\i \p_i (b^\ep\cdot \nabla \bb)\cdot \p_i \bu \ dx
					+\i \p_i (b^\ep \cdot \nabla \bu)\cdot \p_i  \bb \ dx\\
					=
					&\i (\p_i b^\ep\cdot \nabla \bb)\cdot \p_i \bu \ dx
					+\i (\p_i b^\ep \cdot \nabla \bu)\cdot \p_i  \bb \ dx\\
					&+\i (b^\ep\cdot \nabla \p_i \bb)\cdot \p_i \bu \ dx
					+\i (b^\ep \cdot \nabla  \p_i \bu)\cdot \p_i  \bb \ dx\\
					\lesssim
					&\nu\|(\p_1 \nabla \bu, \p_2 \nabla \bb)\|_{L^2}^2
					+\|(\bu, \bb)\|_{H^1}^2\|\nabla b^\ep\|_{H^1}\|\nabla_h \nabla b^\ep\|_{H^1}.
				\end{aligned}
				\deq
				Finally, we apply the H\"{o}lder inequality to obtain
				\beq\label{4707}
				\begin{aligned}
					&\ep  \i \p_i (\p_{22}u^\ep+\p_{33} u^\ep)\cdot \p_i \bu \ dx
					+\ep  \i  \p_{33} \p_i b^\ep \cdot \p_i  \bb \ dx\\
					\lesssim
					&\ep (\|(\p_i \p_{22}u^\ep, \p_i \p_{33} u^\ep)\|_{L^2}\|\p_i \bu \|_{L^2}
					+\|\p_i \p_{33} b^\ep\|_{L^2}\|\p_i  \bb\|_{L^2}).
				\end{aligned}
				\deq
				Substituting the estimates \eqref{4703}-\eqref{4707}
				into \eqref{4702}, then we complete the proof of this lemma.
			\end{proof}
			
			Finally, we will establish the estimate for the quantity $(\bwu, \bwb)$ in $L^2-$norm.
			\begin{lemm}\label{lemma32}
				Under the conditions of Theorem \ref{main_result_two},
				then it holds
				\beq\label{4801}
				\begin{aligned}
					&\frac{d}{dt} \|(\bwu, \bwb)|_{L^2}^2
					+\|(\p_1 \bwu, \nabla_h \bwb)\|_{L^2}^2\\
					\lesssim
					&\nu \mathcal{\overline{D}}(t)
					+(\mathcal{{B}}(t)\mathcal{{B}}_h(t)
					+\|\nabla_h (u^0, u^\ep)\|_{H^1}^{\frac{2}{3}}
					\|\nabla_h^2 (u^0, u^\ep)\|_{H^1}^{\frac{2}{3}})
					\mathcal{\overline{E}}(t)
					+\ep \|(\p_{22}u^\ep, \p_{33} u^\ep, \p_{33}b^\ep)\|_{H^1}
					\mathcal{\overline{E}}(t)^{\frac12},
				\end{aligned}
				\deq
				where $\nu$ is a small positive constant.
			\end{lemm}
			\begin{proof}
				The equation \eqref{402} yields directly
				\beq\label{4802}
				\begin{aligned}
					&\frac{d}{dt}\frac12 \i (|\bwu|^2+|\bwb|^2)dx
					+\i (|\p_1 \bwu|^2 +|\nabla_h \bwb|^2) dx\\
					=
					&\i(-u^\ep \cdot \nabla \bwu-\bu \cdot \nabla w^{u^0}
					+\bwu \cdot \nabla u^\ep+ w^{u^0}\cdot \nabla \bu)\cdot \bwu dx\\
					&+\i(b^\ep \cdot \nabla \bwb +\bb \cdot \nabla w^{b^0}
					-\bwb \cdot \nabla b^\ep-w^{b^0}\cdot \nabla \bb)\cdot \bwu dx\\
					&+\i (-u^\ep \cdot \nabla \bwb-\bu \cdot \nabla w^{b^0}
					-\nabla\bu_i\times \p_i b^\ep-\nabla u^0_i \times \p_i \bb)\cdot \bwb \ dx\\
					&+\i (b^\ep \cdot \nabla \bwu+\bb\cdot \nabla w^{u^0}
					+\nabla \bb_i \times \p_i u^\ep
					+\nabla b^0_i \times \p_i \bu)\cdot \bwb \ dx\\
					&+\ep \i (\p_{22}w^{u^\ep}+\p_{33}w^{u^\ep})\cdot \bwu \ dx
					+\ep \i \p_{33}w^{b^\ep} \cdot \bwb \ dx.
				\end{aligned}
				\deq
				Integrating by part and using the divergence-free condition, we write
				\beqq
				\i (u^\ep \cdot \nabla) \bwu \cdot \bwu dx
				=\i (u^\ep \cdot \nabla) \bwb \cdot \bwb dx=0,
				\deqq
				and
				\beqq
				\i(b^\ep \cdot \nabla)\bwb \cdot \bwu dx
				+\i (b^\ep \cdot \nabla) \bwu \cdot \bwb \ dx=0.
				\deqq
				Using the anisotropic type inequality \eqref{ie:Sobolev}, it is easy to check that
				\beq\label{4803}
				\begin{aligned}
					&\i (\bu \cdot \nabla) w^{u^0} \cdot \bwu dx\\
					=&\i (\bu_h \cdot \nabla_h) w^{u^0} \cdot \bwu dx
					+\i  \bu_3 \p_3 w^{u^0} \cdot \bwu dx\\
					\lesssim
					&\|\bu_h\|_{L^2}^{\frac12}\|\p_2 \bu_h\|_{L^2}^{\frac12}
					\|\nabla_h w^{u^0}\|_{L^2}^{\frac12}\|\p_3 \nabla_h w^{u^0}\|_{L^2}^{\frac12}
					\|\bwu\|_{L^2}^{\frac12}\|\p_1 \bwu\|_{L^2}^{\frac12}\\
					&+\|\bu_3\|_{L^2}^{\frac12}\|\p_3 \bu_3\|_{L^2}^{\frac12}
					\|\p_3 w^{u^0}\|_{L^2}^{\frac12}\|\p_2  \p_3 w^{u^0}\|_{L^2}^{\frac12}
					\|\bwu\|_{L^2}^{\frac12}\|\p_1\bwu\|_{L^2}^{\frac12}\\
					\lesssim
					&\nu(\|\p_1\bwu\|_{L^2}^2+\|\nabla_h \bu\|_{L^2}^2)
					+\|\bu\|_{H^1}^2\|u^0\|_{H^2}\|\nabla_h u^0\|_{H^2},
				\end{aligned}
				\deq
				and
				\beq\label{4804}
				\begin{aligned}
					&\i (\bwu \cdot \nabla) u^\ep \cdot \bwu dx\\
					=&\i (\bwu_h \cdot \nabla_h) u^\ep \cdot \bwu dx
					+\i \bwu_3 \p_3 u^\ep \cdot \bwu dx\\
					\lesssim
					&\|\nabla_h u^\ep\|_{L^2}^{\frac14}\|\p_3\nabla_h u^\ep\|_{L^2}^{\frac14}
					\|\p_2 \nabla_h u^\ep\|_{L^2}^{\frac14}\|\p_{23}\nabla_h u^\ep\|_{L^2}^{\frac14}
					\|\bwu\|_{L^2}^{\frac32}\|\p_1 \bwu\|_{L^2}^{\frac12}\\
					&+\|\bwu_3\|_{L^2}
					\|\p_3 u^\ep\|_{L^2}^{\frac14}\|\p_{33} u^\ep\|_{L^2}^{\frac14}
					\|\p_{23} u^\ep\|_{L^2}^{\frac14}\|\p_{233} u^\ep\|_{L^2}^{\frac14}
					\|\bwu\|_{L^2}^{\frac12}\|\p_1 \bwu\|_{L^2}^{\frac12}\\
					\lesssim
					&\nu(\|\p_1 \bwu\|_{L^2}^2+\|\nabla_h \bu\|_{L^2}^2)
					+\|\bu\|_{H^1}^2(\|\nabla_h u^\ep\|_{H^1}^{\frac{2}{3}}
					\|\nabla_h^2 u^\ep\|_{H^1}^{\frac{2}{3}}
					+\|u^\ep\|_{H^2}\|\nabla_h u^\ep\|_{H^2}).
				\end{aligned}
				\deq
				Similarly, we may write
				\beq\label{4805}
				\begin{aligned}
					\i(w^{u^0}\cdot \nabla) \bu\cdot \bwu dx
					\lesssim
					&\nu(\|\p_1 \bwu\|_{L^2}^2+\|\nabla_h \bu\|_{L^2}^2)
					+\|\bu\|_{H^1}^2(\|\nabla_h u^0\|_{H^1}^{\frac{2}{3}}
					\|\nabla_h^2 u^0\|_{H^1}^{\frac{2}{3}}
					+\|u^0\|_{H^2}\|\nabla_h u^0\|_{H^2}),\\
					\i(\bwb \cdot \nabla) b^\ep \cdot \bwu dx
					\lesssim
					&\nu(\|\p_1 \bwu\|_{L^2}^2+\|\nabla_h \bwb\|_{L^2}^2 )
					+\|(\bu, \bb)\|_{H^1}^2\|b^\ep\|_{H^1}\|\nabla_h b^\ep\|_{H^1},
				\end{aligned}
				\deq
				and
				\beq\label{4806}
				\begin{aligned}
					&|\i(\bb \cdot \nabla w^{b^0}-w^{b^0}\cdot \nabla \bb)\cdot \bwu dx|
					+|\i (\bu \cdot \nabla) w^{b^0}\cdot \bwb \ dx|
					+|\i (\bb\cdot \nabla) w^{u^0} \cdot \bwb \ dx|\\
					\lesssim
					&\nu(\|(\p_1 \bwu, \nabla_h \bwb)\|_{L^2}^2+\|(\nabla_h \bb, \p_1 \bu)\|_{H^1}^2)
					+\|(\bu, \bb)\|_{H^1}^2
					\|(u^0, b^0)\|_{H^2}\|\nabla_h (u^0, b^0)\|_{H^2}.
				\end{aligned}
				\deq
				Using the anisotropic type inequality \eqref{ie:Sobolev}, we have
				\beq\label{4807}
				\begin{aligned}
					&\i \nabla\bu_i\times \p_i b^\ep\cdot \bwb \ dx\\
					\lesssim
					&\|\nabla_h \bu\|_{L^2}
					\|\nabla b^\ep\|_{L^2}^{\frac14}\|\p_1\nabla b^\ep\|_{L^2}^{\frac14}
					\|\p_3 \nabla b^\ep\|_{L^2}^{\frac14}\|\p_{13}\nabla b^\ep\|_{L^2}^{\frac14}
					\|\bwb\|_{L^2}^{\frac12}\|\p_2 \bwb\|_{L^2}^{\frac12}\\
					&+\|\nabla \bu\|_{L^2}^{\frac12}\|\p_1 \nabla \bu\|_{L^2}^{\frac12}
					\|\nabla_h b^\ep\|_{L^2}^{\frac12}\|\p_3 \nabla_h b^\ep\|_{L^2}^{\frac12}
					\|\bwb\|_{L^2}^{\frac12}\|\p_2 \bwb\|_{L^2}^{\frac12}\\
					\lesssim
					&\nu(\|\p_2 \bwb\|_{L^2}^2+\|(\nabla_h \bu, \p_1 \nabla \bu)\|_{L^2}^2)
					+\|(\bu, \bb)\|_{H^1}^2\|b^\ep\|_{H^2}\|\nabla_h b^\ep\|_{H^2}.
				\end{aligned}
				\deq
				Similarly, it is easy to deduce that
				\beq\label{4808}
				\begin{aligned}
					\i \nabla \bb_i \times \p_i u^\ep \cdot \bwb \ dx
					\lesssim
					\nu(\|\p_2 \bwb\|_{L^2}^2+\|\nabla_h \nabla \bb\|_{L^2}^2)
					+\|(\bu, \bb)\|_{H^1}^2\|u^\ep\|_{H^1}\|\nabla_h u^\ep\|_{H^1},
				\end{aligned}
				\deq
				and
				\beq\label{4809}
				\begin{aligned}
					&|\i \nabla u^0_i \times \p_i \bb \cdot \bwb \ dx|
					+|\i \nabla b^0_i \times \p_i \bu \cdot \bwb \ dx|\\
					\lesssim
					&\nu(\|(\p_1 \bwu, \p_2 \bwb)\|_{L^2}^2+\|\nabla_h \bb\|_{H^1}^2
					+\|\nabla_h \bu\|_{L^2}^2+\|\p_1 \nabla \bu\|_{L^2}^2)\\
					&+\|(\bu, \bb)\|_{H^1}^2\|(u^0, b^0)\|_{H^2}\|\nabla_h(u^0, b^0)\|_{H^2}.
				\end{aligned}
				\deq
				Using the H\"{o}lder inequality, we have
				\beq\label{4810}
				\begin{aligned}
					&\ep \i (\p_{22}w^{u^\ep}+\p_{33}w^{u^\ep})\cdot \bwu \ dx
					+\ep \i \p_{33}w^{b^\ep} \cdot \bwb \ dx\\
					\le&
					\ep \|(\p_{22}w^{u^\ep}, \p_{33}w^{u^\ep}, \p_{33}w^{b^\ep})\|_{L^2}\|(\bwu, \bwb)\|_{L^2}.
				\end{aligned}
				\deq
				Substituting the estimates \eqref{4803}-\eqref{4810}
				into \eqref{4802}, then we complete the proof of this lemma.
			\end{proof}	
			
			\begin{proof}[\textbf{Proof of Theorem \ref{main_result_two}}]
				Let us define the energy norm
				$$
				\mathcal{\overline{E}}_q(t):=\mathcal{\overline{E}}(t)+\kappa\i  \p_2  \bb \cdot \bu \ dx,
				$$
				where $\kappa$ is a small constant.
				Then, it is easy to check that $\mathcal{\overline{E}}_q(t)$
				is equivalent to $\mathcal{\overline{E}}(t)$.
				Due to the smallness of $\nu$, the combination of estimates
				\eqref{4501}, \eqref{4601}, \eqref{4701} and \eqref{4801}
				yields directly
				\beqq
				\begin{aligned}
					\frac{d}{dt}\mathcal{\overline{E}}_q(t)
					\le C(\mathcal{{B}}(t)\mathcal{{B}}_h(t)
					+\|\nabla_h (u^0, u^\ep)\|_{H^1}^{\frac{2}{3}}
					\|\nabla_h^2 (u^0, u^\ep)\|_{H^1}^{\frac{2}{3}})
					\mathcal{\overline{E}}_q(t)
					+C\ep \|(\p_{22}u^\ep, \p_{33} u^\ep, \p_{33}b^\ep)\|_{H^1}
					\mathcal{\overline{E}}(t)^{\frac12},
				\end{aligned}
				\deqq
				which, together with the Gronwall inequality, yields directly
				\beq\label{403}
				\begin{aligned}
					\mathcal{\overline{E}}_q(t)
					\le
					&C \ep \int_0^t \|(\p_{22}u^\ep, \p_{33} u^\ep, \p_{33}b^\ep)(\tau)\|_{H^1}\mathcal{\overline{E}}(\tau)^{\frac12} d\tau\\
					&\times\exp\left\{\int_0^t (\mathcal{{B}}(\tau)\mathcal{{B}}_h(\tau)
					+\|\nabla_h (u^0, u^\ep)(\tau)\|_{H^1}^{\frac{2}{3}}
					\|\nabla_h^2 (u^0, u^\ep)(\tau)\|_{H^1}^{\frac{2}{3}}) d\tau\right\}.
				\end{aligned}
				\deq
				With the help of the decay estimates
				\eqref{decay-assumption-new-close} and \eqref{decay-limit}, it is easy to check that
				\beq\label{404}
				\begin{aligned}
					&\int_0^t \mathcal{{B}}(\tau)\mathcal{{B}}_h(\tau)d\tau
					\lesssim
					\underset{0\le \tau \le t}{\sup}[(1+\tau)^{\frac{s}{2}}\mathcal{{B}}(\tau)]
					\left\{\int_0^t (1+\tau)^{\sigma} \mathcal{{B}}_h(\tau)^2 d\tau\right\}^{\frac12}
					\left\{\int_0^t (1+\tau)^{-(\sigma+s)} d\tau\right\}^{\frac12}
					\lesssim 1,\\
					&\int_0^t \|\nabla_h (u^0, u^\ep)(\tau)\|_{H^2}^{\frac{4}{3}}d\tau
					\lesssim
					\left\{\int_0^t (1+\tau)^{\sigma}\|\nabla_h (u^0, u^\ep)(\tau)\|_{H^2}^2d\tau
					\right\}^{\frac{2}{3}}
					\left\{\int_0^t (1+\tau)^{-2\sigma}d\tau
					\right\}^{\frac{1}{3}}
					\lesssim 1,
				\end{aligned}
				\deq
				and
				\beq\label{405}
				\begin{aligned}
					&\sqrt{\ep} \int_0^t \|(\p_{22}u^\ep, \p_{33} u^\ep, \p_{33}b^\ep)(\tau)\|_{H^1}\mathcal{\overline{E}}(\tau)^{\frac12} d\tau\\
					\lesssim
					&
					\underset{0\le \tau \le t}{\sup}[(1+\tau)^{\frac{s}{2}}\|(u^0, b^0, u^\ep, b^\ep)(\tau)\|_{H^1}]
					\left\{\int_0^t (1+\tau)^{-(\sigma+s)}d\tau\right\}^{\frac12}\\
					&\times \left\{\ep \int_0^t (1+\tau)^{\sigma}\|(\p_{22}u^\ep, \p_{33} u^\ep, \p_{33}b^\ep)(\tau)\|_{H^1}^2d\tau\right\}^{\frac12}
					\lesssim 1.
				\end{aligned}
				\deq
				Substituting the estimates \eqref{404} and \eqref{405} into \eqref{403}, we have
				\beqq
				\mathcal{\overline{E}}(t)
				\le
				C\mathcal{\overline{E}}_q(t)
				\le
				C \ep^{\frac12},
				\deqq
				which, together with \eqref{ie:Sobolev}, \eqref{uniform_estimate}
				and \eqref{decay-limit}, yields directly
				\beqq
				\|(\bu, \bb)(t)\|_{L^\infty}
				\lesssim \|(\bu, \bb)(t)\|_{H^1}^{\frac12}
				\|\nabla^2(\bu, \bb)(t)\|_{L^2}^{\f38}
				\|\p_{123}(\bu, \bb)(t)\|_{L^2}^{\f18}
				\le
				C \ep^{\frac{1}{8}}.
				\deqq
				Therefore, we complete the proof of Theorem \ref{main_result_two}.
			\end{proof}

			\section*{Acknowledgments}
			Jincheng Gao was partially supported by the National Key Research and Development Program of China(2021YFA1002100), Guangdong Special Support Project (2023TQ07A961) and Guangzhou Science and Technology Program (2024A04J6410).
			Jiahong Wu was partially supported by the
            National Science Foundation of the United States (DMS 2104682, DMS 2309748).
			Zheng-an Yao was partially supported by
			National Key Research and Development Program of China (2020YFA0712500).

           \section*{Data Availibility}
           Data sharing is not applicable to this article as no new data were created or analysed in this study.

           \section*{Conflict of interest}
           The authors declared that they have no Conflict of interest to this work.

			\begin{appendices}
				\section{Some useful inequalities}\label{usefull-inequality}	
		Now let us state some  anisotropic Sobolev inequalities used frequently in our paper.
				\begin{lemm}\label{lemm:sobolev-ie}
					For any suitable functions $(f(x), g(x), h(x))$ defined on $\mathbb{R}^3_+$ and different numbers $i,j,k \in \{1, 2, 3\}$,
					the following estimates hold
					\beq \label{ie:Sobolev}
					\bal
					\|f\|_{L^\infty} &\; \lesssim \|f\|_{L^2}^{\f18}\|\p_1 f\|_{L^2}^{\f18}\|\p_2 f\|_{L^2}^{\f18}\|\p_{12} f\|_{L^2}^{\f18}
					\|\p_3 f\|_{L^2}^{\f18}\|\p_{13} f\|_{L^2}^{\f18}\|\p_{23} f\|_{L^2}^{\f18}\|\p_{123} f\|_{L^2}^{\f18},\\
					\int_{\mathbb{R}^3_+} |f g h |\, dx
					&\; \lesssim \|f\|_{L^2}
					\|g\|_{L^2}^{\frac12}
					\|\p_i g\|_{L^2}^{\frac12}
					\|h\|_{L^2}^{\frac14}
					\|\p_j h\|_{L^2}^{\frac14}
					\|\p_{k} h\|_{L^2}^{\frac14}
					\|\p_{jk} h\|_{L^2}^{\frac14},\\
					\int_{\mathbb{R}^3_+} |f g h |\, dx
					&\; \lesssim \|f\|_{L^2}^{\frac12}
					\|\p_1 f\|_{L^2}^{\frac12}
					\|g\|_{L^2}^{\frac12}
					\|\p_2 g\|_{L^2}^{\frac12}
					\|h\|_{L^2}^{\frac12}
					\|\p_3 h\|_{L^2}^{\frac12},\\
					\|Z_3 f\|_{L^2}
					&\lesssim \|f\|_{L^2}
					+\|f\|_{L^2}^{\frac34}\|Z_3^3 f\|_{L^2}^{\frac14}
					+\|f\|_{L^2}^{\frac23}\|Z_3^3 f\|_{L^2}^{\frac13},\\
					\left\|\|f\|_{L^\infty(\mathbb{R}_+)}\right\|_{L^{\frac2s}(\mathbb{R}^2)}
					&\lesssim \left(\|f\|_{L^2}\|\partial_2 f\|_{L^2}
					+\|\partial_1 f\|_{L^2}\|\partial_{12}f\|_{L^2}\right)^{\frac{1-s}{2}}
					\|f\|_{L^2}^{\frac{2s-1}{2}}
					\|\partial_3 f\|_{L^2}^{\frac12}.
					\dal
					\deq
				\end{lemm}
				
				\begin{proof}
					First of all, one can follow the idea in \cite[Lemma 1.2]{Wu2021Advance} to establish the inequalities $\eqref{ie:Sobolev}_1$-$\eqref{ie:Sobolev}_3$.
					Let us give the proof of inequality $\eqref{ie:Sobolev}_4$ and $\eqref{ie:Sobolev}_5$.
					Indeed, integrating by part, it holds
					\begin{equation*}
						\|Z_3 f\|_{L^2}^2=-\int_{\mathbb{R}_+^3}(\varphi' Z_3 f+Z_3^2 f)f dx
						\lesssim \|Z_3 f\|_{L^2}\|f\|_{L^2}
						+\|Z_3^2 f\|_{L^2}\|f\|_{L^2},
					\end{equation*}
					which yields directly
					\begin{equation}\label{a10}
						\|Z_3 f\|_{L^2}
						\lesssim \|f\|_{L^2}
						+\|f\|_{L^2}^{\frac12}\|Z_3^2 f\|_{L^2}^{\frac12}.
					\end{equation}
					Applying estimate \eqref{a10} twice, it is easy to check that
					\begin{equation*}
						\begin{aligned}
							\|Z_3^2 f\|_{L^2}
							&\lesssim \|Z_3 f\|_{L^2}
							+\|Z_3 f\|_{L^2}^{\frac12}\|Z_3^3 f\|_{L^2}^{\frac12}\\
							&\lesssim \|f\|_{L^2}
							+\|f\|_{L^2}^{\frac12}\|Z_3^2 f\|_{L^2}^{\frac12}
							+\left(\|f\|_{L^2}
							+\|f\|_{L^2}^{\frac12}\|Z_3^2 f\|_{L^2}^{\frac12}\right)^{\frac12}
							\|Z_3^3 f\|_{L^2}^{\frac12}\\
							&\lesssim \frac12 \|Z_3^2 f\|_{L^2}+\|f\|_{L^2}
							+\|f\|_{L^2}^{\frac12}\|Z_3^3 f\|_{L^2}^{\frac12}
							+\|f\|_{L^2}^{\frac13}\|Z_3^3 f\|_{L^2}^{\frac23},
						\end{aligned}
					\end{equation*}
					which implies directly
					\begin{equation}\label{a11}
						\|Z_3^2 f\|_{L^2}
						\lesssim \|f\|_{L^2}
						+\|f\|_{L^2}^{\frac12}\|Z_3^3 f\|_{L^2}^{\frac12}
						+\|f\|_{L^2}^{\frac13}\|Z_3^3 f\|_{L^2}^{\frac23}.
					\end{equation}
					Then, the combination of estimates \eqref{a10} and \eqref{a11} yields directly
					\begin{equation*}
						\begin{aligned}
							\|Z_3 f\|_{L^2}
							&\lesssim \|f\|_{L^2}
							+\|f\|_{L^2}^{\frac12}
							\left\{\|f\|_{L^2}
							+\|f\|_{L^2}^{\frac12}\|Z_3^3 f\|_{L^2}^{\frac12}
							+\|f\|_{L^2}^{\frac13}\|Z_3^3 f\|_{L^2}^{\frac23}\right\}^{\frac12}\\
							&\lesssim \|f\|_{L^2}
							+\|f\|_{L^2}^{\frac34}\|Z_3^3 f\|_{L^2}^{\frac14}
							+\|f\|_{L^2}^{\frac23}\|Z_3^3 f\|_{L^2}^{\frac13},
						\end{aligned}
					\end{equation*}
					which yields the estimate $\eqref{ie:Sobolev}_4$.
					Finally, let us establish the estimate $\eqref{ie:Sobolev}_5$.
					Indeed, it holds
					\begin{equation}\label{a6}
						\begin{aligned}
							\left\|\|f\|_{L^\infty(\mathbb{R}_+)}\right\|_{L^{\frac2s}(\mathbb{R}^2)}^{\frac2s}
							\lesssim \left\|\|f\|_{L^2(\mathbb{R}_+)}^{\frac12}
							\|\partial_3 f\|_{L^2(\mathbb{R}_+)}^{\frac12}\right\|_{L^{\frac2s}(\mathbb{R}^2)}^{\frac2s}
							\lesssim \|\partial_3 f\|_{L^2}^{\frac1s}
							\left\|\|f\|_{L^2(\mathbb{R}_+)}\right\|_{L^{\frac{2}{2s-1}}
								(\mathbb{R}^2)}^{\frac1s},
						\end{aligned}
					\end{equation}
					and
					\begin{equation}\label{a7}
						\begin{aligned}
							\left\|\|f\|_{L^2(\mathbb{R}_+)}\right\|_{L^{\frac{2}{2s-1}}(\mathbb{R}^2)}
							^{\frac{2}{2s-1}}
							&\lesssim \left\|\|f\|_{L^\infty(\mathbb{R}^2)}\right\|_{L^2(\mathbb{R}_+)}
							^{\frac{4-4s}{2s-1}}\|f\|_{L^2}^2\\
							&\lesssim \left(\|f\|_{L^2}^{\frac12}\|\partial_2 f\|_{L^2}^{\frac12}
							+\|\partial_1 f\|_{L^2}^{\frac12}\|\partial_{12}f\|_{L^2}^{\frac12}\right)
							^{\frac{4-4s}{2s-1}}\|f\|_{L^2}^2.
						\end{aligned}
					\end{equation}
					The combination of estimates \eqref{a6} and \eqref{a7} gives
					\begin{equation*}
						\begin{aligned}
							\left\|\|f\|_{L^\infty(\mathbb{R}_+)}\right\|_{L^{\frac2s}(\mathbb{R}^2)}
							&\lesssim \|\partial_3 f\|_{L^2}^{\frac12}
							\left\|\|f\|_{L^2(\mathbb{R}_+)}\right\|_{L^{\frac{2}{2s-1}}
								(\mathbb{R}^2)}^{\frac12}\\
							&\lesssim \left(\|f\|_{L^2}\|\partial_2 f\|_{L^2}
							+\|\partial_1 f\|_{L^2}\|\partial_{12}f\|_{L^2}\right)^{\frac{1-s}{2}}
							\|f\|_{L^2}^{\frac{2s-1}{2}}
							\|\partial_3 f\|_{L^2}^{\frac12}.
						\end{aligned}
					\end{equation*}
					Therefore, we complete the proof of this lemma.
				\end{proof}
				
				Next, we introduce the Hardy-Littlewood-Sobolev inequality
				in \cite[pp. 119, Theorem 1]{Stein1970} as follow.
				\begin{lemm}\label{H-L}
					Let $0<\alpha<2 , 1<p<q<\infty, \frac{1}{q}+\frac{\alpha}{2}=\frac{1}{p}$, then
					\begin{equation}\label{a16}
						\|\Lambda_h^{-\alpha}f\|_{L^q(\mathbb{R}^2)}\lesssim \|f\|_{L^p(\mathbb{R}^2)}.
					\end{equation}
				\end{lemm}
				In this paper, taking $q=2$ in \eqref{a16}, then
				$
				p=\frac{1}{\frac12+\frac{\alpha}{2}}=\frac{2}{1+\alpha}
				$
				should satisfy the condition
				$$
				1<p=\frac{2}{1+\alpha} <2.
				$$
				This implies the index $\alpha \in (0, 1)$.

				\section{Proof of claimed estimate}\label{claim-estimates}
				In this section, we will give proof for the claimed estimate \eqref{claim-estimate}.
				This process is complicated and long, we will establish the estimate in detail.
				
				\textbf{Estimate of term $II_1$.} Since the velocity field only has
				lower order dissipative structure in the $x_2$ direction.
				Thus, it is somewhat complicated to deal with $II_1$ term.
				Indeed, it holds
				\beqq
				\begin{aligned}
					II_1=
					&\sum_{0\le \beta \le \alpha}C^{\beta}_{\alpha}
					\i (-\p_3Z^{\beta}u_h \cdot Z^{\alpha-\beta}\nabla_h \wu
					-Z^{\beta}u_h \cdot \p_3 Z^{\alpha-\beta}\nabla_h \wu)\p_3 Z^{\alpha} \wu dx\\
					&+\sum_{0\le \beta \le \alpha}C^{\beta}_{\alpha}
					\i (-\p_3Z^{\beta}u_3 \cdot Z^{\alpha-\beta}\p_3 \wu
					-Z^{\beta}u_3 \cdot \p_3 Z^{\alpha-\beta}\p_3 \wu)\p_3 Z^{\alpha} \wu dx\\
					:=&II_{11}+II_{12}.
				\end{aligned}
				\deqq
				Estimate of term $II_{11}$.
				If $\beta=0$, integrating by part and using the anisotropic type
				inequality \eqref{ie:Sobolev}, we have
				\beqq
				\begin{aligned}
					&|\i (-\p_3 u_h \cdot Z^{\alpha}\nabla_h \wu
					-u_h \cdot \p_3 Z^{\alpha}\nabla_h \wu)\p_3 Z^{\alpha} \wu dx|\\
					\lesssim
					&\|\p_3 u_h\|_{L^2}^{\frac14}
					\|\p_{33} u_h\|_{L^2}^{\frac14}
					\|\p_{23} u_h\|_{L^2}^{\frac14}
					\|\p_{233} u_h\|_{L^2}^{\frac14}
					\|Z^{\alpha}\nabla_h \wu\|_{L^2}
					\|\p_3 Z^{\alpha} \wu\|_{L^2}^{\frac12}
					\|\p_{13} Z^{\alpha} \wu\|_{L^2}^{\frac12}\\
					&+\|\nabla_h \cdot u_h\|_{L^2}^{\frac14}
					\|\p_3\nabla_h \cdot u_h\|_{L^2}^{\frac14}
					\|\p_2\nabla_h \cdot u_h\|_{L^2}^{\frac14}
					\|\p_{23}\nabla_h \cdot u_h\|_{L^2}^{\frac14}
					\|\p_3 Z^{\alpha} \wu\|_{L^2}^{\frac32}
					\|\p_{13} Z^{\alpha} \wu\|_{L^2}^{\frac12}\\
					\lesssim
					&\sqrt{\me^m(t)}\md^{m}(t).
				\end{aligned}
				\deqq
				If $0<\beta \le \alpha$, then we may write
				\beqq
				\begin{aligned}
					&|\i (\p_3Z^{\beta}u_h \cdot Z^{\alpha-\beta}\nabla_h \wu
					+Z^{\beta}u_h \cdot \p_3 Z^{\alpha-\beta}\nabla_h \wu)\p_3 Z^{\alpha} \wu dx|\\
					\lesssim
					&\left(\|\p_3Z^{\beta}u_h\|_{L^2}^{\frac12}
					\|\p_{23} Z^{\beta}u_h\|_{L^2}^{\frac12}
					\|Z^{\alpha-\beta}\nabla_h \wu\|_{L^2}^{\frac12}
					\|\p_3 Z^{\alpha-\beta}\nabla_h \wu\|_{L^2}^{\frac12}\right.\\
					&\left.+\|Z^{\beta}u_h\|_{L^2}^{\frac14}
					\|\p_2 Z^{\beta}u_h\|_{L^2}^{\frac14}
					\|\p_3 Z^{\beta}u_h\|_{L^2}^{\frac14}
					\|\p_{23} Z^{\beta}u_h\|_{L^2}^{\frac14}
					\|\p_3 Z^{\alpha-\beta}\nabla_h \wu\|_{L^2}\right)\\
					&\times \|\p_3 Z^{\alpha} \wu\|_{L^2}^{\frac12}
					\|\p_{13} Z^{\alpha} \wu\|_{L^2}^{\frac12}\\
					\lesssim
					&\sqrt{\me^m(t)}\md^{m}(t).
				\end{aligned}
				\deqq
				Thus, the term $II_{11}$ can be bounded by
				\beq\label{3403}
				|II_{11}|\lesssim \sqrt{\me^m(t)}\md^{m}(t).
				\deq
				Estimate of term $II_{12}$.
				If $\beta=0$, using the divergence-free condition
				and integrating by part, we have
				\beq\label{3404}
				\begin{aligned}
					&-\i \p_3 u_3  Z^{\alpha}\p_3 \wu \cdot \p_3 Z^{\alpha} \wu dx\\
					=&\i \p_1 u_1 Z^{\alpha}\p_3 \wu \cdot \p_3 Z^{\alpha} \wu  dx
					+\i \p_2 u_2 Z^{\alpha}\p_3 \wu \cdot \p_3 Z^{\alpha} \wu  dx\\
					\lesssim
					&\|u_1\|_{L^2}^{\frac14}\|\p_3 u_1\|_{L^2}^{\frac14}
					\|\p_2 u_1\|_{L^2}^{\frac14}\|\p_{23} u_1\|_{L^2}^{\frac14}
					\|\p_1 Z^{\alpha}\p_3 \wu\|_{L^2}
					\|\p_3 Z^{\alpha} \wu\|_{L^2}^{\frac12}
					\|\p_{13} Z^{\alpha} \wu\|_{L^2}^{\frac12}\\
					&+\|u_1\|_{L^2}^{\frac12}\|\p_3 u_1\|_{L^2}^{\frac12}
					\|\p_2 u_1\|_{L^2}^{\frac12}\|\p_{23} u_1\|_{L^2}^{\frac12}
					\|Z^{\alpha}\p_3 \wu\|_{L^2}^{\frac12}
					\|\p_1 Z^{\alpha}\p_3 \wu\|_{L^2}^{\frac12}
					\|\p_{13} Z^{\alpha} \wu\|_{L^2}\\
					&+\i |\p_2 u_2 Z^{\alpha}\p_3 \wu \cdot \p_3 Z^{\alpha} \wu|dx\\
					\lesssim
					&\sqrt{\me^m(t)}\md^{m}(t)+\i |\p_2 u_2 Z^{\alpha}\p_3 \wu \cdot \p_3 Z^{\alpha} \wu|dx.
				\end{aligned}
				\deq
				On the other hand, it is easy to check that
				\beqq
				\begin{aligned}
					&-\i u_3 \p_3 Z^{\alpha}\p_3 \wu \p_3 Z^{\alpha} \wu dx\\
					=&-\i u_3 \p_{33} Z^{\alpha} \wu \p_3 Z^{\alpha} \wu dx
					+\i u_3 \p_3 (\p_3 Z^{\alpha} \wu-Z^{\alpha}\p_3 \wu)\p_3 Z^{\alpha} \wu dx\\
					=&-\frac12\i \p_3 u_3  |\p_3 Z^{\alpha} \wu|^2 dx
					+\i u_3 \p_3 (\p_3 Z^{\alpha} \wu-Z^{\alpha}\p_3 \wu)\p_3 Z^{\alpha} \wu dx,
				\end{aligned}
				\deqq
				and
				\beqq
				\p_3 Z^{\alpha} \wu-Z^{\alpha}\p_3 \wu
				=\sum_{0< \beta_3 \le \alpha_3}C^{\beta_3}_{\alpha_3}
				Z^{\beta_3}(\frac{1}{\varphi})Z^{\alpha_3-\beta_3+e_3+\alpha_h}\wu.
				\deqq
				Thus, we can obtain
				\beq\label{3405}
				\begin{aligned}
					&\i u_3 \p_3 (\p_3 Z^{\alpha} \wu-Z^{\alpha}\p_3 \wu)\p_3 Z^{\alpha} \wu dx\\
					=&C^{e_3}_{\alpha_3}
					\i u_3 Z^{e_3}(\frac{1}{\varphi})|\p_3  Z^{\alpha}\wu |^2dx
					+C^{e_3}_{\alpha_3}
					\i u_3 \p_3 Z^{e_3}(\frac{1}{\varphi})Z^{\alpha}\wu\p_3 Z^{\alpha} \wu dx\\
					&+\sum_{1< \beta_3 \le \alpha_3}C^{\beta_3}_{\alpha_3}
					\i u_3 \p_3 Z^{\beta_3}(\frac{1}{\varphi})Z^{\alpha_3-\beta_3+e_3+\alpha_h}\wu
					\p_3 Z^{\alpha} \wu dx\\
					&+\sum_{1< \beta_3 \le \alpha_3}C^{\beta_3}_{\alpha_3}
					\i u_3  Z^{\beta_3}(\frac{1}{\varphi}) \p_3 Z^{\alpha_3-\beta_3+e_3+\alpha_h}\wu
					\p_3 Z^{\alpha} \wu dx\\
					\lesssim
					&\sqrt{\me^m(t)}\md^{m}(t)
					+|\i u_3 Z^{e_3}(\frac{1}{\varphi})|\p_3  Z^{\alpha}\wu |^2dx|.
				\end{aligned}
				\deq
				If $0<\beta \le \alpha$, using the anisotropic type inequality \eqref{ie:Sobolev}
				and divergence-free condition, we get
				\beq\label{3406}
				\begin{aligned}
					&|\i \p_3Z^{\beta}u_3 \cdot Z^{\alpha-\beta}\p_3 \wu \cdot \p_3 Z^{\alpha} \wu dx|\\
					\lesssim
					&\|\p_3 Z^{\beta}u_3\|_{L^2}^{\frac12}\|\p_{33} Z^{\beta}u_3 \|_{L^2}^{\frac12}
					\|Z^{\alpha-\beta}\p_3 \wu\|_{L^2}^{\frac12}\|\p_2 Z^{\alpha-\beta}\p_3 \wu\|_{L^2}^{\frac12}
					\|\p_3 Z^{\alpha} \wu\|_{L^2}^{\frac12}\|\p_{13} Z^{\alpha} \wu\|_{L^2}^{\frac12}\\
					\lesssim
					&\sqrt{\me^m(t)}\md^{m}(t).
				\end{aligned}
				\deq
				If  $2\le |\beta|\le |\alpha|$, then we may write
				\beq\label{3407}
				\begin{aligned}
					&\i Z^{\beta}u_3 \cdot \p_3 Z^{\alpha-\beta}\p_3 \wu \p_3 Z^{\alpha} \wu dx\\
					\lesssim
					&\|\p_3Z^{\beta}u_3\|_{L^2}^{\frac12}
					\|\p_{33} Z^{\beta}u_3\|_{L^2}^{\frac12}
					\|Z^{\alpha-\beta+e_3}\p_3 \wu\|_{L^2}^{\frac12}
					\|\p_2 Z^{\alpha-\beta+e_3}\p_3 \wu\|_{L^2}^{\frac12}
					\|\p_3 Z^{\alpha} \wu\|_{L^2}^{\frac12}
					\|\p_{13} Z^{\alpha} \wu\|_{L^2}^{\frac12}\\
					&+\|Z^{\beta}u_3\|_{L^2}^{\frac12}
					\|\p_3 Z^{\beta}u_3\|_{L^2}^{\frac12}
					\|Z^{\alpha-\beta+e_3}\p_3 \wu\|_{L^2}^{\frac12}
					\|\p_2 Z^{\alpha-\beta+e_3}\p_3 \wu\|_{L^2}^{\frac12}
					\|\p_3 Z^{\alpha} \wu\|_{L^2}^{\frac12}
					\|\p_{13} Z^{\alpha} \wu\|_{L^2}^{\frac12}\\
					\lesssim
					&\sqrt{\me^m(t)}\md^{m}(t).
				\end{aligned}
				\deq
				If $\beta=e_i(i=1,2)$, which implies $\alpha$ includes $e_i(i=1,2)$, then we have
				\beq\label{3408}
				\begin{aligned}
					&-\i Z^{e_i}u_3 \cdot \p_3 Z^{\alpha-\beta}\p_3 \wu \p_3 Z^{\alpha} \wu dx\\
					\lesssim
					&\|Z^{e_i}u_3 \|_{L^2}^{\frac14}
					\|\p_3 Z^{e_i}u_3\|_{L^2}^{\frac14}
					\|\p_2 Z^{e_i}u_3 \|_{L^2}^{\frac14}
					\|\p_{23}Z^{e_i}u_3\|_{L^2}^{\frac14}\\
					&\times \|Z^{\alpha-e_i+e_3}\p_3 \wu\|_{L^2}^{\frac12}
					\|\p_1 Z^{\alpha-e_i+e_3}\p_3 \wu\|_{L^2}^{\frac12}
					\|\p_3 Z^{\alpha} \wu\|_{L^2}\\
					&+\|\p_3 Z^{e_i}u_3 \|_{L^2}^{\frac14}
					\|\p_{33} Z^{e_i}u_3\|_{L^2}^{\frac14}
					\|\p_{23} Z^{e_i}u_3 \|_{L^2}^{\frac14}
					\|\p_{233}Z^{e_i}u_3\|_{L^2}^{\frac14}\\
					&\times \|Z^{\alpha-e_i+e_3}\p_3 \wu\|_{L^2}^{\frac12}
					\|\p_1 Z^{\alpha-e_i+e_3}\p_3 \wu\|_{L^2}^{\frac12}
					\|\p_3 Z^{\alpha} \wu\|_{L^2}\\
					\lesssim
					&\sqrt{\me^m(t)}\md^{m}(t).
				\end{aligned}
				\deq
				If $\beta=e_3$, we apply the divergence-free condition to obtain
				\beq\label{3409}
				\begin{aligned}
					&-\i Z^{e_3}u_3 \cdot \p_3 Z^{\alpha-e_3}\p_3 \wu \p_3 Z^{\alpha} \wu dx\\
					=&\i (\p_1 u_1+\p_2 u_2)Z^{\alpha}\p_3 \wu \p_3 Z^{\alpha} \wu dx
					\lesssim
					\sqrt{\me^m(t)}\md^{m}(t)
+\i |\p_2 u_2 Z^{\alpha}\p_3 \wu \cdot \p_3 Z^{\alpha} \wu|dx.
				\end{aligned}
				\deq
				The combination of estimates \eqref{3404}-\eqref{3409}  yields directly
				\beqq
				II_{12}\lesssim
				\sqrt{\me^m(t)}\md^{m}(t)
+\i |\p_2 u_2 Z^{\alpha}\p_3 \wu \cdot \p_3 Z^{\alpha} \wu|dx
				+|\i u_3 Z^{e_3}(\frac{1}{\varphi})|\p_3  Z^{\alpha}\wu |^2dx|,
				\deqq
				which, together with estimate \eqref{3403}, gives directly
				\beq\label{3410}
				II_1\lesssim
				\sqrt{\me^m(t)}\md^{m}(t)
+\i |\p_2 u_2 Z^{\alpha}\p_3 \wu \cdot \p_3 Z^{\alpha} \wu|dx
				+|\i u_3 Z^{e_3}(\frac{1}{\varphi})|\p_3  Z^{\alpha}\wu |^2dx|.
				\deq
				Using the Sobolev inequality \eqref{ie:Sobolev}
				and divergence-free condition, we have
				\beqq
				\begin{aligned}
					&\i |\p_2 u_2 Z^{\alpha}\p_3 \wu \cdot \p_3 Z^{\alpha} \wu|dx\\
					\lesssim
					&\|\p_2 u_2\|_{L^2}^{\frac14}\|\p_{23} u_2\|_{L^2}^{\frac14}
					\|\p_{22} u_2\|_{L^2}^{\frac14}\|\p_{322} u_2\|_{L^2}^{\frac14}
					\|\p_{13} Z^{\alpha} \wu\|_{L^2}^{\frac{1}{2}}
					\|\p_3 \wu\|_{H^{m-2}_{co}}^{\frac32}
				\end{aligned}
				\deqq
				and
				\beqq
				\begin{aligned}
					&|\i u_3 Z^{e_3}(\frac{1}{\varphi})|\p_3  Z^{\alpha}\wu |^2dx|\\
					\lesssim
					&\|\nabla_h \cdot u_h\|_{L^2}^{\frac14}
					\|\p_3 (\nabla_h \cdot u_h)\|_{L^2}^{\frac14}
					\|\p_2 (\nabla_h \cdot u_h)\|_{L^2}^{\frac14}
					\|\p_{23} (\nabla_h \cdot u_h)\|_{L^2}^{\frac14}
					\|\p_{13} Z^{\alpha} \wu\|_{L^2}^{\frac{1}{2}}
					\|\p_3 \wu\|_{H^{m-2}_{co}}^{\frac32}.
				\end{aligned}
				\deqq
				Thus, we can obtain the estimate
				\beq\label{3411}
				II_1\lesssim
				\sqrt{\me^m(t)}\md^{m}(t)
				+\|(\nabla_h u_h, \nabla_h \p_3 u_h)\|_{H^1_{tan}}
				\|\p_{13} Z^{\alpha} \wu\|_{L^2}^{\frac{1}{2}}
				\|\p_3 \wu\|_{H^{m-2}_{co}}^{\frac32}.
				\deq
				\textbf{Estimate of term $II_2$.}
				Using the anisotropic type inequality \eqref{ie:Sobolev} repeatedly, we may conclude
				\beq\label{3412}
				\begin{aligned}
					II_2
					=&\sum_{0\le \beta < \alpha}C^{\beta}_{\alpha}
					\i \p_3 Z^\beta \wu_h \cdot Z^{\al-\beta}\nabla_h u \cdot \p_3 Z^\al \wu dx
					+\i \p_3 Z^\al \wu_h \cdot \nabla_h u \cdot \p_3 Z^\al \wu dx\\
					&+\!\!\sum_{0\le \beta \le \alpha}\!\!C^{\beta}_{\alpha}\!\!
					\i Z^\beta \wu_h \cdot \p_3 Z^{\al-\beta}\nabla_h u \cdot \p_3 Z^\al \wu dx
					+\sum_{0\le \beta \le \alpha}C^{\beta}_{\alpha}
					\i \p_3 Z^\beta \wu_3 Z^{\al-\beta}\p_3 u \cdot \p_3 Z^\al \wu dx\\
					&
					+\sum_{0< \beta \le \alpha}C^{\beta}_{\alpha}
					\i Z^\beta \wu_3 \p_3 Z^{\al-\beta}\p_3 u \cdot \p_3 Z^\al \wu dx
					+\i \wu_3 \p_3 Z^{\al}\p_3 u \cdot \p_3 Z^\al \wu dx\\
					&\lesssim
					\sqrt{\me^m(t)}\md^{m}(t)
					+\|(\nabla_h u_h, \nabla_h \p_3 u_h)\|_{H^1_{tan}}
					\|\p_3 \wu\|_{H^{m-2}_{co}}^{\frac32}\|\p_{13} Z^\al \wu\|_{L^2}^{\frac12},
				\end{aligned}
				\deq
				where we have used the following estimate
				\beqq
				\begin{aligned}
					&\i \wu_3 \p_3 Z^{\al}\p_3 u \cdot \p_3 Z^\al \wu dx\\
					=&\i \p_1 u_2 \p_3 Z^{\al}\p_3 u \cdot \p_3 Z^\al \wu dx
					-\i \p_2 u_1 \p_3 Z^{\al}\p_3 u \cdot \p_3 Z^\al \wu dx\\
					=&-\i u_2 \p_1(\p_3 Z^{\al}\p_3 u \cdot \p_3 Z^\al \wu)dx
					-\i \p_2 u_1 \p_3 Z^{\al}\p_3 u \cdot \p_3 Z^\al \wu dx\\
					\lesssim
					&\sqrt{\me^m(t)}\md^{m}(t)
					+\|(\nabla_h u_h, \nabla_h \p_3 u_h)\|_{H^1_{tan}}
					\|\p_3 \wu\|_{H^{m-2}_{co}}^{\frac32}\|\p_{13} Z^\al \wu\|_{L^2}^{\frac12}.
				\end{aligned}
				\deqq
				\textbf{Estimate of terms $II_3$ and $II_6$.}
				It is easy to check that
				\beqq
				\begin{aligned}
					II_3
					=&\sum_{0\le \beta \le \alpha}C^{\beta}_{\alpha}
					\i \{\p_3 (Z^\al b_h \cdot Z^{\alpha-\beta}\nabla_h \wb)
					+\p_3 Z^\al b_3 Z^{\alpha-\beta}\p_3 \wb\}
					\cdot \p_3 Z^\al \wu dx\\
					&+\sum_{0< \beta \le \alpha}C^{\beta}_{\alpha}
					\i  Z^\beta b_3 \p_3Z^{\alpha-\beta}\p_3 \wb
					\cdot \p_3 Z^\al \wu dx
					+\i  b_3 \p_3Z^{\alpha}\p_3 \wb\cdot \p_3 Z^\al \wu dx,
				\end{aligned}
				\deqq
				and
				\beqq
				\begin{aligned}
					II_6
					=&\sum_{0\le \beta \le \alpha}C^{\beta}_{\alpha}
					\i \{\p_3 (Z^\al b_h \cdot Z^{\alpha-\beta}\nabla_h \wu)
					+\p_3 Z^\al b_3 Z^{\alpha-\beta}\p_3 \wu\}
					\cdot \p_3 Z^\al \wb dx\\
					&+\sum_{0< \beta \le \alpha}C^{\beta}_{\alpha}
					\i  Z^\beta b_3 \p_3Z^{\alpha-\beta}\p_3 \wu
					\cdot \p_3 Z^\al \wb dx
					+\i  b_3 \p_3Z^{\alpha}\p_3 \wu \cdot \p_3 Z^\al \wb dx.
				\end{aligned}
				\deqq
				On the other hand, integration by parts yields directly
				\beq\label{3413}
				\begin{aligned}
					&\i  b_3 \p_3Z^{\alpha}\p_3 \wb\cdot \p_3 Z^\al \wu dx
					+\i  b_3 \p_3Z^{\alpha}\p_3 \wu \cdot \p_3 Z^\al \wb dx\\
					=&-\underset{L_1}{\underbrace{\i  \p_3 b_3 Z^{\alpha}\p_3 \wb\cdot Z^\al \p_3 \wu dx}}
					-\underset{L_2}{\underbrace{\i  \p_3 b_3 Z^{\alpha}\p_3 \wb\cdot(\p_3 Z^\al \wu-Z^\al \p_3 \wu)dx}}\\
					&-\underset{L_3}{\underbrace{\i  b_3 Z^{\alpha}\p_3 \wb \cdot \p_3(\p_3 Z^\al \wu-Z^\al \p_3 \wu)dx}}
					-\underset{L_4}{\underbrace{\i  \p_3 b_3 Z^{\alpha}\p_3 \wu \cdot(\p_3 Z^\al \wb-Z^\al \p_3 \wb) dx}}\\
					&-\underset{L_5}{\underbrace{\i  b_3  Z^{\alpha}\p_3 \wu \cdot \p_3(\p_3 Z^\al \wb-Z^\al \p_3 \wb) dx}}.
				\end{aligned}
				\deq
				Using the anisotropic type inequality \eqref{ie:Sobolev} repeatedly, we may conclude
				\beqq
				\begin{aligned}
					|L_1|
					\lesssim& \|\p_3 b_3\|_{L^2}^{\frac12}\|\p_{33} b_3\|_{L^2}^{\frac12}
					\|Z^{\alpha}\p_3 \wu\|_{L^2}^{\frac12}\|\p_1Z^{\alpha}\p_3 \wu\|_{L^2}^{\frac12}
					\|\p_3 Z^\al \wb\|_{L^2}^{\frac12}\|\p_{23} Z^\al \wb\|_{L^2}^{\frac12}\\
					\lesssim&\sqrt{\me^m(t)}\md^{m}(t),\\
					|L_2|
					\lesssim&\|\p_3 b_3\|_{L^2}^{\frac12}\|\p_{33} b_3\|_{L^2}^{\frac12}
					\|Z^{\alpha}\p_3 \wb\|_{L^2}^{\frac12}
					\|\p_1 Z^{\alpha}\p_3 \wb\|_{L^2}^{\frac12}\\
					&\times\|\p_3 Z^\al \wu-Z^\al \p_3 \wu\|_{L^2}^{\frac12}
					\|\p_2(\p_3 Z^\al \wu-Z^\al \p_3 \wu)\|_{L^2}^{\frac12}\\
					\lesssim&\sqrt{\me^m(t)}\md^{m}(t),\\
					|L_4|
					\lesssim&\|\p_3 b_3\|_{L^2}^{\frac12}\|\p_{33} b_3\|_{L^2}^{\frac12}
					\|Z^{\alpha}\p_3 \wu\|_{L^2}^{\frac12}
					\|\p_1 Z^{\alpha}\p_3 \wu\|_{L^2}^{\frac12}\\
					&\times\|\p_3 Z^\al \wb-Z^\al \p_3 \wb\|_{L^2}^{\frac12}
					\|\p_2(\p_3 Z^\al \wb-Z^\al \p_3 \wb)\|_{L^2}^{\frac12}\\
					\lesssim&\sqrt{\me^m(t)}\md^{m}(t),\\
					|L_3|
					\lesssim &\|(b_3, \p_3 b_3)\|_{L^2}^{\frac12}
					\|\p_3(b_3, \p_3 b_3)\|_{L^2}^{\frac12}
					\|Z^{\alpha}\p_3 \wb\|_{L^2}^{\frac12}
					\|\p_2 Z^{\alpha}\p_3 \wb\|_{L^2}^{\frac12}\\
					&\times\|Z_3(\p_3 Z^\al \wu-Z^\al \p_3 \wu)\|_{L^2}^{\frac12}
					\|\p_1 Z_3(\p_3 Z^\al \wu-Z^\al \p_3 \wu)\|_{L^2}^{\frac12}\\
					\lesssim &\sqrt{\me^m(t)}\md^{m}(t),\\
					|L_5|
					\lesssim &\|(b_3, \p_3 b_3)\|_{L^2}^{\frac12}
					\|\p_3(b_3, \p_3 b_3)\|_{L^2}^{\frac12}
					\|Z^{\alpha}\p_3 \wu\|_{L^2}^{\frac12}
					\|\p_1 Z^{\alpha}\p_3 \wu\|_{L^2}^{\frac12}\\
					&\times\|Z_3(\p_3 Z^\al \wb-Z^\al \p_3 \wb)\|_{L^2}^{\frac12}
					\|\p_2 Z_3(\p_3 Z^\al \wb-Z^\al \p_3 \wb)\|_{L^2}^{\frac12}\\
					\lesssim &\sqrt{\me^m(t)}\md^{m}(t).
				\end{aligned}
				\deqq
				Substituting the estimates for $L_1$ through $L_5$ into \eqref{3413}, we have
				\beqq
				\i  b_3 \p_3Z^{\alpha}\p_3 \wb\cdot \p_3 Z^\al \wu dx
				+\i  b_3 \p_3Z^{\alpha}\p_3 \wu \cdot \p_3 Z^\al \wb dx
				\lesssim  \sqrt{\me^m(t)}\md^{m}(t).
				\deqq
				Using anisotropic type inequality \eqref{ie:Sobolev}
				repeatedly to deal with the other term in $II_3$ and $II_6$, it holds
				\beq\label{3414}
				II_3+II_6 \lesssim \sqrt{\me^m(t)}\md^{m}(t).
				\deq
				\textbf{Estimate of term $II_4$ and $II_5$.}
				It is easy to check that
				\beqq
				\begin{aligned}
					II_5
					=&\sum_{0\le \beta \le \alpha}C^{\beta}_{\alpha}
					\i \{\p_3 (Z^\al u_h \cdot Z^{\alpha-\beta}\nabla_h \wb)
					+\p_3 Z^\al u_3 Z^{\alpha-\beta}\p_3 \wb\}
					\cdot \p_3 Z^\al \wb dx\\
					&+\sum_{0< \beta \le \alpha}C^{\beta}_{\alpha}
					\i  Z^\beta u_3 \p_3Z^{\alpha-\beta}\p_3 \wb
					\cdot \p_3 Z^\al \wb dx
					+\i  u_3 \p_3 Z^{\alpha}\p_3 \wb \cdot \p_3 Z^\al \wb dx.
				\end{aligned}
				\deqq
				Integrating by part and using the anisotropic type
				inequality \eqref{ie:Sobolev}, we obtain
				\beq\label{3415}
				\begin{aligned}
					&\i  u_3 \p_3 Z^{\alpha}\p_3 \wb \cdot \p_3 Z^\al \wb dx\\
					=&\i  u_3 \p_3 \p_3 Z^{\alpha} \wb \cdot \p_3 Z^\al \wb dx
					+\i  u_3 \p_3 (Z^{\alpha}\p_3 \wb-\p_3 Z^{\alpha}\wb)\cdot \p_3 Z^\al \wb dx\\
					\lesssim
					&\|\p_3 u_3\|_{L^2}^{\frac12}\|\p_3^2 u_3\|_{L^2}^{\frac12}
					\|\p_3 Z^{\alpha} \wb \|_{L^2}^{\frac12}\|\p_{23} Z^{\alpha} \wb \|_{L^2}^{\frac12}
					\|\p_3 Z^{\alpha} \wb \|_{L^2}^{\frac12}\|\p_{13} Z^{\alpha} \wb \|_{L^2}^{\frac12}\\
					&+\|(u_3, \p_3 u_3)\|_{L^2}^{\frac12}
					\|\p_3 (u_3, \p_3 u_3)\|_{L^2}^{\frac12}
					\|\p_3 Z^\al \wb\|_{L^2}^{\frac12}
					\|\p_{23} Z^\al \wb\|_{L^2}^{\frac12}\\
					&\times \|Z_3(Z^{\alpha}\p_3 \wb-\p_3 Z^{\alpha}\wb)\|_{L^2}^{\frac12}
					\|\p_1 Z_3(Z^{\alpha}\p_3 \wb-\p_3 Z^{\alpha}\wb)\|_{L^2}^{\frac12}\\
					\lesssim
					& \sqrt{\me^m(t)}\md^{m}(t).
				\end{aligned}
				\deq
				Using anisotropic type inequality \eqref{ie:Sobolev} repeatedly to deal with the other
				term in $II_5$ and using the estimate \eqref{3415}, we may write
				\beq\label{3416}
				II_5 \lesssim \sqrt{\me^m(t)}\md^{m}(t).
				\deq
				Similarly, it is easy to check that
				\beq\label{3417}
				II_4 \lesssim \sqrt{\me^m(t)}\md^{m}(t).
				\deq
				\textbf{Estimate of term $II_7$ and $II_8$.}
				It is easy to check that
				\beq\label{3418}
				\nabla(u\cdot \nabla)\times b
				=\nabla u_1 \times \p_1 b
				+\nabla u_2 \times \p_2 b
				+\nabla u_3 \times \p_3 b,
				\deq
				and
				\beq\label{3419}
				\begin{aligned}
					\nabla u_3 \times \p_3 b
					&=(\p_2 u_3\p_3 b_3-\p_3 u_3\p_3 b_2,
					\p_3 u_3\p_3 b_1-\p_1 u_3\p_3 b_3,
					\p_1 u_3\p_3 b_2-\p_2 u_3\p_3 b_1)^T\\
					&=(-\p_2 u_3 \nabla_h\cdot b_h+\nabla_h \cdot u_h\p_3 b_2,
					-\nabla_h\cdot u_h\p_3 b_1+\p_1 u_3 \nabla_h\cdot b_h,
					\p_1 u_3\p_3 b_2-\p_2 u_3\p_3 b_1)^T,
				\end{aligned}
				\deq
				where we have used the conditions $\nabla \cdot u=0$ and $\nabla \cdot b=0$.
				Then, due to this basic fact \eqref{3418}
				and \eqref{3419}, we can apply the anisotropic type inequality
				\eqref{ie:Sobolev} to obtain
				\beq\label{3420}
				\begin{aligned}
					II_7
					\lesssim&
					\sum_{0\le \beta \le \alpha}
					\left\{\|\p_3 Z^\beta \nabla u\|_{L^2}^{\frac12}
					\|\p_{13} Z^\beta \nabla u\|_{L^2}^{\frac12}
					\|Z^{\alpha-\beta}\nabla_h b\|_{L^2}^{\frac12}
					\|\p_3 Z^{\alpha-\beta}\nabla_h b\|_{L^2}^{\frac12}\right.\\
					&\left.+\|Z^\beta \nabla u\|_{L^2}^{\frac12}
					\|\p_1 \p_3 Z^\beta \nabla u\|_{L^2}^{\frac12}
					\|\p_3 Z^{\alpha-\beta}\nabla_h b\|_{L^2}^{\frac12}
					\|\p_{13} Z^{\alpha-\beta}\nabla_h b\|_{L^2}^{\frac12}
					\right.\\
					&\left.
					+\|\p_3Z^\beta \nabla b\|_{L^2}^{\frac12}
					\|\p_{13}Z^\beta \nabla b\|_{L^2}^{\frac12}
					\|Z^{\alpha-\beta}\nabla_h u\|_{L^2}^{\frac12}
					\|\p_3 Z^{\alpha-\beta}\nabla_h u\|_{L^2}^{\frac12}\right.\\
					&\left.+\|Z^\beta \nabla b\|_{L^2}^{\frac14}
					\|\p_3 Z^\beta \nabla b\|_{L^2}^{\frac14}
					\|\p_1 Z^\beta \nabla b\|_{L^2}^{\frac14}
					\|\p_{13} Z^\beta \nabla b\|_{L^2}^{\frac14}
					\|\p_3 Z^{\alpha-\beta}\nabla_h u\|_{L^2}
					\right\}\\
					&\times \|\p_3 Z^\alpha \wb\|_{L^2}^{\frac12}
					\|\p_{23} Z^\alpha \wb\|_{L^2}^{\frac12}\\
					\lesssim
					&\sqrt{\me^m(t)}\md^{m}(t).
				\end{aligned}
				\deq
				Similarly, it is easy to check that
				\beq\label{3421}
				II_8 \lesssim \sqrt{\me^m(t)}\md^{m}(t).
				\deq
				\textbf{Estimate of term $II_9$ and $II_{10}$.}
				If $\alpha_3=0$, we have
				$
				II_9=-\ep \i |\p_3^2 Z^\al \wu|^2 dx.
				$
				If $\alpha_3 \neq 0$, we have
				\beqq
				\begin{aligned}
					II_9
					=&-\ep \i Z^\al \p_3^2 \wu \cdot \p_3^2 Z^\al \wu dx\\
					=&-\ep \i |\p_3^2 Z^\al \wu|^2 dx
					-\ep \i (Z^\al \p_3^2 \wu-\p_3^2 Z^\al \wu) \cdot \p_3^2 Z^\al \wu dx.
				\end{aligned}
				\deqq
				If $\alpha_3 \neq 0$, then it is easy to check that
				\beqq
				\begin{aligned}
					Z^{\alpha_3}\p_3^2 \wu-\p_3^2 Z^{\alpha_3}\wu
					=&\sum_{0<\beta_3 \le \alpha_3}C^{\beta_3}_{\alpha_3}
					\p_3[Z_3^{\beta_3}(\frac{1}{\varphi})\varphi]
					\p_3 Z_3^{\alpha_3-\beta_3}\wu
					+\sum_{0<\beta_3 \le \alpha_3}C^{\beta_3}_{\alpha_3}
					[Z_3^{\beta_3}(\frac{1}{\varphi})\varphi]
					\p_3^2 Z_3^{\alpha_3-\beta_3}\wu\\
					&+\sum_{0<\beta_3 \le \alpha_3}C^{\beta_3}_{\alpha_3}
					[Z_3^{\beta_3}(\frac{1}{\varphi})\varphi]
					\p_3 Z_3^{\alpha_3-\beta_3} \p_3 \wu,
				\end{aligned}
				\deqq
				which yields directly
				\beqq
				\begin{aligned}
					&|\i (Z^\al \p_3^2 \wu-\p_3^2 Z^\al \wu) \cdot \p_3^2 Z^\al \wu dx|\\
					\lesssim
					&(\|\p_3 \wu\|_{H^{|\alpha|-1}_{co}}+\|\p_3^2 \wu\|_{H^{|\alpha|-1}_{co}})
					\|\p_3^2 Z^\al \wu\|_{L^2}\\
					\lesssim
					&\frac14\|\p_3^2 Z^\al \wu\|_{L^2}^2
					+\|\p_3 \wu\|_{H^{|\alpha|-1}_{co}}^2+\|\p_3^2 \wu\|_{H^{|\alpha|-1}_{co}}^2.
				\end{aligned}
				\deqq
				Then, the term $II_9$ can be bounded by
				\beqq
				II_9
				\lesssim
				-\ep \i |\p_3^2 Z^\al \wu|^2 dx
				+\ep (\|\p_3 \wu\|_{H^{|\alpha|-1}_{co}}^2+\|\p_3^2 \wu\|_{H^{|\alpha|-1}_{co}}^2).
				\deqq
				Similarly, it is easy to check that
				\beqq
				II_{10}
				\lesssim
				-\ep \i |\p_3^2 Z^\al \wb|^2 dx
				+\ep (\|\p_3 \wb\|_{H^{|\alpha|-1}_{co}}^2+\|\p_3^2 \wb\|_{H^{|\alpha|-1}_{co}}^2).
				\deqq
				Therefore, we complete the proof of the claimed estimate \eqref{claim-estimate}.
			\end{appendices}

			\phantomsection

		\end{sloppypar}
	\end{document}